\documentclass{siamart1116}
\usepackage{lipsum}
\usepackage{epstopdf}
\ifpdf
  \DeclareGraphicsExtensions{.eps,.pdf,.png,.jpg}
\else
  \DeclareGraphicsExtensions{.eps}
\fi

\usepackage{comment}
\usepackage{algpseudocode}
\usepackage{algorithm}

\usepackage{amsmath,amsfonts,amssymb,graphicx,enumitem,mathtools,amssymb}
\usepackage{url}
\algrenewcommand\algorithmicindent{1em}
\usepackage{subfig}
\usepackage{arydshln}
\usepackage{hyperref}
\makeatletter
\DeclareTextCommand{\textprime}{\encodingdefault}{%
  \mbox{$\m@th'\kern-\scriptspace$}% 
}
\makeatother
\renewcommand{\texttt}[1]{%
  \begingroup
  \ttfamily
  \begingroup\lccode`~=`/\lowercase{\endgroup\def~}{/\discretionary{}{}{}}%
  \begingroup\lccode`~=`[\lowercase{\endgroup\def~}{[\discretionary{}{}{}}%
  \begingroup\lccode`~=`.\lowercase{\endgroup\def~}{.\discretionary{}{}{}}%
  \begingroup\lccode`~=`-\lowercase{\endgroup\def~}{-\discretionary{}{}{}}%
  \catcode`/=\active\catcode`[=\active\catcode`.=\active\catcode`-=\active
  \scantokens{#1\noexpand}%
  \endgroup
}
\usepackage{tikz-cd}

\newsiamthm{fact}{Fact}
\newsiamthm{ex}{Example}
\newtheorem*{rmk}{Remark}

  \def\Cc{\mathcal{C}}  \def\Ec{\mathcal{E}}  \def\Gc{\mathcal{G}}    \def\Kc{\mathcal{K}}   \def\Nc{\mathcal{N}}  \def\Pc{\mathcal{P}}  \def\Rc{\mathcal{R}} \def\Sc{\mathcal{S}}

\def\I{\mathbb{I}}

\def\R{\mathbb{R}}

\def\al{\alpha}
\def\be{\beta}
\def\Ga{\Gamma}
\def\ga{\gamma}
\def\Dl{\Delta}
\def\dl{\delta}
\def\ep{\epsilon}

\def\lm{\lambda}
\def\sg{\sigma}
\def\Om{\Omega}

\def\zt{\zeta}

\def\gd{\nabla}
\def\pa{\partial}
\def\es{\enspace}
\def\ra{\rightarrow}
\def\la{\leftarrow}
\def\Ra{\Rightarrow}

\def\1{\mathbf{1}}
\def\0{\mathbf{0}}

\def\limti{\lim_{t\rightarrow\infty}}

\def\XB{\mathbf{X}}

\def\xB{\mathbf{x}}
\def\xb{\bar{x}}
\def\xt{\tilde{x}}

\def\yB{\mathbf{y}}
\def\yb{\bar{y}}

\def\zB{\mathbf{z}}

\def\taut{\tilde{\tau}}

\def\ft{\tilde{f}}
\def\xBt{\tilde{\mathbf{x}}}
\def\pit{\tilde{\pi}}
\def\pt{\tilde{p}}
\def\pb{\bar{p}}
\def\pB{\mathbf{p}}
\def\pBb{\bar{\mathbf{p}}}
\def\xBb{\bar{\mathbf{x}}}
\def\xBh{\hat{\mathbf{x}}}
\def\PB{\mathbf{P}}
\def\WB{\mathbf{W}}
\def\piB{\mathbf{\pi}}
\def\piBt{\tilde{\mathbf{\pi}}}
\def\rB{\mathbf{r}}
\def\rBb{\bar{\mathbf{r}}}
\def\yBb{\bar{\mathbf{y}}}
\def\PBh{\hat{\mathbf{P}}}
\def\WBh{\hat{\mathbf{W}}}
\def\Gab{\bar{\Ga}}
\def\uB{\mathbf{u}}
\def\IB{\mathbf{I}}
\def\eB{\mathbf{e}}

\def\nb{\bar{n}}
\def\Si{{\mathcal{S}_i}}
\def\Nm{{\mathcal{N}_m}}
\def\taumin{\tau^{\min}}
\def\Lmax{L^{\max}}
\def\etamax{\eta^{\max}}
\def\RB{\mathbf{R}}
\def\wB{\mathbf{w}}

\def\rU{\subset}

\def\uU{\cup}
\def\dU{\cap}
\def\rUe{\subseteq}

\def\sm{\setminus}
\def\buU{\bigcup}

\def\S{\mathcal{S}}

\def\bms{\begin{bmatrix}}
\def\bme{\end{bmatrix}}
\def\beq{\begin{equation}}
\def\eeq{\end{equation}}
\def\bal{\begin{equation}\begin{aligned}}
\def\eal{\end{aligned}\end{equation}}
\def\bals{\begin{equation*}\begin{aligned}}
\def\eals{\end{aligned}\end{equation*}}
\numberwithin{theorem}{section}

\newcommand{\TheTitle}{Localization and Approximations for Distributed Non-convex Optimization} 
\newcommand{\TheAuthors}{H. Kao and V. Subramanian}

\headers{\TheTitle}{\TheAuthors}

\title{{\TheTitle}
\thanks{Submitted to the editors on June 14, 2021.
}
}

\author{
  Hsu Kao\thanks{Electrical and Computer Engineering, University of Michigan, Ann Arbor, MI (\email{hsukao@umich.edu}).}
  \and
  Vijay Subramanian\thanks{Electrical and Computer Engineering, University of Michigan, Ann Arbor, MI (\email{vgsubram@umich.edu}).}
}

\usepackage{amsopn}

\ifpdf
\hypersetup{
  pdftitle={\TheTitle},
  pdfauthor={\TheAuthors}
}
\fi

\begin{document}

\maketitle

\sloppy

% REQUIRED
\begin{abstract}
Distributed optimization has many applications, in communication networks, sensor networks, signal processing, machine learning, and artificial intelligence. Methods for distributed convex optimization are widely investigated, while those for non-convex objectives are not well understood. One of the first non-convex distributed optimization frameworks over an arbitrary interaction graph was proposed by Di Lorenzo and Scutari [\textit{IEEE Trans. on Signal and Information Processing over Network}, 2 (2016), pp. 120-136], which iteratively applies a combination of local optimization with convex approximations and local averaging. We generalize the existing results in two ways. In the case when the decision variables are separable such that there is partial dependency in the objectives, we reduce the communication complexity of the algorithm so that nodes only keep and communicate local variables instead of the whole vector of variables. In addition, we relax the assumption that the objectives' gradients are bounded and Lipschitz by means of successive proximal approximations. Having developed the methodology, we then discuss many ways to apply our algorithmic framework to resource allocation problems in multi-cellular networks, where the two generalizations are found useful and practical. Simulation results show the superiority of our resource allocation algorithms over naive single cell methods, and furthermore, our approximation framework lead to algorithms that are numerically more stable.
\end{abstract}

% REQUIRED

% REQUIRED

%%%%%%%%%%%%%%%%%%%%%%%%%%%%%%%%%%%%%%%%%%%%%%%%%%
%%%%%              Introduction              %%%%%
%%%%%%%%%%%%%%%%%%%%%%%%%%%%%%%%%%%%%%%%%%%%%%%%%%

\section{Introduction}\label{sec:intro}
Distributed computation has received great attention in response to the overwhelming need in applications \cite{Rabbat_Sensor_2004,Duchi_DistDual_2012}. In all of these applications there are multiple agents or devices with their own local data that want to perform a joint computational task that is either impractical or infeasible to be centralized. Reasons for this choice include high communication costs, availability of large amount of information, unavailability of centralized processing, or simply harnessing the efficiency of parallelization \cite{Rabbat_Sensor_2004,Duchi_DistDual_2012}. Some specific examples are as follows: communication networks, e.g. scheduling and allocation in a multi-cell setting; sensor networks, e.g., remote parameter estimation with sensor data \cite{Rabbat_Sensor_2004}; and statistical machine learning, with large datasets. In all of these cases the underlying computational task can be formulated as an optimization problem, with objective being utilities, loss functions, etc. \cite{Duchi_DistDual_2012}. Such motivations have lead to a burgeoning of the research in distributed optimization.

Distributed optimization methods with convex objective functions have been well investigated in the literature. Many proposed methods fall into the category of gradient or subgradient methods \cite{Nedic_DistSubgrad_2009}, which take gradient descent steps at each node and then average the results. Another class of methods utilize a dual-decomposition idea, like the Alternating Direction Method of Multipliers (ADMM) method \cite{Boyd_ADMM_2011}. In contrast, distributed non-convex optimization has received much less attention. One of the first provably convergent algorithms for \emph{non-convex} objectives using a fully distributed scheme, called \emph{NEXT}, over a network with arbitrary graphical structure was introduced in \cite{Scutari_NEXT_2016}. The main idea in \cite{Scutari_NEXT_2016} is to perform local optimization by finding surrogate convex functions based on the current iterate and utilizing successive convex approximations of the non-convex objective, and then enforcing consensus among the network so that a global objective can be solved in a distributed manner. Other papers on this topic assume much stronger conditions, such as existence of a central controller to align the outputs in each step \cite{Hajinezhad_NESTT_2016}, or even a complete graph network (interaction) structure \cite{Scaglione_Gossip_2013}. Given this we will adopt the framework of \emph{NEXT} \cite{Scutari_NEXT_2016} in our work. Here we consider the scenario where the network that specifies the communication structure is given; the reader is referred to \cite{Mingyan_NetCompute_2017} for how to decide the network structure actively, with minimizing the energy consumption for delay-constrained singular value decomposition computation as a motivating example.

There are some fundamental issues with \emph{NEXT}~\cite{Scutari_NEXT_2016} though, which we will address in our work. First, the algorithm requires each node to store and update the entire vector of decision variables, irrespective of the underlying dimension or structure; this is also an issue more broadly for most distributed optimization algorithms. In certain applications, the decision variables might be the ensemble of sets of control parameters at each node, which could be of a significant dimension themselves: e.g., multiple platoons of automated cars with a local controller for each team, or cellular base-stations each with many connected devices. Directly using \emph{NEXT} would necessitate greatly increased storage at each node and also high-rate and low-latency communication between all nodes, which is impractical for a large network. In the illustrative examples above, the decision variables typically can be decomposed into blocks with a sparse interconnection between different blocks. Such a block structure could be used in reducing the storage and communication requirements. This, however, has not received much attention in the distributed optimization literature. The block coordinate descent method for centralized optimization is studied in \cite{Beck_BlockDescent_2013}, wherein gradient descent is effectively carried out one block at a time. When the objective is the sum of separable functions, convergence is shown for extensions of ADMM when the number of blocks is two \cite{Beck_BlockDescent_2013}, but this no longer holds when there are three blocks \cite{Yuan_BlockADMM_2016}. A distributed optimization scheme for variables with block structure is proposed in \cite{Scutari_Parallel_2015}, but only the convex part of the objective is decomposable. An optimizatiton problem where the separable variables of agents are coupled through a convex social cost function is studied in \cite{Mengdi_Sep_2017}, where the author uses a dual method to decouple the variables and leverage the separability; the goal is to show that the duality gap vanishes when the number of agents grows. In our work we will address this lacuna and present an algorithm that \emph{exploits the underlying block structure of the decision variables} through a process which we call \emph{localization} when the objective is the sum of separable non-convex functions. Our idea of localization is similar to \cite{Hu_LocalDist_2017}, which exploit the sparsity of the constraints in convex feasibility problems (CFPs) to reduce the memory and communication needed. As we will explain in \cref{sec:dis-3}, not only is our framework more general, but it also works for non-convex problems. 

A second issue with the approach in \cite{Scutari_NEXT_2016} is that the objective in many applications may not have the required smoothness. For example, the common assumption is of Lipschitz and bounded gradients as in \cite{Scutari_NEXT_2016}, but this may not hold; we will demonstrate this explicitly with the motivating application. In centralized convex optimization, apart from subgradient methods, proximal methods are used for non-smooth functions~\cite{Boyd_Proximal_2013}: e.g., a substitute is the Moreau envelope that is strongly convex and maintains the minimizer. We will develop a general scheme for continuous objective functions with non-Lipschitz unbounded gradients that takes \emph{any sequence of smooth approximations} as input, such as the Moreau envelope.

\paragraph{Motivating Example}
There is a trend for increased access-point deployment density coupled with the increasing usage of high-speed connections as well as fiber for backhaul~\cite{ZagerFibre2BTS_2012,AllevenBackhaul_2016} in modern wireless networks. Hence, it is feasible to envisage high-rate and low-latency communications based coordination between neighboring base-sites to implement distributed optimization methods for resource allocation. Even this only allows communication of locally relevant decision variables, and rules out communicating network-wide decision variables, as would be the case if one used the algorithm in \cite{Scutari_NEXT_2016}. We will use the resource allocation problem, solved via the Network Utility Maximization framework~\cite{Chiang_NUMDecomp_2006} with some specific use cases described in \cite{VJ_DLSchedule_2009,VJ_ResourceAlloc_2009}, as our motivating example. In the one-shot weighted sum-rate maximization problem derived from the decomposition of network utility across time instances in a time-varying channels environment, we have to jointly decide the power transmitted by base stations (BSs) at each channel (power control), as well as the resource blocks (RBs) a BS should transmit data to its users (scheduling); jointly this is termed resource allocation. This problem is hard because the objective is non-convex, and the constraints are knapsack-like constraints; the latter are usually solved by relaxing to real-valued variables and then rounding. We will follow the same approach, and concentrate on obtaining (locally) optimal solutions to the relaxed problem. While this problem for the single-cell is well characterized \cite{VJ_DLSchedule_2009}, the solution to even the multi-cell power control problem with interference impacts remains unresolved \cite{Chiang_PowCtrl_2008}.

Given the difficulty of solving the multi-cell resource allocation problem, existing methods in the literature use heuristic approaches such as decomposition of the problem followed by greedy algorithms that lead to sub-optimal solutions \cite{Venturino_MultiPowAlloc_2009}, or make strong assumptions on the network interference graph \cite{Shen_MultiPowAlloc_2014}. Some algorithms are centralized \cite{Yassin_Centralized_2016}, which is an impractical assumption in a realistic scenario. Prior work  \cite{Schmidt_Distributed_2009} also utilizes \emph{interference prices} to solve the multi-cell power control problem, where each BS-user pair maximizes its own utility minus the sum of marginal ``costs" to all other users with increase in its power. This literature only considers power control and not scheduling. Also, though having extensions in multi-channel settings, only one receiver is considered under each transmitter, and users need to sequentially broadcast their interference prices to ensure convergence in the multiple-input single-output case. In \cite{Lei_PowCtrl_2008}, a distributed power control and scheduling algorithm is proposed; $\ep$-optimality is established for the grid network with a $K$-hop interference model.

In this paper, we generalize the results in \cite{Scutari_NEXT_2016} by resolving the two issues mentioned earlier. For the first issue, we exploit the separability of the decision variables and the objectives' partial dependencies to reduce the storage and communication needs. Although all components of the decision variable are entangled via each node's objective, we show that each component can be maintained and optimized within a local network -- this is different from directly applying \emph{NEXT} to the setting, and its convergence is not obvious from \cite{Scutari_NEXT_2016}. Secondly, inspired by the proximal method, we use a series of functions with Lipschitz gradients to approximate the original objective, and significantly relax the smoothness assumptions made in \cite{Scutari_NEXT_2016}. We show that as long as the series of approximation functions approach the original function slowly enough, we are still guaranteed to obtain stationary solutions in many situations. While following the same steps as \emph{NEXT} when the gradients are Lipschitz, our algorithm appears to have superior numerical stability over \emph{NEXT} with non-Lipschitz gradients. We establish convergence for our algorithm with the proposed two generalizations under the condition of no unbounded gradients on the boundary. We then apply the results and algorithms to the multi-cell resource allocation problem in many different ways, which
gives numerous algorithms with provable convergence to locally optimal solutions. Last but not least, we give a stochastic approximation interpretation of \emph{NEXT} in \cref{sec:sa}, where we provide an alternative proof of \emph{NEXT} and discuss its relation to \cite{Bianchi_DistGrad_2013}.

The rest of the paper is organized as follows. We describe the distributed optimization problem setup and the idea of \emph{NEXT} in \cref{sec:pre}. Then we present our generalized algorithm with localization and proximal approximations in \cref{sec:main}. In \cref{sec:dis}, we discuss the effect of localization, as well as address the practicability issues of our algorithm. We give examples where gradients are unbounded such that \emph{NEXT} fails to converge to the correct solution while our algorithm does. We describe the application to the multi-cell resource allocation problem in \cref{sec:app}. We discuss simulation results for the approximation functions and resource allocation application in \cref{sec:sim}, and conclude in \cref{sec:conc}.

%%%%%%%%%%%%%%%%%%%%%%%%%%%%%%%%%%%%%%%%%%%%%%%%%%
%%%%%             Preliminaries              %%%%%
%%%%%%%%%%%%%%%%%%%%%%%%%%%%%%%%%%%%%%%%%%%%%%%%%%

\section{Preliminaries}\label{sec:pre}
In this section, we give the system model of distributed non-convex optimization and assumptions. The bulk of the setup directly follows from \cite{Scutari_NEXT_2016}. Consider a network $\Nc=\{1,\dots,I\}$ that consists of $I$ nodes. We aim to solve an optimization problem of the form
\beq\label{4}
\min_{\xB\in\Kc}\quad U(\xB)=F(\xB)+G(\xB)=\sum_{i=1}^If_i(\xB)+G(\xB),
\eeq
where all $f_i$'s are $C^1$ smooth but can be non-convex, and $G$ is convex but may be non-smooth. The goal is to let these nodes cooperatively solve the problem in a distributed fashion. Therefore, each $j\in\Nc$ maintains a copy of the entire decision variable $\xB$, referred to as $\xB_j$. Then \cref{4} is equivalent to solving the optimization problem
\beq\label{5}
\min_{\xB_j\in\Kc}\quad\sum_{i=1}^If_i(\xB_j)+G(\xB_j)
\eeq
at each $j\in\Nc$ subject to the constraint that all nodes agree on their optimal choices, i.e., we enforce
\beq\label{6}
\xB_1=\xB_2=\dots=\xB_I.
\eeq
In the context of distributed optimization, node $i$ only has the information of $f_i$. It would require communication between the nodes to solve the problem in \cref{5}-\cref{6}.

Below are the standard assumptions on the objective functions and the constraint set.\\
{\bf Assumption A}\\
{\bf (A1)} The set $\Kc\in\R^d$ is closed and convex;\\
%{\bf (A2)} $f_i\in C^1$ for all $i$;\\
%{\bf (A2)} $\gd f_i$ is Lipschitz continuous for all $i$;\\
%{\bf (A3)} $\gd F$ is bounded by $L_F$ for all $\xB\in\Kc$;\\
{\bf (A2)} $G$ is convex with bounded subgradient $L_G$ for all $\xB\in\Kc$;\\
{\bf (A3)} $U$ is coercive, that is, $\lim_{\xB\in\Kc,|\xB|\ra\infty}U(\xB)=\infty$; based on this we can effectively assume that $\mathcal{K}$ is compact;\\
{\bf (A4)} $f_i$'s have bounded gradients, i.e. $\exists\es B$ s.t. $\|\gd f_i(\xB)\|<B$ for all $\xB\in\Kc$.

The set of nodes $\Nc$ along with a set of undirected edges $\Ec$ form a graph $\Gc=(\Nc,\Ec)$. This graph captures how communications take place -- node $i$ and $j$ can only communicate if $(i,j)\in\Ec$. $\Gc$ is assumed to be connected to foster communication between the nodes; otherwise, the problem is generally unsolvable\footnote{All results can be trivially extended to directed time-varying graphs: see \cite{Scutari_NEXT_2016} for \emph{NEXT} and the Appendix for our algorithm. For easy of presentation we adopt the above settings.}.

Our methods follow the solution scheme of \emph{NEXT}. In the \emph{NEXT} algorithm, each node performs a local convex optimization, and then some information will be exchanged in the network. At a high level, the first step is the ``descent step" and the second is the ``consensus step;" the two steps are iteratively applied to obtain the solution \cite{Scutari_NEXT_2016}. In the first step of time $n$, the node $i$ solves a convex approximation of the whole objective function by convexizing its own objective function $f_i$ \emph{parametrized by} the current iterate $\xB_i[n]$ to be a strongly convex surrogate function $\ft_i(\bullet;\xB_i[n])$, while linearizing the sum of other nodes' objective functions $\sum_{j\neq i}f_j$. We assume the surrogate function satisfies the following assumption:
\\
{\bf Assumption F}\\
%{\bf (F1)} $\ft_i(\bullet;\xB)$ is uniformly strongly convex with constant $\tau_i>0$;\\
%{\bf (F2)} $\gd\ft_i(\xB;\xB)=\gd f_i(\xB)$ for all $\xB\in\Kc$;\\
%{\bf (F3)} $\gd\ft_i(\xB;\bullet)$ is uniformly Lipschitz continuous.
{\bf (F1)} $\ft_i(\bullet;\xB)$ is convex;\\
{\bf (F2)} $\gd\ft_i(\xB;\xB)=\gd f_i(\xB)$ for all $\xB\in\Kc$;\\
{\bf (F3)} Either $\ft_i(\bullet;\xB)$'s are coercive for all $\xB\in\Kc$ and $i\in\Nc$ or $G(\bullet)$ is coercive.

The result established is convergence to the stationary solutions, whose definition is given as follows.
\begin{definition}
A point $\xB^*$ is a stationary solution of Problem \cref{4} if a subgradient $g\in\pa G(\xB^*)$ exists such that $(\gd F(\xB^*)+g)^T(\yB-\xB^*)\geq0$ for all $\yB\in\Kc$.
\end{definition}

%%%%%%%%%%%%%%%%%%%%%%%%%%%%%%%%%%%%%%%%%%%%%%%%%%
%%%%%              Main Result               %%%%%
%%%%%%%%%%%%%%%%%%%%%%%%%%%%%%%%%%%%%%%%%%%%%%%%%%

\section{Main Result}\label{sec:main}
In this section, we introduce our two generalizations -- localization and approximations of \emph{NEXT}, which are the key contribution of this paper. We will first describe our settings for the generalizations, and then give the revised algorithm and the convergence theorem.

\subsection{Localization Setting and Assumptions}\label{sec:main-1}
Consider the setup where there are $M$ local dependency sets $\Nc_1,\dots,\Nc_M$, and the local objective function of node $i$, i.e. $f_i$, only depends on a common variable $\xB^c$ and the local variables $\xB^m$ whenever $i\in\Nc_m$. To be more specific, the decision variables can be split into $M+1$ parts $\xB\triangleq(\xB^1,\dots,\xB^M,\xB^c)$ where $M$ is an arbitrary positive integer. For all $m\in[M]$ (which stands for $1,\dots,M$), define the local dependency set $\Nc_m\triangleq\{i:f_i\text{ is a function of }\xB^m\}\rUe\Nc$. Denote the sizes of $\Nc_1,\dots,\Nc_M$ as $I_1,\dots,I_M$. We can think $\Nc$ itself as the $(M+1)$-th local dependency set $\Nc_{M+1}$ and every $f_i$ may depend on $\xB^c\triangleq\xB^{M+1}$. We adopt the convention that whenever $M+1$ appears in either superscript or subscript, it means $c$ or anything associated with the original network $\Gc$; this includes $I_{M+1}\triangleq|\Nc|=I$. In the other direction, define the dependent part $\Sc_i\triangleq\{m:f_i\text{ is a function of }\xB^m,m\in[M+1]\}$. Note that for all $i$, $\Sc_i$ contains $M+1$. Also, when $\Sc_i$ appears in the superscript of a variable, for example $\xB$, it means the vector concatenated from all $\xB^k$'s such that $k\in\Sc_i$, i.e. $\xB^{\Sc_i}\triangleq(\xB^k)_{k\in\Sc_i}$. Concrete examples of local dependency sets are provided in \cref{sec:dis-1}. Furthermore, we have the following assumptions:\\
\noindent{\bf Assumption L}\\
{\bf (L1)} Besides the fact that $f_i$ depends on $\xB^m$ only if $i\in\Nc_m$ for $m\in[M+1]$, $G$ also only depends on $\xB^c$;\\
{\bf (L2)} The set $\Kc$ is \emph{separable}, i.e. it is the direct product of $(M+1)$ convex sets in proper subspaces
$\Kc=\Kc_1\times\dots\times\Kc_{M+1}\rU\R^d$ such that $\xB^m\in\Kc_m\rU\R^{d_m}$ for all $m\in[M+1]$ if and only if $\xB\triangleq(\xB^1,\dots,\xB^{M+1})\in\Kc$;\\
{\bf (L3)} The local network $\Gc_m=(\Nc_m,\Ec_m=\{(i,j)\in\Ec:i\in\Nc_m\text{ and }j\in\Nc_m\}$ is connected for all $m\in[M]$\footnote{We add nodes to get connectedness if it does not hold. More on this in \cref{sec:dis-2}}, we already assume the connectedness of $\Gc_{M+1}$ in \cref{sec:pre};\\
{\bf (L4)} For all $m\in[M+1]$ there is a matrix $\WB^m$ associated with $\Nc_m$ -- each entry is non-zero if and only if there is a corresponding edge in $\Ec_m$, and all non-zero entries must be greater than or equal to some fixed $\vartheta>0$. Equivalently, for all rows $i\not\in\Nc_m$ and all columns $j\not\in\Nc_m$, we have $(\WB^m)_{i,:}=0$ and $(\WB^m)_{:,j}=0$. In addition, $\WB^m$ is doubly-stochastic after deleting these zero rows and columns. $\WB^{M+1}$ does not contain zero row or column, and corresponds to the $\WB$ defined in \emph{NEXT}.\\
Note that $\Nc_m$'s do not have to be disjoint and form a partition of $\Nc$. Having $\Nc_m\dU\Nc_l\neq\emptyset$ for some $m\neq l$ is allowed. Without loss of generality we can assume they form a covering of $\Nc$; i.e., $\buU_{m=1}^{M+1}\Nc_m=\Nc$. When there exists a part $\xB^c$ on which all $f_i$'s and $G$ depend, then $\Nc_{M+1}$ is $\Nc$ itself; otherwise, nodes outside the union do not have any cross-dependence with all the rest and can be optimized themselves.

\subsection{Proximal Approximations Setting and Assumptions}\label{sec:main-2}
In contrast to the common setting in the optimization literature, we consider a scenario where $\gd f_i$ is not Lipschitz continuous for some $i$. Our idea to relax this Lipschitz assumption is to use a sequence of functions whose gradients \emph{are} Lipschitz continuous to approximate $f_i$. This is commonly known as the proximal approximation method in the literature of convex optimization, except that our objective is now non-convex.

To be more specific, we want to find a series of functions $\{f^*_{i,n}\}_{n\geq 1}$ such that $\gd f^*_{i,n}$ is globally Lipschitz continuous with constant $L_{i,n}$ and that as $n\ra\infty$ we have $f^*_{i,n}\ra f_i$ pointwise, or even better -- uniformly. Then at iteration $n$ we can use the well-behaved $f^*_{i,n}$ instead of $f_i$. We will see that as long as the schedule of $\{L_{i,n}\}_{n\geq 1}$ satisfies certain conditions, we can still have convergence to optimality.

The following assumption is the key feature of $f^*_{i,n}$ that our algorithm needs for convergence to optimality.\\
{\bf Assumption N}:\\
{\bf (N1)} $\gd f^*_{i,n}$ is globally Lipschitz continuous with constant $L_{i,n}$;\\
{\bf (N2)} $\lim_{n\ra\infty}f^*_{i,n}\ra f_i$ uniformly, and $\lim_{n\ra\infty}\gd f^*_{i,n}\ra\gd f_i$ pointwise.

We will also need a surrogate function of $f^*_{i,n}$, denoted as $\ft^*_{i,n}$. These surrogate functions should satisfy Assumption {F\textprime} similar to Assumption F given as follows.\\
{\bf Assumption {F\textprime}}\\
{\bf (F1\textprime)} $\ft^*_{i,n}(\bullet;\xB)$ is uniformly strongly convex with constant $\tau_{i,n}>0$;\\
{\bf (F2\textprime)} $\gd\ft^*_{i,n}(\xB;\xB)=\gd f^*_{i,n}(\xB)$ for all $\xB\in\Kc$;\\
{\bf (F4\textprime)} $\gd\ft^*_{i,n}(\xB;\bullet)$ is uniformly Lipschitz continuous with constant $L_{i,n}$.

As \cite{Scutari_NEXT_2016} does not have approximation functions, they assume Assumption {F\textprime} for the surrogate function of $f_i$, i.e. $\ft_i$, with fixed $\tau_i>0$ and $L_i<\infty$. Here our assumptions are more general than \emph{NEXT} in that we do not require the strong convexity of $\ft_i(\bullet;\xB)$ and the Lipschitz continuity of $\gd\ft_i(\xB;\bullet)$. In particular, we can have $\lim_{n\ra\infty}\tau_{i,n}=0$ and $\lim_{n\ra\infty}L_{i,n}=\infty$ as long as the schedules of $\{\tau_{i,n}\}_n$ and $\{L_{i,n}\}_n$ satisfies certain conditions. We do have an additional Assumption F3. However, that assumption is implicitly implied if $\ft_i(\bullet;\xB)$'s are strongly convex; thus, we do not lose any generality from \cite{Scutari_NEXT_2016}. Also for simplicity, we assume $L_{i,n}$ is a non-decreasing sequence. We denote $F^*_n=\sum_{i=1}^If^*_{i,n}$.

\subsection{Localized Proximal Inexact NEXT and Main Convergence Theorem}\label{sec:main-3}
Our localized proximal inexact version of \emph{NEXT}, which requires less storage and communication, and allows unbounded non-Lipschitz objective gradients, is presented in \cref{a6}. All operations that contain index $i$, i.e. Line 5, 6, 7, 9, 10, and 11, are performed for all $i\in\Nc$. Also, except Line 5, the operations have a superscript $k$, and are performed for all $k\in\Sc_i$. In Line 5,
\beq\label{8-2}
\xBt^*_i(\xB_i[n],\piBt_i[n])=\underset{\mathclap{\xB_i\in\prod_{k\in\Sc_i}\Kc_k}}{\arg\min\quad}\left[\ft^*_{i,n}(\xB_i;\xB_i[n])+\sum_{k\in\Sc_i}\piBt^k_i[n]^T(\xB^k_i-\xB^k_i[n])+G(\xB^c_i)\right],
\eeq
with $\xB_i=(\xB^k_i)_{k\in\Sc_i}=\xB^{\Sc_i}_i$ and $
\piBt_i=\{\piBt^k_i:k\in\Sc_i\}=\piBt^{\Sc_i}_i$. Note that Line 6 along with Line 5 means that one solves the optimization problem in \cref{8-2} with accuracy $\ep_i[n]$. Also denote $\gd f^*_{i,n}[n]=\gd f^*_{i,n}(\xB_i[n])$. The output of the algorithm is the concatenation of $[\xBb^m]_{m\in[M+1]}$, where $\xBb^m\triangleq\frac{1}{I_m}\sum_{i\in\Nc_m}\xB^m_i$.

\begin{algorithm}
\caption{Localized Proximal Inexact NEXT}
\label{a6}
%\begin{algo}[Localized inexact NEXT]\label{a6}
%\hfill
%\begin{breakablealgorithm}
%{
\begin{algorithmic}[1]
\State{Initialization: $\xB^k_i[0]\in\Kc_k$, $\yB^k_i[0]=\gd_{\xB^k}f^*_{i,0}[0]$, $\piBt^k_i[0]=(I_k-1)\yB^k_i[0]$}
\While{$\xB[n]$ does not satisfy the termination criterion}
\State{$n\la n+1$}
\State{\textit{Local SCA optimization: for all $i\in\Nc$ and for all $k\in\Sc_i$}}
\State{$\xBt_i[n]=\xBt^*_i(\xB_i[n],\piBt_i[n])$}
\State{Find a $\xB^{k,inx}_i[n]$ s.t. $\|\xBt^{k}_i[n]-\xB^{k,inx}_i[n]\|\leq\ep^{k}_i[n]$}
\State{$\zB^{k}_i[n]=\xB^{k}_i[n]+\al[n](\xB^{k,inx}_i[n]-\xB^{k}_i[n])$}
\State{\textit{Consensus update: for all $i\in\Nc$ and for all $k\in\Sc_i$}}
\State{$\xB^{k}_i[n+1]=\sum_{j\in\Nc_k}w^k_{ij}\zB^k_j[n]$}
\State{$\yB^k_i[n+1]=\sum_{j\in\Nc_k}w^k_{ij}\yB^k_j[n]+(\gd_{\xB^k}f^*_{i,n+1}[n+1]-\gd_{\xB^k}f^*_{i,n}[n])$}
\State{$\piBt^k_i[n+1]=I_k\cdot\yB^k_i[n+1]-\gd_{\xB^k}f^*_{i,n+1}[n+1]$}
\EndWhile
\Ensure{$\xB[n]\triangleq(\xBb^1[n],\dots,\xBb^M[n],\xBb^c[n])$}
\end{algorithmic}
%}
%\end{breakablealgorithm}
%\end{algo}
\end{algorithm}

\emph{NEXT} is a special case of \cref{a6}. First, \emph{NEXT} would either be the case where $I_m=I$ for all $m\in[M]$, or the equivalent case where $\xB$ consists of one part $\xB^c$ only. In the \emph{NEXT} algorithm, node $i$ keeps the whole $\xB$, and communicates the whole $\xB$ as well with all of its neighbors; the same also applies to the variables $\yB$ and $\piBt$. On the contrary, under our localization setting, node $i$ turns out only has to maintain $\xB^{\S_i}$ (also $\yB^{\S_i}$ and $\piBt^{\S_i}$) in \cref{a6}; moreover, it only communicates $\xB^{\Sc_i\dU\Sc_j}$ (also $\yB^{\S_i\dU\Sc_j}$) with its neighbor $j$. This could mean substantial savings for memory and communication, especially when the dependency structure is sparse. \cref{ex5-2} provides an example that explicitly specifies the storage and communication needed before and after localization. Secondly, \emph{NEXT} would be when $f^*_{i,n}=f_i$, $L_{i,n}=L_i<\infty$, and $\tau_{i,n}=\tau_i>0$ for all $i$ and $n$. \cref{a6} also accommodates a bigger class of objective functions and supports a more flexible choice of the schedules of $\{L_{i,n}\}$ and $\{\tau_{i,n}\}$.

\begin{rmk}
The insight of the variables $\yB$ and $\piBt$ can be found in \cite{Scutari_NEXT_2016}. In short, node $i$ needs the information of $\sum_{j\neq i}\gd f_j$ at the current iterate $\xB_i[n]$ when linearizing others' objectives. For this purpose, node $i$ tracks the average of gradients $\frac{1}{I}\sum_j\gd f_j$ using $\yB$, keeps its neighbors updated, and obtains $\piBt_i$, an approximation of $\sum_{j\neq i}\gd f_j(\xB_i[n])$, by subtracting its own gradient from $\yB$. In our localized algorithm, only nodes in $\Nc_m$ participate in the decision of $\xB^m$ and the tracking of $\yB^m$ and $\piBt^m$, which requires some subtle treatment.
\end{rmk}

\begin{rmk}
Note that \cref{a6} is not the result of the direct application of \emph{NEXT} in the localization setting given in \cref{sec:main-1}. Suppose $m'$ is such that $i\not\in\Nc_{m'}$ and under the connectedness assumption. Different from \emph{NEXT}, node $i$ no longer keeps the gradient trace $\piBt^{m'}_i$ in \cref{a6}, nor does the variable $\xB^{m'}_i$ appear in the objective of the local optimization. The fact that convergence can still be obtained with \cref{a6} is not obvious from \cite{Scutari_NEXT_2016}.
\end{rmk}

The following theorem generalizes the convergence to stationary solutions result of \emph{NEXT} when gradients are bounded, specifying restrictions on the schedules of $\{L_{i,n}\}$, $\{\tau_{i,n}\}$. When gradients are unbounded, stricter constraints are needed.

\begin{theorem}\label{t7}
For all $m\in[M+1]$, let $\{\xB^m[n]\}_n\triangleq\{(\xB^m_i[n])_{i\in\Nc_m}\}_n$ be the sequences generated by \cref{a6}, $\{\xBb^m[n]\}_n\triangleq\left\{\frac{1}{I_m}\sum_{i\in\Nc_m}\xB^m_i[n]\right\}_n$ be their averages, and $\{\xBb[n]\}_n=\{(\xBb^1[n],\dots,\xBb^M[n],\xBb^c[n])\}_n$ be the ensemble of averages. Let $\Lmax_n=\max_iL_{i,n}$, $\taumin_n=\min_i\tau_{i,n}$, $\ep[n]=\min_{i,k}\ep^k_i[n]$, $\zeta_n=\max_{\xB\in\Kc}\|F^*_n(\xB)-F^*_{n-1}(\xB)|$, $\eta_{i,n}=\max_{\xB\in\Kc}\|\gd f^*_{i,n}(\xB)-\gd f^*_{i,n-1}(\xB)\|$, and $\etamax_n=\max_i\eta_{i,n}$.
%$\lim_{n\ra\infty}|L_{i,n+1}-L_{i,n}|=0\es\forall\es i$
%$\lim_{n\ra\infty}\al[n]\left(\frac{\Lmax_n}{\taumin_n}\right)^2=0$
\begin{enumerate}[label=(\alph*),leftmargin=*]
\item
Suppose Assumptions A, F, L, N, and {F\textprime} hold\footnote{Note that Assumption N1 requires $\lim_{n\ra\infty}L^{\min}_n=\infty$ where $L^{\min}_n=\min_iL_{i,n}$ with $f^*_{i,n}$ given as \cref{10}. For any other choices of $f^*_{i,n}$ other than the double Moreau envelope function \cite{LasryLions_DoubleEnv_1986}, Assumption N1 requires $\lim_{n\ra\infty}L_{i,n}=\infty$ for any $i$ such that $\gd f_i$ is non-Lipschitz continuous.}. Also, $\al[n]\in(0,1]$ is such that $\sum_{n=0}^\infty(\Lmax_n)^3\left(\frac{\al[n]}{\taumin_n}\right)^2<\infty$, $\sum_{n=0}^\infty\taumin_n\al[n]=\infty$, $\sum_{n=0}^\infty\al[n]\Lmax_n\ep[n]<\infty$. Then (1) all sequences $\{\xB^m_i[n]\}_n\es\forall\es m\in[M+1]$ asymptotically agree, i.e., $\lim_{n\ra\infty}\|\xB^m_i[n]-\xBb^m[n]\|=0\quad\forall\es i\in\Nc_m$; (2) $\{\xBb[n]\}_n$ is bounded, and its limit points are stationary solutions of the original problem.
\item
If we do not assume Assumption A4, but we have $\lim_{n\ra\infty}\al[n]\frac{(\Lmax_n)^5}{(\taumin_n)^3}=0$, $\lim_{n\ra\infty}\frac{\etamax_n}{\taumin_n}=0$, $\sum_{n=0}^\infty\frac{\al[n]\Lmax_n\etamax_n}{\taumin_n}<\infty$, $\sum_{n=0}^\infty\zeta_n<\infty$, and $\lim_{n\ra\infty}\gd F^*_n\ra\gd F$. Then we still have results (1) and first part of (2). When the limit points lie in the interior of $\Kc$, or $\gd F$ is bounded on those limit points, they will be stationary solutions.
\item
Continuing (b), if a limit point $\xBb^\infty$ lies on the boundary of $\Kc$ and $\|\gd F(\xBb^\infty)\|=\infty$, then the definition of stationary solution does not apply, and $\xBb^\infty$ could be a local minimum, or a point that is not a local minimum.
\end{enumerate}
\end{theorem}
\begin{proof}
See \cref{app:b}.
\end{proof}
\begin{rmk}
Note that the conditions in (a) imply $\lim_{n\ra\infty}\al[n]\frac{(\Lmax_n)^3}{(\taumin_n)^3}=0$. In (b) we apply the stricter $\lim_{n\ra\infty}\al[n]\frac{(\Lmax_n)^5}{(\taumin_n)^3}=0$ as well as other constraints.
\end{rmk}
\noindent We can see that although all parts of the variable $\xB$ are coupled through all the objectives $f_i$'s, the decision on $\xB^m$ is actually dictated by the nodes in $\Nc_m$. We give a class of approximation functions that satisfies Assumption N, namely the Lasry-Lions envelope or double envelope \cite{LasryLions_DoubleEnv_1986}, in \cref{sec:dis-4}. Note that the convergence of \cref{a6} only requires the Lipschitz constants of $\{\gd f^*_{i,n}\}_{n\geq 1}$ follow certain schedules, but $\{f^*_{i,n}\}_{n\geq1}$ can be any sequence of functions that approaches $f_i$ in the limit, and does not need to be double envelope. We will give many examples of such sequences in \cref{sec:dis-6} and \cref{sec:app-4}. Also, an example of the schedules of the the parameters $\{\Lmax_n\}_n$, $\{\taumin_n\}_n$, $\{\al[n]\}_n$, and $\{\ep[n]\}_n$ that satisfies the related conditions in \cref{t7} (a) and (b) using $p$-series is given in \cref{sec:dis-5}. We will discuss the unbounded gradient issue in full detail in \cref{sec:dis-6}, study different examples there, and show the numerical stability of our algorithm for these examples in \cref{sec:sim-1}.

%%%%%%%%%%%%%%%%%%%%%%%%%%%%%%%%%%%%%%%%%%%%%%%%%%
%%%%%              Discussions               %%%%%
%%%%%%%%%%%%%%%%%%%%%%%%%%%%%%%%%%%%%%%%%%%%%%%%%%

\section{Discussions}\label{sec:dis}
In this section, we discuss some details of our proposed algorithm. Two localization examples that compare the effect of localization are given in \cref{sec:dis-1}. We comment on how to exploit localization when Assumption L3 is violated in \cref{sec:dis-2}. Comparison to the localization algorithm proposed by Hu \textit{et al.} in \cite{Hu_LocalDist_2017} is given in \cref{sec:dis-3}. \cref{sec:dis-4} and \cref{sec:dis-5} summarize the property of double envelope that we will use and provide the $p$-series example for the parameters $(L,\tau,\al,\ep)$, respectively.

\subsection{Examples of Localization}\label{sec:dis-1}
In this subsection we give a few examples to illustrate the concept and effectiveness of localization.

\begin{figure}
\centering
\subfloat[The communication graph.]{
\label{fig_local1a}\includegraphics[scale=0.4]{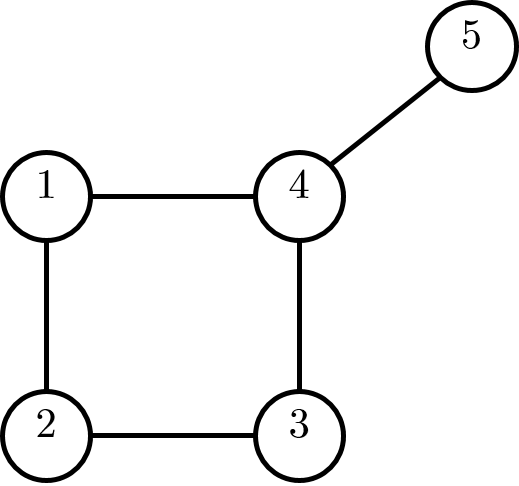}}
\subfloat[Local dependency sets for \cref{ex5-1}.]{
\label{fig_local1b}\includegraphics[scale=0.4]{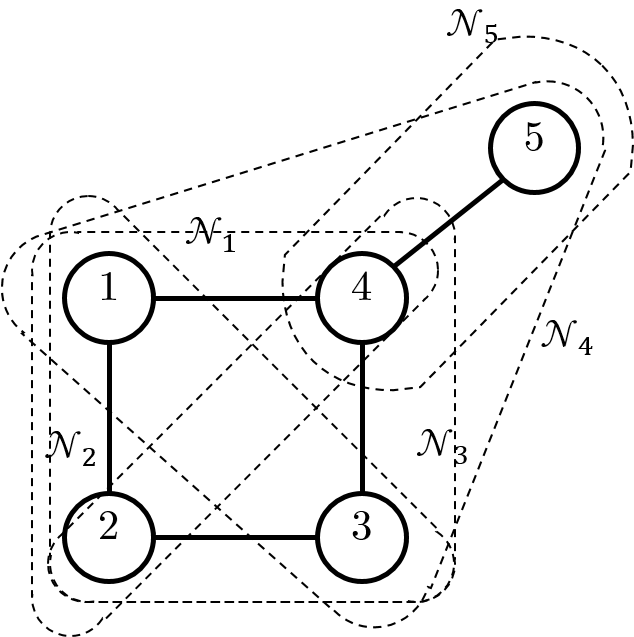}}
\subfloat[Local dependency sets for \cref{ex5-2}]{
\label{fig_local1c}\includegraphics[scale=0.4]{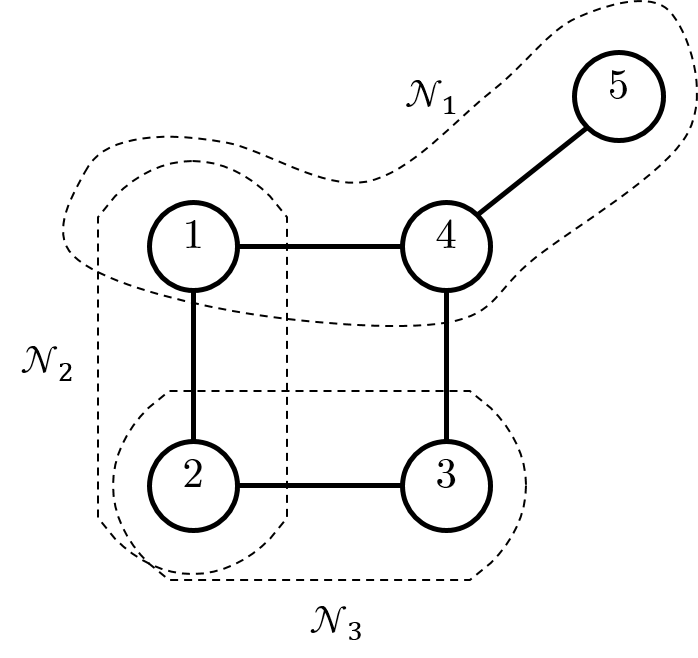}}
\caption{Graphical illustrations of the local dependency set examples.}
\label{fig_local1}
\end{figure}

\begin{ex}\label{ex5-1}
Consider the communication graph given in \cref{fig_local1} (a). In this example we consider the most common situation, where the dependency is directly given by the communication graph. Specifically, every node $i$ has a corresponding variable part $\xB^i$ in the whole variable tuple $\xB$. Moreover, the objective at node $i$ only depends on its own variable $\xB^i$ and its neighbors' variables $\{\xB^j:j\in N(i)\}$. In other words, the local dependency set $\Nc_i$ corresponding to $\xB^i$ is equal to $Nb(i):=N(i)\uU\{i\}$. $G$ is assumed to be $0$. This situation also arises in the resource allocation application we discuss in \cref{sec:app}.

For this communication graph, this situation will be the following. The variable $\xB=(\xB^1,\xB^2,\xB^3,\xB^4,\xB^5)$ can be split into five parts, and we have $f_1=f_1(\xB^1,\xB^2,\xB^4)$ ($f_1$ only depends on $\xB^1$, $\xB^2$, and $\xB^4$), $f_2=f_2(\xB^1,\xB^2,\xB^3)$, $f_3=f_3(\xB^2,\xB^3,\xB^4)$, $f_4=f_4(\xB^1,\xB^3,\xB^4,\xB^5)$, and $f_5=f_5(\xB^4,\xB^5)$. The local dependency sets for this example are given in \cref{fig_local1} (b). Before localization, node 1 needs to store all of $\xB^1$ to $\xB^5$ and communicates this information to its neighbors $\{2,4\}$. In contrast, after localization node 1 only keeps $\xB^1$, $\xB^2$, and $\xB^4$, and exchanges some of this information with $\{2,4\}$. In addition, node 1 does not maintain the information of $\yB^3$, $\yB^5$, or $\piBt^3$, $\piBt^5$, either. Before localization, \emph{NEXT} performs 
\[
\underset{\xB_5\in\Kc}{\arg\min}\ft_5(\xB_5;\xB_5[n])+\piBt_5[n]^T(\xB_5-\xB_5[n])
\]
in the local optimization step at node 5, while after localization \cref{a6} performs the following
\[
\underset{(\xB_5^4,\xB_5^5)\in\Kc_4\times\Kc_5}{\arg\min}\ft_5(\xB_5^4,\xB_5^5;\xB_5^4[n],\xB_5^5[n])+\piBt_5^4[n]^T(\xB_5^4-\xB_5^4[n])+\piBt_5^5[n]^T(\xB_5^5-\xB_5^5[n]).
\]
The first terms of the two operations are actually the same. Indeed, since $f_5$ only depends on $\xB^4$ and $\xB^5$, we can definitely choose a surrogate function that only depends on the local storage of $\xB^4$ and $\xB^5$. The difference is node 5 does not have to store, say $\piBt_5^1$, and optimize $\xB_5^1$ based on $\piBt_5^1$. The variable $\piBt_5^1$ is asymptotically tracking $\sum_{j\neq5}\gd_{\xB^1}f_j$. Our localization result says that since node 5 does not have any preference on $\xB^1$, it only follows others' decision of $\xB^1$ through $\piBt_5^1$. Hence, it is unnecessary for node 5 to maintain $\xB_5^1$ and $\piBt_5^1$ -- it can just take other nodes' decision at the end, even through $\xB^1$ and $\xB^5$ are coupled through, say $f_4$. This saves nodes from unnecessary memory storage and communication in the presence of sparse dependency structure, which is crucial when the network is large.
\end{ex}

\begin{ex}\label{ex5-2}
We consider the same communication network but a different dependency structure: $f_1=f_1(\xB^1,\xB^2)$, $f_2=f_2(\xB^2,\xB^3)$, $f_3=f_3(\xB^3)$, $f_4=f_4(\xB^1)$, and $f_5=f_5(\xB^1)$. There are three local dependency sets in this example as shown in \cref{fig_local1} (c). The variable parts that are stored at each node and the communication required to other node for each part are given in \cref{tab:5-1}. Naturally, before localization every node keeps all parts and communicates them with all of its neighbors. On the contrary, the storage and communication required are greatly reduced after localization. Notice that even though node 3 is directly linked with node 4, they do not communicate as there is no common part they both depend on.
\end{ex}

\renewcommand{\arraystretch}{1.2}
\begin{table}[tbhp]
\caption{Storage and communication required in \cref{ex5-2}.}
\label{tab:5-1}
\centering
\small
\begin{tabular}{|c|c|c|}
\hline
Node & Before localization & After localization \\\hline
1 & $\xB^1$, $\yB^1$, $\piBt^1$, $\xB^2$, $\yB^2$, $\piBt^2$, $\xB^3$, $\yB^3$, $\piBt^3$ to 2, 4 & $\xB^1$, $\yB^1$, $\piBt^1$ to 4, $\xB^2$, $\yB^2$, $\piBt^2$ to 2 \\\hline
2 & $\xB^1$, $\yB^1$, $\piBt^1$, $\xB^2$, $\yB^2$, $\piBt^2$, $\xB^3$, $\yB^3$, $\piBt^3$ to 1, 3 & $\xB^2$, $\yB^2$, $\piBt^2$ to 1, $\xB^3$, $\yB^3$, $\piBt^3$ to 3 \\\hline
3 & $\xB^1$, $\yB^1$, $\piBt^1$, $\xB^2$, $\yB^2$, $\piBt^2$, $\xB^3$, $\yB^3$, $\piBt^3$ to 2, 4 & $\xB^3$, $\yB^3$, $\piBt^3$ to 2 \\\hline
4 & $\xB^1$, $\yB^1$, $\piBt^1$, $\xB^2$, $\yB^2$, $\piBt^2$, $\xB^3$, $\yB^3$, $\piBt^3$ to 1, 3, 5 & $\xB^1$, $\yB^1$, $\piBt^1$ to 1, 5 \\\hline
5 & $\xB^1$, $\yB^1$, $\piBt^1$, $\xB^2$, $\yB^2$, $\piBt^2$, $\xB^3$, $\yB^3$, $\piBt^3$ to 4 & $\xB^1$, $\yB^1$, $\piBt^1$ to 4 \\\hline
\end{tabular}
\end{table}

\subsection{Relaxing Assumption L3}\label{sec:dis-2}
Our requirement of connectedness of $\Gc_m$ is not a strong assumption. If any $\Gc_m$ is not connected, we can always ``add" nodes as ``relays" into the local dependency set. For example, consider \cref{ex5-2} with $f_3$ revised to be a constant and $f_4$ revised as $f_4=f_4(\xB^1,\xB^3)$. Then $\Gc_3=(\{2,4\},\phi)$ is not connected. We can add node 3 into $\Gc_3$ by making $f_3=f_3(\xB^3)$ where the dependency is actually trivial. Node 3 can thus relay the information of $\xB^3$ for node 2 and 4. Alternatively, node 1 can also do the relay job. Hence, we assume without loss of generality that the nodes have been added by some algorithm that might depends on network structure and communication requirements so that every $\Gc_m$ is connected.

\subsection{Comparison to Hu et al.}\label{sec:dis-3}
In \cite{Hu_LocalDist_2017}, a similar idea of localization for convex feasibility problems (CFPs) is proposed, where they also exploit the sparsity of the constraints to reduce the storage and communication required. Our framework is different from theirs in two aspects. First, in their framework each node $i$ owns a part of the variable tuple $\xB^i$, whose corresponding ``dependency network graph" $(\Nc_i,\{(i,j):j\in\Nc_i\})$ must be a subgraph of $i$'s local graph $(N(i),\{(j,k)\in\Ec:j\in N(i)\text{ and }k\in N(i)\})$ where $N(i)$ is $i$'s neighbors in the communication graph $\Gc$. On the other hand, in our framework each part of the variable tuple $\xB^m$ does not have specific relation with the nodes and the local dependency set $\Nc_m$ is only required to be a connected component in $\Gc$, which is more general in the sense that their framework is a sub-case of ours. Second, their dependency is built in the constraint sets, while ours is based on objective function's dependency. This difference arises from the nature of CFPs and optimization problems. Namely, we focus on solving the optimal solution for optimization problems while they aim to find the solution lying in the intersection of a batch of sets. Although one would be able to solve convex optimization with CFP algorithms \cite{Borwein_VA_2005}, it is still unclear whether this could be generalized to the case of non-convex optimization.

\subsection{An Example of Approximation Functions Satisfying Assumption N}\label{sec:dis-4}
The Lasry-Lions envelope or double envelope \cite{LasryLions_DoubleEnv_1986}\cite{Rockafellar_VarAna_1998} is a class of approximation functions that serves this purpose. We use this to illustrate the feasibility of our approach but also point out that any such sequence of functions satisfying the assumptions and conditions can be used instead.

\begin{definition}\label{d8}
The double envelope, or Lasry-Lions envelope \cite{LasryLions_DoubleEnv_1986}\cite{Rockafellar_VarAna_1998}, of a function $f$ is defined by
\beq\label{9}
f_{t,s}(x)=\sup_z\inf_y\left\{f(y)+\frac{1}{2t}\|z-y\|^2-\frac{1}{2s}\|x-z\|^2\right\},
\eeq
where $0<s<t<\infty$.
\end{definition}
\begin{fact}\label{f9}
If $f$ is lower bounded, then $\gd f_{t,s}(\cdot)$ is Lipschitz continuous with constant $\max\left\{\frac{1}{s},\frac{1}{t-s}\right\}$.
\end{fact}
\begin{fact}\label{f10}
$f_{t,s}\ra f$ pointwise as $s,t\ra0$. If further we have $f$ being uniformly continuous, then $f_{t,s}\ra f$ uniformly as $s,t\ra0$. Furthermore, $\gd f_{t,s}\ra\gd f$ pointwise as $s,t\ra0$.
\end{fact}
\begin{proof}
See \cite{Aze_DoubleEnv_1993}.
\end{proof}
Now it is clear that if we define
\beq\label{10}
f^*_{i,n}(x)=\sup_z\inf_y\Big\{f_i(y)+\frac{L_{i,n}}{4}\|z-y\|^2-\frac{L_{i,n}}{2}\|x-z\|^2\Big\},
\eeq
then we have $\gd f^*_{i,n}$ being globally Lipschitz continuous with constant $L_{i,n}$. Since $U$ is coercive (Assumption A3), we can restrict our attention to some compact set in $\Kc$, where $f_i$ is uniformly continuous. Then over this set, we will have $\lim_{n\ra\infty}f^*_{i,n}\ra f_i$ uniformly as well.

\subsection{An Example of Sequences Satisfying the Conditions Using $p$-series}\label{sec:dis-5}
Examples of the tuples $(\al[n],\ep[n],\Lmax_n,L^{\min}_n,\taumin_n)$ satisfying the conditions of \cref{t7} exist with all schedules in $p$-series form. Assume $\al[n]=\al_0n^{-\be}$, $\ep[n]=\ep_0n^{-\ga}$, $\Lmax_n=L^{\min}_n=L_0n^{\lm}$, and $\taumin_n=\tau_0n^{-\dl}$ for some positive constants $\al_0$, $\ep_0$, $L_0$, and $\tau_0$. Then the constraints on the parameters are
\begin{equation*}
\left\{\begin{array}{lll}
\lim_{n\ra\infty}\al[n]\frac{(\Lmax_n)^5}{(\taumin_n)^3}=0&\Ra&\be-5\lm-3\dl>0,\\
\sum_{n=0}^\infty(\Lmax_n)^3\left(\frac{\al[n]}{\taumin_n}\right)^2<\infty&\Ra&2\be-3\lm-2\dl>1,\\
\sum_{n=0}^\infty\taumin_n\al[n]=\infty&\Ra&\be+\dl\leq1,\\
\lim_{n\ra\infty}L^{\min}_n=\infty&\Ra&\lm>0,\\
\sum_{n=0}^\infty\al[n]\Lmax_n\ep[n]<\infty&\Ra&\be+\ga-\lm>1.
\end{array}\right.
\end{equation*}
%\lim_{n\ra\infty}|L_{i,n+1}-L_{i,n}|=0\es\forall\es i&\Ra&\lm<1,\\
A possible tuple $(\be,\ga,\lm,\dl)$ satisfying the above equations is $(0.9,0.2,0.05,0.1)$. If the $\gd f_i$'s are Lipschitz continuous and then constant $\tau_i>0$ for all $i$, then $\lm=\dl=0$ and the above requirements degenerate to $0.5<\be\leq1$ and $\ga>1-\be$ as in \cite{Scutari_NEXT_2016}.

\subsection{Examples of Objectives with Unbounded Gradients}\label{sec:dis-6}
When it comes to non-Lipschitz gradients, a first example might be functions with unbounded gradients, but this need not be the only case; for example, $x\sqrt{x}$ on $[0,1]$. Its derivative is $\frac{3\sqrt{x}}{2}$, which is obviously bounded on $[0,1]$ but actually not Lipschitz continuous on $[0,1]$. Though convergence is not established for this case in \cite{Scutari_NEXT_2016}, \emph{NEXT} still works well numerically in these kind of simple examples. Next we turn to more challenging examples with unbounded gradients.

\begin{ex}[Interior local optimum]\label{ex5-3}
Consider $\Nc=\{1,2,3\}$ in a triangle network, with $f_1(x)=x^2-(\ln2)x$, $f_2(x)=x\ln\left(\frac{8}{x^2}\right)$, and $f_3(x)=-x\ln\left(\frac{4}{x^2}\right)$, where $x\in\Kc=[-1,2]$. There is no $G$. The unique stationary solution is $x^*=0$ as $F(x)=x^2$. The derivatives for the node objectives are $f_1'(x)=2x-(\ln2)$, $f_2'(x)=\ln\left(\frac{8}{x^2}\right)-2$, and $f_3'(x)=-\ln\left(\frac{4}{x^2}\right)+2$. Obviously the derivatives are unbounded at $x=0$ for $f_2'(x)$ and $f_3'(x)$, and this example is thus not covered in the theory developed in \cite{Scutari_NEXT_2016}.

On the other hand, we can simply choose the following approximation functions: $f^*_{2,n}(x)=x\ln\left(\frac{8}{x^2+1/p(n)}\right)$ and $f^*_{3,n}(x)=-x\ln\left(\frac{4}{x^2+1/p(n)}\right)$ with derivatives being ${f^*_{2,n}}'(x)=\ln\left(\frac{8}{x^2+1/p(n)}\right)-\frac{2x^2}{x^2+1/p(n)}$ and ${f^*_{3,n}}'(x)=-\ln\left(\frac{4}{x^2+1/p(n)}\right)+\frac{2x^2}{x^2+1/p(n)}$. We can choose $f^*_{1,n}(x)=f_1(x)$ throughout. One may check that $L_{1,n}=O(1)$ and $L_{2,n}=L_{3,n}=O(p(n)^{1/2})$. The conditions in \cref{t7} (b) are checked in the following.
\begin{itemize}[leftmargin=*]
\item
$(\al[n],\Lmax_n,\taumin_n)$ series conditions: suppose we choose $\al[n]=\al_0n^{-0.7}$ and $\taumin_n=\tau>0$, since $\Lmax_n=O(p(n)^{1/2})$, if we further choose $p(n)=n^{0.2}$ then it is evident that all conditions are satisfied.
\item
$\sum_{n=0}^\infty\zeta_n<\infty$: we actually have $F^*_n(x)=F(x)$ for all $x$ and hence $\zeta_n=0\es\forall\es n$.
\item
$\lim_{n\ra\infty}\frac{\etamax_n}{\taumin_n}=0$: the maximum differences of derivatives for node 2 and 3 always occur at $x=0$, and $\gd f^*_{2,n}(0)=\gd f^*_{3,n}(0)=O(\log p(n))=O(\log n)$ for the choice $p(n)=n^{0.2}$. The condition holds because $\log n-\log(n-1)=O(1/n)\ra0$.
\item
$\sum_{n=0}^\infty\frac{\al[n]\Lmax_n\etamax_n}{\taumin_n}<\infty$: notice that for our choices of $\al[n]=\al_0n^{-0.7}$ and $p(n)=n^{0.2}$ the summed term equals $O(n^{-0.7+0.2/2-1})$ and thus is summable.
\item
$\lim_{n\ra\infty}\gd F^*_n\ra\gd F$: notice $\gd F^*_n(x)=\gd F(x)$ for all $x$.
\end{itemize}
As we have $\gd F$ finite in $\Kc$, by \cref{t7} it is guaranteed that our algorithm converges to the unique stationary solution $x=0$, which is also a global minimum in this example. Note that this minimum lies in the interior of $\Kc$.
\end{ex}

\begin{ex}[Boundary local optimum]\label{ex5-4}
Consider a one node network, and the objective function is $f(x,y)=\sqrt{1-(x^2+y^2)}+\frac{x}{2}$ on the region inside the unit circle $\Kc=\{(x,y):x^2+y^2\leq1\}$, so that the graph of $f(x,y)$, namely $(x,y,f(x,y))$, is the upper half of the unit sphere lifted in the direction of positive $x$-axis. For the upper half of the unit sphere, the set of global minima is the unit circle; now that we tilt the sphere, the unique global minimum is $(-1,0)$. There is no stationary solution inside the unit circle. Since the gradient
\[
\gd f(x,y)=\begin{bmatrix}\frac{1}{2}-\frac{x}{\sqrt{1-(x^2+y^2)}}\\-\frac{y}{\sqrt{1-(x^2+y^2)}}\end{bmatrix}
\]
is unbounded on the unit circle, the definition of stationary solution fails there.
This example again clearly lies outside the theory of \cite{Scutari_NEXT_2016}.

We consider the approximation function $f^*_n(x,y)=\sqrt{1-(x^2+y^2)+1/p(n)}+\frac{x}{2}$. Its gradient can also be obtained by changing the $1$'s in the denominators of the original gradient to $1+1/p(n)$. We have $L_n=O(n^{3/2})$.
\begin{itemize}[leftmargin=*]
\item
$(\al[n],\Lmax_n,\taumin_n)$ series conditions: choose $\al[n]=\al_0n^{-0.8}$ and $\taumin_n=\tau>0$, and $p(n)=n^{0.1}$.
\item
$\sum_{n=0}^\infty\zeta_n<\infty$: notice that $F^*_n\downarrow F$ monotonically. Since the largest decreasing of $F^*_n-F^*_{n-1}$ always happen on the unit circle, $\sum_{n=0}^\infty\zeta_n<\infty$ is simply $F^*_1-F$ evaluated on the unit circle, which is $\sqrt{2}$.
\item
$\lim_{n\ra\infty}\frac{\etamax_n}{\taumin_n}=0$: roughly $\|\gd f^*_n\|=O(\sqrt{p(n)})=O(n^{0.05})$ for the choice of $p(n)=n^{0.1}$. The condition is true because $n^{0.05}-(n-1)^{0.05}=O(n^{-0.95})\ra0$.
\item
$\sum_{n=0}^\infty\frac{\al[n]\Lmax_n\etamax_n}{\taumin_n}<\infty$: the term is decaying in the rate of $O(n^{-0.8+0.15-0.95})$, which is summable.
\item
$\lim_{n\ra\infty}\gd F^*_n\ra\gd F$: this one is obvious as we only have one node.
\end{itemize}
Our method will not converge to any point in $int(\Kc)$; otherwise, by part (b) of \cref{t7} we know it must be a stationary solution, but there is no stationary solution inside the unit circle. Thus, our method will converge to some point in $bd(\Kc)$; since $\gd F$ is infinite on the unit circle, this falls into the case of \cref{t7} (c), and the point is not guaranteed to be a local minimum. Even so, we find in \cref{sec:sim-1} that for a wide range of initialization of $x$, our algorithm can converge to $(-1,0)$ while \emph{NEXT} does not. Unfortunately, since the objective is symmetric with respect to $x$-axis and the unique global maximum is $(\frac{1}{\sqrt{5}},0)$, if we start with any point to the right of the maximum on the $x$-axis, the process will inevitably approach $(1,0)$ in the limit.
\end{ex}

We will see in \cref{sec:sim-1} that \emph{NEXT} fails numerically in the above examples while our method works much better.

%%%%%%%%%%%%%%%%%%%%%%%%%%%%%%%%%%%%%%%%%%%%%%%%%%
%%%%%  Applications to Resource Allocation   %%%%%
%%%%%%%%%%%%%%%%%%%%%%%%%%%%%%%%%%%%%%%%%%%%%%%%%%

\section{Application to Resource Allocation}\label{sec:app}
We now describe how to apply our algorithmic framework to wireless resource allocation, and along way also describe how the two issues that motivated our generalizations arise.

\subsection{Problem Formulation}\label{sec:app-1}
We consider an OFDMA wireless cellular network, where a set of base stations (BSs) $B$ transmit downlink data to users through the set of channels or resource blocks (RBs) $K$. For a BS $b\in B$, $I_b$ denotes the set of users associated with it, which is an input that's fixed. The transmitted power of BS $b$ in channel $k$ is denoted by $p_{bk}$, and the maximum total sum power transmitted by BS $b$ is limited to $P_b$. The allocation variable of BS $b$ to user $i$ in channel $k$ is denoted by $x_{bik}$, with gain $g_{bik}$: $x_{bik}=1$ means $b$ transmits to $i$ in the RB $k$, and $x_{bik}=0$ otherwise. For all users, we also introduce a scheduling weight, $w_i$ for user $i$. Finally, $\sg^2$ is the variance of the independent zero mean additive white Gaussian noise (AWGN) for all BSs. We assume that BS $b$ only possesses the information of $\{g_{b'ik}:b'\in B,i\in I_b,k\in K\}$. In other words, BS $b$ can only compute the weighted-sum rate of the users associated with itself (knowing the powers of the other base-sites). This is a reasonable assumption, as each user equipment (UE) reports its measured channel gains to the BS it is associated with, whereas all the channel gains of UEs served by other BSs is unknown.

When the BS $b$ transmits a non-zero power in channel $k$, it interferes with all the other transmissions in channel $k$. However, owing to propagation-based loss, the powers of nearby BSs will dominate the whole interference term. Hence, with the definition that $N(b)$ is the neighboring BSs of BS $b$, we can neglect all the interference from $b'\not\in N(b)$ to $b$. This is a modeling assumption that is reasonably accurate in practice, and will be in force henceforth. For ease of exposition we assume that neighbor relation is mutual, i.e. $b'\in N(b)$ if and only if $b\in N(b')$. In terms of the interference graph, where nodes are BSs and edges only exist between BSs that interfere with each other, we reduce a complete graph to an undirected and connected one. Our work can be trivially extended to the directed case assuming that strong connectivity holds.

We consider a one-shot weighted sum-rate maximization problem, subject to the allocation limit constraint, the power limit constraint, non-negative power constraint, and the fact that $x_{bik}$ is either $0$ or $1$; we will justify the weighted sum-rate maximization problem in the next section. To make the overall constraint set convex, we relax the integer constraints on $x_{bik}$ as in \cite{VJ_DLSchedule_2009,VJ_ResourceAlloc_2009}. In future work we will study appropriate integer rounding schemes.

The joint power control and scheduling problem (P1) is then formalized as:
\bal\label{1}
\text{(P1)}\qquad\max_{p_{BK},x_{BI(B)K}}&\sum_{b\in B}\sum_{i\in I_b}w_i\sum_{k\in K}x_{bik}\log\left(1+\frac{\Ga_{bik}}{\sg^2+\Gab_{bik}}\right)\\
\text{subject to }&\sum_{i\in I_b}x_{bik}\leq1\quad\forall\es b\in B,k\in K\\
&\sum_{k\in K}p_{bk}\leq P_b\quad\forall\es b\in B\\
&0\leq x_{bik}\leq 1\quad\forall\es b\in B,k\in K,i\in I_b\\
&0\leq p_{bk}\quad\forall\es b\in B,k\in K,
\eal
where $\Ga_{bik}=p_{bk}g_{bik}$ and $\Gab_{bik}=\sum_{b'\in N(b)}p_{b'k}g_{b'ik}$ are the signal and interference for user $i$ in channel $k$. We use $p_{BK}$ to refer to the collection of the variables $p_{bk}\es\forall\es b\in B,k\in K$; also, $x_{BI(B)K}$ can be viewed in a similar way, where $I(B)\triangleq\buU_{b\in B}I_b$. This is a shorthand for easy referencing.

As we will solve the problem in a distributed manner, we let each BS maintain the decision variables $p_{BK}$. Denote the copy of $p_{b'k}$ at BS $b$ by $p^b_{b'k}$ for all $b'\in B,k\in K$. The idea is to perform the optimization at each BS, and then enforce consensuses of the decision variables among all BSs, transforming (P1) into (P2) given in the following:
\begin{equation}\label{2}
\text{(P2)}\qquad\max_{p^B_{BK},x_{BI(B)K}}\sum_{b\in B}\sum_{i\in I_b}w_i\sum_{k\in K}x_{bik}\log\big(1+\frac{\Ga^b_{bik}}{\sg^2+\Gab^b_{bik}}\big)
\end{equation}
\begin{comment}
\bal\label{2}
\max_{p^B_{BK},x_{BI(B)K}}&\sum_{b\in B}\sum_{i\in I_b}w_i\sum_{k\in K}x_{bik}\log\left(1+\frac{\Ga^b_{bik}}{\sg^2+\Gab^b_{bik}}\right)\\
\text{subject to }&\sum_{i\in I_b}x_{bik}\leq1\quad\forall\es b\in B,k\in K\\
&\sum_{k\in K}p^b_{b'k}\leq P_b'\quad\forall\es b,b'\in B\\
&0\leq x_{bik}\leq 1\quad\forall\es b\in B,k\in K,i\in I_b\\
&0\leq p^b_{b'k}\quad\forall\es b,b'\in B,k\in K\\
&p^{b_1}_{b'k}=p^{b_2}_{b'k}\quad\forall\es b_1,b_2,b'\in B,k\in K,
\eal
\end{comment}
where $\Ga^b_{bik}=p^b_{bk}g_{bik}$ and $\Gab^b_{bik}=\sum_{b'\in N(b)}p^b_{b'k}g_{b'ik}$. The set of constraints includes all the constraints in \cref{1} now with $p^b_{b'k}$ and the second and fourth constraints hold for all copies at $b\in B$, and an additional constraint that the consensus is reached $p^{b_1}_{b'k}=p^{b_2}_{b'k}\quad\forall\es b_1,b_2,b'\in B,k\in K$.% At this stage we haven't stated the exact meaning of ``performing optimization at each BS," but this will become clear later in the paper.
%Note that for BS $b$ without the knowledge of channel gains of other BSs $b'\neq b$ (and also $w_i$ where $i\not\in I_b$), it doesn't really know the rates of other BSs. This issue will be addressed later in a more general way.

We can split the $\log(\cdot)$ term in the objective to two parts 
$x_{bik}\log(\sg^2+\Ga^b_{bik}+\Gab^b_{bik})$ and $-x_{bik}\log(\sg^2+\Gab^b_{bik})$, and then modify the former to be $x_{bik}\log(\sg^2+\frac{\Ga^b_{bik}+\Gab^b_{bik}}{x_{bik}})$ as in \cite{VJ_DLSchedule_2009,VJ_ResourceAlloc_2009}, which is jointly strictly concave in $x_{bik}$ and $p^b_{Bk}$. We define the modified function to be 0 when $x_{bik}=0$ so that it is continuous. Then (P2) becomes the following:
\beq\label{3}
\text{(P3)}\quad\max_{p^B_{BK},x_{BI(B)K}}
\sum_{b\in B}\sum_{i\in I_b}w_i\sum_{k\in K}\left[x_{bik}\log\left(\sg^2+\frac{\Ga^b_{bik}+\Gab^b_{bik}}{x_{bik}}\right)-x_{bik}\log\big(\sg^2+\Gab^b_{bik}\big)\right],
\eeq
subject to the same set of constraints as in \cref{2}. Note that (P2) and (P3) are the same if $x_{bik}$'s are restricted to be integers, that is, $x_{bik}\in\{0,1\}$.

We will now reiterate the two issues we identified earlier with existing distributed optimization approaches but in the specific context of \cref{3}. If we take the approach in \cite{Scutari_NEXT_2016}, then every BS would need to keep a copy of all the decision variables, both $p_{BK}$ and $x_{BI(B)K}$, and perform consensus on them and also any relevant gradient terms. This is simply impractical and forces the localization idea. We also note that \cref{3} contains functions of the form $x\log
\big(a+(p+p')/x\big)-x\log(a+p')$ that are smooth but where the gradients are not Lipschitz; in particular $x\log(a+(p+p')/x)$ has some terms of its gradient becoming infinite when $x\downarrow0$. These functions clearly fall outside the framework of \cite{Scutari_NEXT_2016}, and force approaches like our proximal approximations scheme. In the following, we apply the developed distributed optimization framework to the problem in \cref{1}-\cref{3}. In the problem, the set of BSs in the cellular network $B$ corresponds to $\Nc$ in the framework, a BS $b$ corresponds to a node $i$, and the set of edges $\Ec$ in the framework is the one-tier interference graph here. There are many different ways to apply our framework to the resource allocation, which we discuss in detail next.

\subsection{Direct Method}\label{sec:app-2}
In this method, we directly let $f_b$ be the weighted sum-rate of BS $b$: $f_b=-\sum_{i,k}w_ix_{bik}\log\left(1+\frac{\Ga_{bik}}{\sg^2+\Gab_{bik}}\right)$, and $G=0$. In the first version of this method, every BS $b$ keeps the powers of all BSs as decision variables, that is, we have $p^b_{b'k}$ for all $b,b'\in B$. But instead of keeping $x^b_{b'ik}$'s as decision variables at BS $b$ for all $b'\in B$ if we exactly follow \emph{NEXT}, we allow every BS $b$ only keeps its own $x^b_{bik}$, which we denote as $x_{bik}$ for short. As we see from \cref{sec:main-3}, this suffices because BS $b$ dictates the decision of $x_{bik}$.

If we reconsider the problem from the perspective of our localization framework \cref{sec:main-1}, there are $2|B|$ local dependency sets. There are $|B|$ local dependency sets $\{b\}$ for all $b\in B$ which correspond to the variables $x_{bI_bK}$, and $|B|$ local dependency sets $Nb(b)\triangleq N(b)\uU\{b\}$ for all $b\in B$ which correspond to the variables $p_{bK}$. In the first version of the direct method, referred to as the Localized X Globalized P-diRect Method (LXGP-RM) algorithm, only $x_{BI(B)K}$ follows the localization framework, $p_{BK}$ is still globalized in the sense that every BS keeps a copy of the whole variable.

We only present algorithm in words here, for the pseudo code see Appendix B. The algorithm basically proceeds as \cref{a6} with the common variable $p_{BK}$ and $|B|$ local variables $x_{bI_bK}\es\forall\es b$ (and without approximation functions since objectives have Lipschitz gradients). At BS $b$ we use $r^b_{b'k}$ to track $\frac{1}{|B|}\sum_{b''\in B}\frac{\pa f_{b''}}{\pa p_{b'k}}$ and $\pit^b_{b'k}$ to track $\sum_{b''\neq b\in B}\frac{\pa f_{b''}}{\pa p_{b'k}}$, which correspond to $\yB$ and $\piB$ in \cref{a6}. In each iteration, we let $\al[n]=\frac{\al_0}{(n+1)^\be}$, and BS $b$ performs the minimization of 
\beq\label{11-1}
\ft_b(p^b_{BK},x_{bI(b)K};\pb^b_{BK},\xb_{bI(b)K})+\pit^b_{BK}\cdot(p^b_{BK}-\pb^b_{BK})
\eeq
with respect to the variables $p^b_{BK}$ and $x_{bI(b)K}$, and $\pb^b_{BK}$, $\xb_{bI(b)K}$, and $\pit^b_{BK}$ being the current iterate of the variables. The surrogate function $\ft_b(\pB^b,\xB_b;\pBb^b,\xBb_b)$ is chosen as
\begingroup\makeatletter\def\f@size{8.1}\check@mathfonts
\def\maketag@@@#1{\hbox{\m@th\normalsize\normalfont#1}}%
\bal\label{11}
&f_b+\frac{\tau_b}{2}\Big[\sum_{i,k}(x_{bik}-\xb_{bik})^2+\sum_{b'\in B,k}(p^b_{b'k}-\pb^b_{b'k})^2\Big]\\
&+\sum_{i,k}\frac{\pa f_b}{\pa x_{bik}}\cdot(x_{bik}-\xb_{bik})+\sum_{b'\in Nb(b),k}\frac{\pa f_b}{\pa p^b_{b'k}}\cdot(p^b_{b'k}-\pb^b_{b'k}),
\eal
\endgroup
where $f_b$ and $\frac{\pa f_b}{\pa p^b_{b'k}}$ are functions of $(\pBb^b,\xBb_b)$, but $\frac{\pa f_b}{\pa x_{bik}}$ is just a function of $\pBb^b$. The quadratic term in \cref{11} is to maintain the strict convexity of the surrogate. Finally, we have a universal doubly stochastic matrix $W$ to average $p_{BK}$ and $r_{BK}$ for them to reach consensus. We remark that with the objective in \cref{11-1} only having up to quadratic terms and our constraints being linear, the minimization can be solved efficiently using quadratic programming (QP) with coefficient matrices of the quadratic term being positive-semidefinite \cite{Kozlov_QP_1980}.

The second version of the direct method, termed as the Localized X Localized P-diRect Method (LXLP-RM) algorithm, makes better use of our localization idea as opposed LXGP-RM so that for all $b\in B$ we now only maintain the variables $p^b_{b'k}$ for $b'\in Nb(b)$ in BS $b$, as the BSs in $Nb(b')$ dictate the decision of $p_{b'k}$; the variables $x_{BI(B)K}$ still follow the localization framework just as before in LXGP-RM.
The main change from LXGP-RM is that the index set of the variable tuple $B$ is now replaced by $Nb(b)$, as BS $b$ does not keep the variable $p^b_{b'k}$ for $b'\not\in Nb(b)$ any more. As a result, the steps regarding the weighted sum for reaching consensus need to be modified. We introduced the matrix $W(b)$ for each BS $b$. The matrix $W(b)$ concerns with the weighting of the local dependency set $Nb(b)$ regarding the variable $p_{bK}$, that is, $\Gc(b)=(Nb(b),\Ec(b))$ where $\Ec(b)=\{(i,j)\in\Ec:i,j\in Nb(b)\}$ if the whole network is $\Gc=(B,\Ec)$. Its $i$-th row $W_{i:}(b)$ and $j$-th column $W_{:j}(b)$ should be zero if and only if $i\not\in Nb(b)$ or $j\not\in Nb(b)$. After deleting all these zero rows and columns, it would become a doubly-stochastic matrix, as described in Assumption L4.

The second version of the direct method still follows the framework of \cref{a6}, with $2|B|$ local variables $p_{bK},x_{bI_bK}\es\forall\es b$ and no common variable. Now at BS $b$ we use $r^b_{b'k}$ to track $\tfrac{1}{|Nb(b')|}\sum_{b''\in Nb(b')}\tfrac{\pa f_{b''}}{\pa p_{b'k}}$, and $\pit^b_{b'k}$ to track $\sum_{b''\neq b\in Nb(b')}\tfrac{\pa f_{b''}}{\pa p_{b'k}}$, where $b'\in Nb(b)$. Also, the surrogate $\ft_b$ is minorly changed so that now the quadratic term of $p^b_{b'k}$ only sums over $b'\in Nb(b)$ instead of $b'\in B$.

We finally remark that one can also consider a GXGP-RM algorithm (basically \emph{NEXT} from \cite{Scutari_NEXT_2016}) where copies of both the power and allocation variables are maintained at each node, or even a GXLP-RM algorithm where localization is performed only for the power variables. Note that only the fully localized scheme, i.e. LXLP-RM, will be scalable in practice. However, we will evaluate its performance relative to the other schemes.
\begin{comment}
\bal
\ft_b(\pB^b,\xB_b;\pBb^b,\xBb_b)&=f_b(\pBb^b,\xBb_b)+\sum_{i,k}\frac{\pa f_b}{\pa x_{bik}}(\pBb^b)\cdot(x_{bik}-\xb_{bik})+\sum_{k}\frac{\pa f_b}{\pa p^b_{bk}}(\pBb^b,\xBb_b)\cdot(p^b_{bk}-\pb^b_{bk})\\&+\sum_{k}\sum_{b'\in N(b)}\frac{\pa f_b}{\pa p^b_{b'k}}\cdot(p^b_{b'k}-\pb^b_{b'k})+\frac{\tau_b}{2}\left[\sum_{i,k}(x_{bik}-\xb_{bik})^2+\sum_{b'\in Nb(b)}(p^b_{b'k}-\pb^b_{b'k})^2\right].
\eal
The partial derivatives remain the same.
\end{comment}

\subsection{Decomposed Method}\label{sec:app-3}
In (P3), there is a part of the objective that is concave (or convex after taking minus sign), and the optimization of this part should be easy. The algorithm might run faster if we properly exploit this fact. To achieve this goal, let us assume that the channel gains $g_{bik}$'s are known to all BSs. Then we could apply the framework in Section III by letting $f_b=\sum_{i,k}w_ix_{bik}\log(\sg^2+\Gab_{bik})$ and $G=-\sum_{b,i,k}w_ix_{bik}\log\big(\sg^2+\tfrac{\Ga_{bik}+\Gab_{bik}}{x_{bik}}\big)$.

As $G$ is in general a function of not only $p_{bk}$ but also $x_{bik}$ for all $b\in B$ and we need to optimize $G$ at every BS, this method does not allow any localization. In other words, the tuple consisting of all decision variables is the common variable $\xB^c$ in \cref{a6} itself, and there is no local dependency set besides $\Nc$ itself. At BS $b$ we need to maintain $p^b_{b'k}$ as well as $x^b_{b'ik}$ $\forall\es b'\in B$. Note that with the derivatives of $f_b$ being Lipschitz continuous and no localization, this method is a direct application of \cite{Scutari_NEXT_2016}.

The algorithm, which we call the Globalized X Globalized P-deComposed Method (GXGP-CM), is also largely the same as LXGP-RM, except that now we have to optimize $G$ as well, and we need to maintain and update $x^b_{b'ik}$. The algorithm and the surrogate function also need minor revisions (see Appendix B).

\subsection{Partially Linearized Method}\label{sec:app-4}
While the decomposed method enjoys the benefit of using the intrinsic convex part in the objective, it is impractical since it requires every BS knows all channel gains. We could instead put the convex part in $f_b$ as well, and take advantage of it by not linearizing it when forming the surrogate function. Specifically, we let $f_b=f_{b\uU}+f_{b\dU}$ where $f_{b\uU}=-\sum_{i,k}w_ix_{bik}\log\left(\sg^2+\frac{\Ga_{bik}+\Gab_{bik}}{x_{bik}}\right)$ and $f_{b\dU}=\sum_{i,k}w_ix_{bik}\log\left(\sg^2+\Gab_{bik}\right)$. Then we can choose the surrogate function $\ft_b(\pB^b,\xB_b;\pBb^b,\xBb_b)$ as $f_{b\uU}(\pB^b,\xB_b)+\ft_{b\dU}(\pB^b,\xB_b;\pBb^b,\xBb_b)$, where $\ft_{b\dU}$ has exactly the same form as in \cref{11} (for LXGP case), i.e., linearized with the current iterate $(\pBb^b,\xBb_b)$ plus the quadratic terms.

With $\gd f_b$ not being Lipschitz continuous, we have to apply the approximation functions detailed in \cref{sec:main-2} to guarantee convergence. We may choose $f^*_{b,n}=f^*_{b\uU,n}+f_{b\dU}$ where
\beq\label{11-2}
f^*_{b\uU,n}=-\sum_{i,k}w_i(x_{bik}+e[n])\log\Big(\sg^2+\frac{\Ga_{bik}+\Gab_{bik}}{x_{bik}+e[n]}\Big).
\eeq
One can easily show that $\gd f^*_{b,n}$ is Lipschitz continuous with constant $L_{b,n}$ the reciprocal of $e[n]$. We can then choose a schedule of $e[n]\ra0$ according to \cref{t7} and \cref{sec:dis-5}. We refer to this method as the Partially Linearized method (PL) algorithm, which could be LXLP or any of the other combinations. Note that we do not have the guarantee of convergence to stationary solution in this case because the objective function has unbounded gradient on the boundary.

\subsection{Consensus Scheme}\label{sec:app-5}
Let $\Gc=(B,\Ec)$ be the BS network we are considering, and let $d_i$ be the degree of BS $i$. The choice of $W$ must meet the following two criteria to conform to Assumption L4: (1) it must be doubly-stochastic; (2) $W_{ij}\geq0$ is non-zero if and only if $(i,j)\in\Ec$. We denote the set of $W$'s that satisfy these criteria as $\Om$, which is a subset in $\R_+^{|B|\times|B|}$.
We choose $W$ as follows
\beq\label{20}
W_{ij}=\left\{
\begin{array}{cl}0&\text{if }j\not\in N(i)\\\frac{1}{\bar{d}}&\text{if }j\in N(i)\text{ and }i\neq j\\\frac{\bar{d}-d_i}{\bar{d}}&\text{if }j\in N(i)\text{ and }i=j\end{array}
\right.,
\eeq
where $\bar{d}=\max_id_i+1$. It is easy to verify that this choice of $W$ is row-stochastic. Since $W$ is symmetric, it is then also column-stochastic. By definition of $\bar{d}$, we will have $W_{ii}>0$. By construction, $W_{ij}>0$ if $(i,j)\in\Ec,\es i\neq j$.

In LXLP-RM algorithm we need a $W(b)$ for every BS $b$. Let $\Ec(b)=\{(i,j)\in\Ec:i,j\in Nb(b)\}$. Then we choose $W(b)$ as described above but treat $\Gc$ as $\Gc(b)=(Nb(b),\Ec(b))$.

With symmetric weights $W=W^T$, \cite{Boyd_Average_2004} suggests that the best convergence speed is obtained with the solution of 
\beq\label{22-1}
\min\|W-\1_{|B|}\1^T_{|B|}/|B|\|_2\text{ subject to }W\in\Om.
\eeq
For symmetric graphs, the optimized result is $W_{ij}=\frac{1}{d_i}$ when $(i,j)\in\Ec$ and $W_{ij}=0$ otherwise, which is exactly our choice.% for the simulations.

%%%%%%%%%%%%%%%%%%%%%%%%%%%%%%%%%%%%%%%%%%%%%%%%%%
%%%%%           Simulation Results           %%%%%
%%%%%%%%%%%%%%%%%%%%%%%%%%%%%%%%%%%%%%%%%%%%%%%%%%

\section{Simulation Results}\label{sec:sim}
In \cref{sec:sim-1} we compare the performance of \cref{a6} with \emph{NEXT} of the examples with unbounded gradients on either interior point or boundary point described in \cref{sec:dis-6}. In \cref{sec:sim-2} we compare our algorithms with single-cell scheduling and resource allocation method.

\subsection{Approximation Functions}\label{sec:sim-1}
\cref{fig8} depicts how the local decision variables in \cref{ex5-3} change for \emph{NEXT} and \cref{a6} within $500$ iterations. We start with initial value of $(x_1[0],x_2[0],x_3[0])=(-1,-0.5,-0.25)$. We choose the following parameters: $\tau=0.1$ (independent of $i$ and $n$), $\al[n]=0.98n^{-0.7}$, $p(n)=10n^{0.2}$,
\[
\WB=\begin{bmatrix}1/2&1/4&1/4\\1/4&1/2&1/4\\1/4&1/4&1/2\end{bmatrix},
\]
and the surrogate functions are chosen to be direct linearization plus the quadratic regularization term (see \cref{11} for an example of such kind of choices). In \cref{fig8} (a) we can see that \emph{NEXT} oscillates and is not numerically stable for this example. Whenever the iterates go near the global minimum at $x=0$, they jump to values far away. This is because the gradients $\gd f_2$ and $\gd f_3$ are infinite at the point, even though $\gd F$ is zero there. The trackings of $(x_1,x_2,x_3)$ to $\xb$ and $(y_1,y_2,y_3)$ to $\yb$ are in the fast time-scale (see \cref{sec:sa}) and actually happen quite fast, as can be seen in the figure. However, a slight mismatch between $(x_1,x_2,x_3)$ and $(y_1,y_2,y_3)$ is sufficient to cause very large $\{\piBt_i[n]+\gd f_i[n]\}_i$, which is supposed to be very small when $\xb\approx0$, driving $x_i$'s to the boundary. If two of them jump to $2$ and one of them to $-1$, then in the next few iterations they jump up; if two jump to $-1$ and one to $2$, they jump down.

\setlength{\textfloatsep}{3pt}
\begin{figure}[t!]
\centering
\subfloat[\emph{NEXT}]{
\includegraphics[scale=0.33]{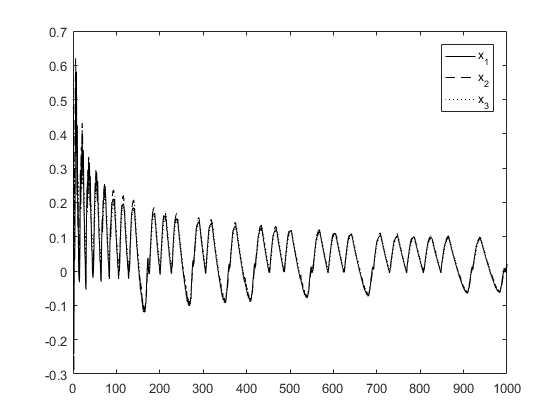}}
\label{fig8a}\hfill
\subfloat[\cref{a6}]{
\includegraphics[scale=0.33]{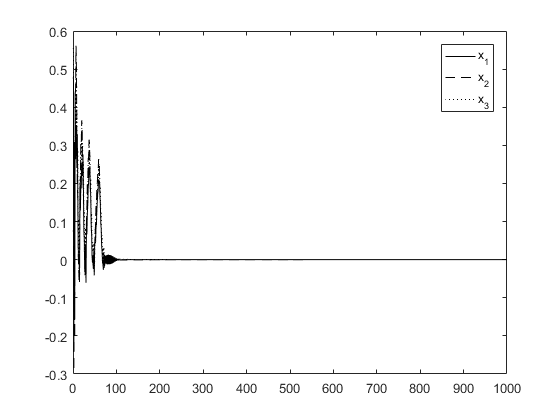}}
\label{fig8b}\hfill
\subfloat[\cref{a6} with wrong $f^*_{2,n}$]{
\includegraphics[scale=0.33]{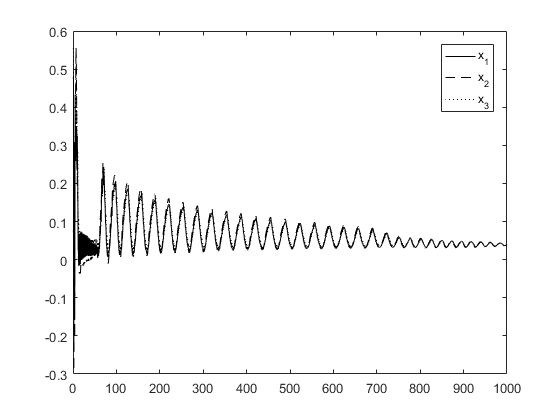}}
\label{fig8c}\hfill
\caption{Comparison between \emph{NEXT} and \cref{a6} for the local variables versus iterations in \cref{ex5-3}.}
\label{fig8}
\end{figure}

We cannot ensure that \emph{NEXT} is not converging when $n\ra\infty$ theoretically, but we observe it is still oscillating when $n$ is as large as $10^4$. In contrast, \cref{a6} essentially converges to the global minimum $x=0$ within $150$ iterations in \cref{fig8} (b). In \cref{fig8} (c), we show the case when $f^*_{2,n}(x)=x\ln\left(\frac{8}{x^2+1.1/p(n)}\right)$, where we have everything satisfied except $\lim_{n\ra\infty}\gd F^*_n\ra\gd F$ -- one can check there will be an additional $\ln(1/1.1)$ term. From the figure we see that not only it exhibits oscillating behavior, but it seems to converge to a wrong point other than $x=0$.

The converging behaviors of \emph{NEXT} and \cref{a6} for \cref{ex5-4} are shown in \cref{fig9} for two different initializations. The parameters are $\tau=0.05$, $\al[n]=0.98n^{-0.85}$, and $p(n)=10n^{0.1}$. The surrogate function is again direct linearization plus quadratic regularizer. The 2-D iterates for both algorithms are plotted from red to blue, with \emph{NEXT} being circle dots and \cref{a6} being square dots. In \cref{fig9} (a), we start with $(0.5,0.5)$, and both algorithms are executed for $5000$ iterations (the dots are down sampled though). We see that while our algorithm is converging to the global minimum at $\xB=(-1,0)$, \emph{NEXT} is ``converging" to some point $(-0.2499,0.9683)$ near the boundary. Due to the unbounded gradient near the boundary, the $\frac{1}{2}$ term in $\pa_xf(x,y)$ is completely dominated by the rest $(-x,-y)/\sqrt{1-(x^2+y^2)}$, which directs the iterate only to descend in the radial direction. Again we cannot ensure \emph{NEXT} does not converge to $(-1,0)$ if we run it forever; however, it does not visit the region $x<-0.25$ after $10^5$ iterations. By slowly changing the objective, we are able to escape this dominance and obtain the correct solution.

In (b) we start from $\xB=(0.6,0)$, which is to the right of the global maximum $(\frac{1}{\sqrt{5}},0)$, for $100$ iterations, and both algorithms are converging to $(1,0)$. Since we start on the $x$-axis and the gradients always direct to $(1,0)$, without any perturbation there is no way to escape $x$-axis for all gradient-descent-like methods. \cref{ex5-4} falls into the case (c) in \cref{t7}, where the algorithm is converging to some point in the boundary and $\gd F$ is not bounded. The definition of stationary solution does not apply, and \emph{NEXT} fails to converge to global minimum for both $\xB[0]=(0.5,0.5)$ and $\xB[0]=(0.6,0)$, while \cref{a6} succeeds for $\xB[0]=(0.5,0.5)$ but also fails for $\xB[0]=(0.6,0)$.

\setlength{\textfloatsep}{3pt}
\begin{figure}[t!]
\centering
\subfloat[$\xB{[0]}=(0.5,0.5)$]{
\includegraphics[scale=0.4]{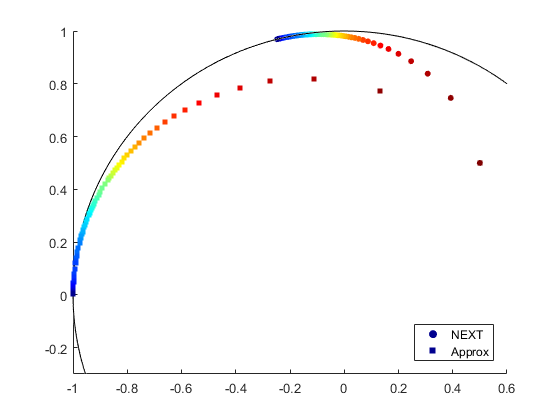}}
%The ranges in z \in [..] contains square brackets which confuses LaTeX when reading the optional argument of \subfloat. So, you need to conceal it by using {[...]}
\label{fig9a}\hspace{20pt}
\subfloat[$\xB{[0]}=(0.6,0)$]{
\includegraphics[scale=0.4]{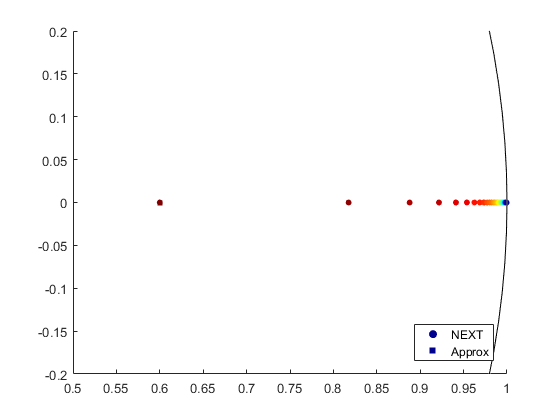}}
\label{fig9b}
\caption{Comparison between \emph{NEXT} and \cref{a6} for the evolution of the local variable in \cref{ex5-4}.}
\label{fig9}
\end{figure}

\subsection{The Resource Allocation Application}\label{sec:sim-2}
We adopt the framework of the network utility maximization problem as in \cite{VJ_DLSchedule_2009,VJ_ResourceAlloc_2009} where we maximize
\beq\label{22}
U(\RB_T)=\sum_{i\in I(B)}U_i(R_{i,T}),
\eeq
where $U_i(\cdot)$ is given by
\beq\label{23}
U_i(\RB_{i,t})=
\left\{
\begin{array}{cl}
%\frac{c_i}{\al}(W_{i,t})^\al,&\al\leq1,\al\neq0,\\
%c_i\log(W_{i,t}),&\al=0,
\frac{c_i}{\eta}(R_{i,t})^\eta,&\eta\leq1,\eta\neq0,\\
c_i\log(R_{i,t}),&\eta=0,
\end{array}
\right.
\eeq
$R_{i,t}$ is the average throughput of user $i$ up to time $t$, $\eta\leq 1$ is the fairness parameter, and $c_i$ is a QoS weight. The gradient-based scheduling approach \cite{VJ_Approx_2009} leads to solving the optimization problem given below at each time instance
\beq\label{24}
%\max_{\rB_t\in\Rc(e_t)}\sum_ic_i(W_{i,t})^{\al-1}r_{i,t}.
\max_{\rB_t\in\Rc(e_t)}\sum_ic_i(R_{i,t})^{\eta-1}r_{i,t}.
\eeq
This is exactly the one-shot optimization problem we consider in \cref{1}, where $w_i=c_i(R_{i,t})^{\eta-1}$ is the weight of user $i$, $r_{i,t}$ is the rate of user $i$ given by the Shannon capacity, and $\Rc(e_t)$ is the capacity region dependent on current channel state $e_t$ and constrains the choice of $r_{i,t}$ as the constraints set in \cref{1}.

Now consider problem \cref{1}. A naive solution would be disregarding the interference and solving the resource allocation and scheduling for each cell separately. The optimization for a single-cell is well solved in literature, e.g. in \cite{VJ_DLSchedule_2009}. Specifically, neglecting the interference, for a BS $b$ we can solve
\beq\label{25}
\max_{p_{bK},x_{bI_bK}}\sum_{i\in I_b}w_i\sum_{k\in K}x_{bik}\log\left(1+\frac{p_{bk}g_{bik}}{\sg^2x_{bik}}\right),
\eeq
subject to the constraints. This is a convex problem, and can be solved with existing methods in convex optimization. We call this method the Single-Cell No-Iteration (SC-NI) algorithm.

A refinement of the SC-NI algorithm is to update the interference terms after first optimization for each cell. We then optimize for each cell again while treating the powers of neighboring BSs as constants, and then repeat until convergence. We call this the Single-Cell (SC) algorithm.

\begin{comment}
\begin{table}[!t]
\centering
\caption{Simulation Parameters.}
\label{tab1}
\begin{tabular}{|l|l|}
\hline
\textit{\textbf{Parameter}} & \textit{\textbf{Value}} \\ \hline
$|B|$   & 4   \\\hline
$|K|$   & 3   \\\hline
$\sg^2$ & 0.01\\\hline
$\tau_b$& 0.001$\es\forall\es b$\\\hline
$g_{bik},i\in I_b$ (channel gain, signal link) & Rayleigh$(1)$\\\hline
$g_{b'ik},i\in I_b,b'\in N(b)$ (channel gain, interference link) & Rayleigh$(0.5)$\\\hline
$P_b$   & 10$\es\forall\es b$\\\hline
$c_i$   & 1$\es\forall\es i$\\\hline
$\al_0$ & 0.99\\\hline
$\be$   & 0.53\\\hline
\end{tabular}
\end{table}
\end{comment}

We adopt the 19 cell wrap-around model from \cite{Huo_Wrap_2005} as the network scenario used in our simulations. Furthermore, each UE associates with exactly one BS and each BS has five UEs associated with it. Suppose a UE is served by a BS. Then there is a signal link between the UE and the BS, while all the neighbors of the BS cause interferences to the UE. The time horizon $T$ is chosen to be $20$. The channel gains are directly generated by Rayleigh distribution, with parameter $1$ for associated BS-UE pair, and $0.5$ for interference, instead of choosing random locations for the UEs and calculating the path loss. The channel gains are assumed to be independent among all links in a scheduling instance and also across all scheduling instances. We use identical QoS weights ($c_i=1$). Other parameters include: $|K|=3$, $\sg^2=0.01$, $\al_0=0.99$, $\be=0.53$, and $P_b=10\es\forall\es b$. For simplicity we treat all scheduling terms $x_{bik}$ as real numbers and use the local optimal results  to compute utilities. In future work we will include integer rounding procedures in the simulations. The entire process is simulated only one time, as multiple time slots already brought in the averaging effect.

\begin{figure}%[htbp]
\centering
\subfloat[$\eta=1$]{
\includegraphics[scale=0.45]{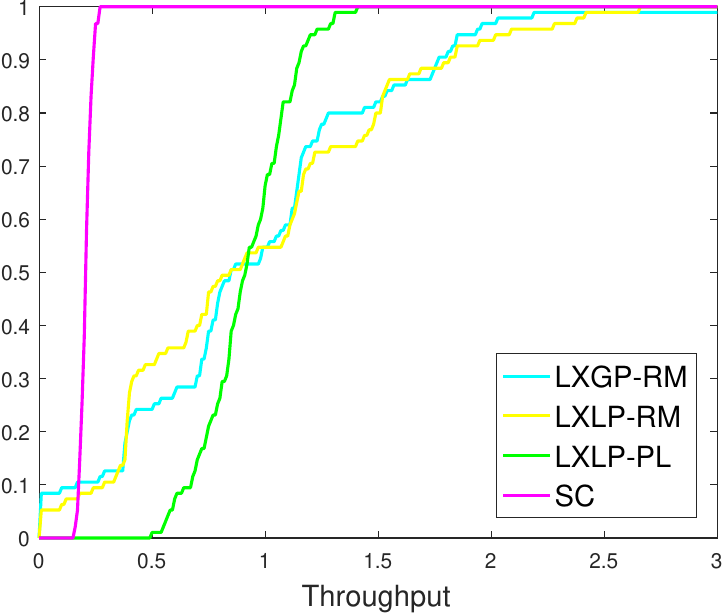}}
\label{fig2a}\hspace{20pt}
\subfloat[$\eta=0.5$]{
\includegraphics[scale=0.45]{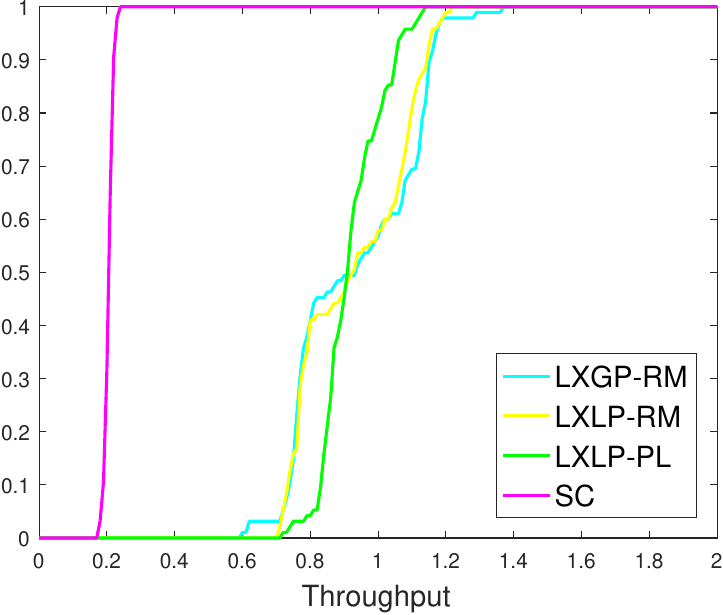}}
\label{fig2b}
\caption{Empirical CDFs of users' throughputs for $\eta=1$ and $\eta=0.5$.}
\label{fig2}
\end{figure}

\begin{figure}[t!]%[htbp]
\centering
\subfloat[Distribution of transmission power.]{
\includegraphics[scale=0.45]{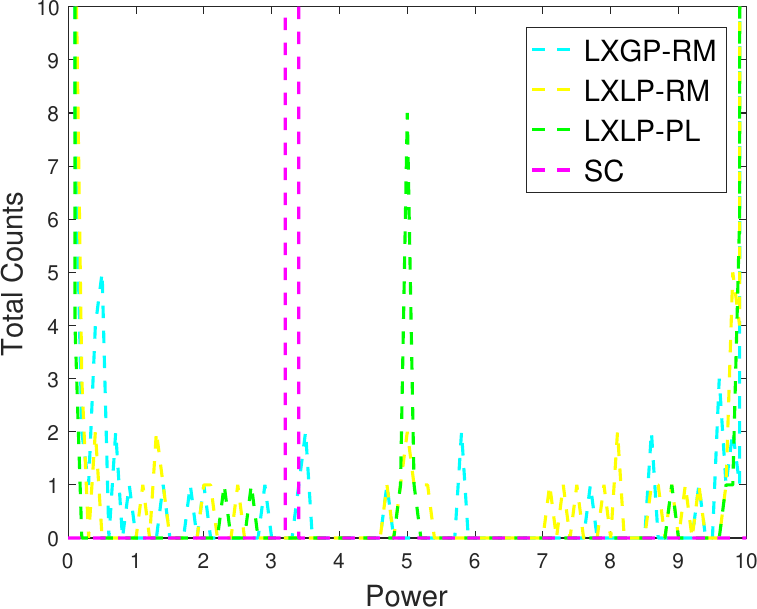}}
\label{fig3a}\hspace{20pt}
\subfloat[Distribution of SINR]{
\includegraphics[scale=0.45]{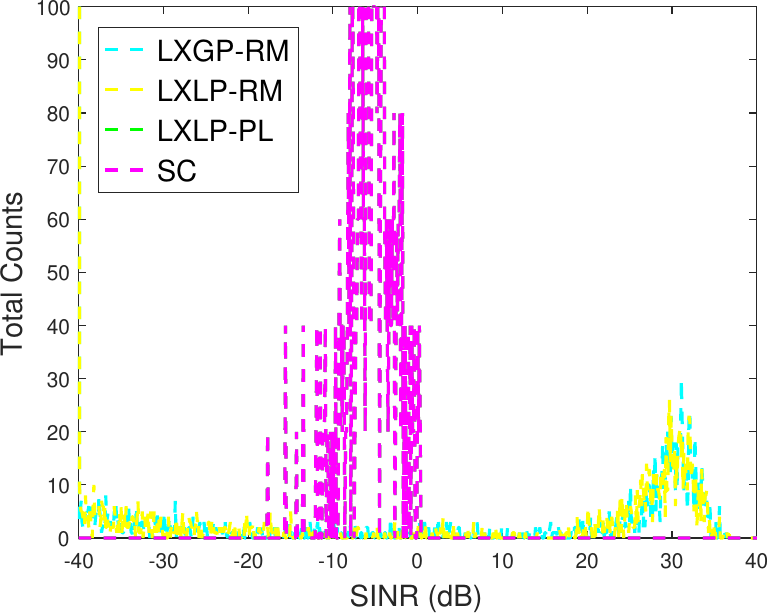}}
\label{fig3b}
\caption{Distributions of transmission power and SINR for $\eta=1$.}
\label{fig3}
\end{figure}

\cref{fig2} depicts the CDFs of user throughput of algorithms LXGP-RM, LXLP-RM, LXLP-PL, and SC for $\eta=1$ (maximum total throughput) and $\eta=0.5$, respectively. When $\eta=0.5$, we can observe that the RM and PL methods stochastically dominate the SC algorithm, and this is also nearly the case when $\eta=1$. In fact, the RM and PL methods yield a roughly 4-fold average throughput gain over SC method. This is not surprising, as we simulate a rich interference environment, and the new methods can coordinate the scheduled UEs and transmission powers of nearby BSs, while the SC method does not. The RM methods show similar performance as we use the same termination criteria. They are a little bit different from the PL method possibly due to optimizing different objective functions (only equivalent before integer relaxation).

\cref{fig3} illustrates the power and SINR distributions of the same four algorithms for $\eta=1$. The proposed three new methods choose one of three values for the power: the maximum, half of the maximum or zero. This corresponds to assigning either one, two or no blocks; the two blocks can be assigned to one UE or two. On the contrary, the SC methods schedule all three blocks and UEs with the power to each block around $\tfrac{10}{3}$. At each time instance the three new methods give up serving some subcarriers and users in exchange of boosting the SINR of the scheduled UEs on the chosen blocks. On the other hand, the SC method tries to serve everyone, and ends up with lower power and increased interference. The details of the scheduling decisions are in \cref{tab1}.

\begin{table}[!t]
\centering
\caption{Fraction of utilized channels and scheduled users per BS.}
\label{tab1}
\begin{tabular}{|l|l|l|l|l|l|l|}
\hline
Channels/Users & 0 & 1 & 2 & 3 & 4 & 5 \\\hline
LXLP-RM Channels & 0.3447 & 0.6316 & 0.0237 & 0 & -      & -      \\\hdashline
LXLP-RM Users    & 0.3447 & 0.5395 & 0.1158 & 0 & 0      & 0      \\\hline
LXLP-PL Channels & 0.3763 & 0.6105 & 0.0132 & 0 & -      & -      \\\hdashline
LXLP-PL Users    & 0.3763 & 0      & 0      & 0 & 0.0079 & 0.6158 \\\hline
SC Channels      & 0        & 0      & 0    & 1 & -      & -      \\\hdashline
SC Users         & 0      & 0      & 0      & 0 & 0      & 1      \\\hline
\end{tabular}
\end{table}

\begin{table}[!t]
\centering
%\caption{Total utilities and average numbers of iterations required of different algorithms and choices of $\eta$.}
\caption{Total utilities and average numbers of iterations required}
\label{tab2}
\begin{tabular}{|l|l|l|l|l|l|}
\hline
Algorithm & LXGP-RM & LXLP-RM & LXLP-PL & SC      \\ \hline\hline
\multicolumn{6}{|l|}{Utilities}\\ \hline
$\eta=1$   & 89.61 & 89.73 & 86.98 & 19.62 \\ \hline
$\eta=0.5$ & 182.7 & 182.5 & 182.1 & 86.10 \\ \hline\hline
\multicolumn{6}{|l|}{Iterations required}\\ \hline
$\eta=1$   & 280.3 & 176.9 & 176.2 & 2 \\ \hline
$\eta=0.5$ & 273.6 & 216.9 & 142.9 & 2 \\ \hline
\end{tabular}
\end{table}

\cref{tab2} compares the performance of the four algorithms with $\eta=0.5$ and $\eta=1$. The three coordination-based methods outperform the SC method significantly in this problem instance. The LXGP method always requires more iterations to converge than the fully-localized methods like the LXLP family. When $\eta=0.5$, LXLP-PL converges much faster than RM methods.%, due to partial linearization.

\section{Conclusion}\label{sec:conc}
In this paper, we generalized existing distributed non-convex optimization methods in two directions. First, we reduced the algorithm storage and communication complexity by exploiting a decomposable structure of the problem, and obtained a localized algorithm. Second, we relaxed the requirement of Lipschitz continuous gradients with a series of slowly-changing approximations. We then applied the developed algorithmic framework in different ways to generate distributed algorithms for the multi-cell resource allocation problem; the flexibility of implementing our framework is also a contribution. We compared these algorithms with the single-cell algorithm via simulation and showed the potential gains of using the distributed optimization methods developed.

%\section*{Acknowledgments}
%We would like to acknowledge the assistance of volunteers in putting together this example manuscript and supplement.

\bibliographystyle{siamplain}
\bibliography{references}
%\printbibliography

\newpage
\appendix
\section{Generalizations to time-varying graphs}\label{app:a}
In this appendix we provide the less strict assumptions to accommodate the case where the underlying graph is time-varying and directed. The proof of our result is based on this more general setting. But the purpose of these assumptions remains the same -- a distribution on the set of nodes will go to the uniform distribution exponentially fast with repeated application of the $W$ matrix, which is captured in \cref{app_c:l1}.

At each time slot $n$, the set of nodes $\Nc$ along with a set of time-variant directed edges $\Ec[n]$, form an directed graph $\Gc[n]=(\Nc,\Ec[n])$. Node $j$ can only send message to node $i$ in time slot $n$ if $j\ra i\in\Ec[n]$.\\
\noindent{\bf Assumption L\textprime}\\
{\bf (L3\textprime)} $\Gc_m[n]$ is $B_m$-strongly connected for all $m\in[M+1]$, where $\Gc_m[n]=(\Nc_m,\Ec_m[n]=\{i\ra j\in\Ec[n]:i\in\Nc_m\text{ or }j\in\Nc_m\})$, i.e. $(\Nc_m,\buU_{n=kB_m}^{(k+1)B_m-1}\Ec_m[n])$ is strongly connected for all $k\geq0$;\\
{\bf (L4\textprime)} For all $m\in[M+1]$ there is a matrix $\WB^m[n]$ associated with $\Nc_m$ -- each entry is non-zero if and only if there is a corresponding edge in $\Ec_m[n]$, and all positive entries must be greater than or equal to some fixed $\vartheta>0$. As before, $\WB^m[n]$ is doubly-stochastic after deleting the zero rows and columns whose indices are not in $\Ec_m[n]$.

\section{Proof of the Main Result}\label{app:b}
In this appendix we prove \cref{t7}. We start with an intuitive description of the roadmap of the entire proof in the following. Our proof follows the main structure of \cite{Scutari_NEXT_2016}. The convergence of \emph{NEXT} basically consists of two parts -- at the faster time scale the ``consensus convergence" of local variable iterates $\xB_i$ and local gradient iterates $\yB_i$ to the average iterate $\xBb$ and $\overline{\gd f}(\xBb)\triangleq\sum_{i\in\Nc}\gd f_i(\xBb)$ (the average gradient evaluated at the average iterate), respectively, and on the slower time scale the fixed point iterations of $\xBh_i(\bullet)$ (defined in \cref{51-2}) which we call ``path convergence" -- for details of this viewpoint and connection to two time-scale stochastic approximation see \cref{sec:sa}. It is shown in \cite{Scutari_NEXT_2016} that the fixed points of $\xBh_i(\bullet)$ coincide with the stationary solutions. The goal is thus to prove the iterates converge to the fixed points of $\xBh_i(\bullet)$ and the iterates of all nodes asymptotically agree. Our contributions are introducing partial dependency structure with a localization scheme and in addition successively approximating the possibly non-Lipschitz gradient objective functions. The latter significantly increases the proof hardness as the gradients of the objective functions may now be unbounded.

The core idea of \emph{NEXT} as in most primal algorithms, is in performing the following steps iteratively.
\begin{itemize}
\item The local optimization step for each node $i$ is to find the value of $\xBh_i$ at its iterate $\xB_i$. This $\xBh_i$ maps the iterate to the optimum of an approximated strongly convex version of the overall objective function $U$, which involves using a strongly convex surrogate for $i$'s own objective function, and a linearized approximation of others objective functions. This step is similar to doing gradient descent in a much more efficient way by taking advantage of the surrogates, and corresponds to the ``path convergence" mentioned above.
\item The consensus step is where every node communicates its iterate to its neighbors and also takes an average of its neighbors' iterates. We refer to the $\xB_i$'s asymptotically agreeing on their average $\xBb$ as the ``$x$-consensus convergence."
\end{itemize}
To make the algorithm more practical and fully decentralized, the original \emph{Inexact NEXT} \cite{Scutari_NEXT_2016} and our \emph{Localized Proximal Inexact NEXT} add multiple layers of approximations.
\begin{enumerate}[label=(\arabic*)]
\item We mentioned that the node $i$ linearize other nodes' gradients with $\pi_i\triangleq\sum_{j\in\Nc\sm\{i\}}\gd f_j(\xB_i)$. However, node $i$ only has the information of $\gd f_i(\xB_i)$. It hence keeps a variable $\yB_i$ that tracks the average gradient $\frac{1}{I}\sum_{j\in\Nc}\gd f_j$ so that node $i$ can approximate $\pi_i$ by $\piBt_i$ with the knowledge of $\gd f_i(\xB_i)$ and $\yB_i$ as in Algorithm \ref{a6}. The convergence of $\yB_i$ to the average gradient is then the ``$y$-consensus convergence." The convergence is also achieved through a gossip-type consensus scheme similar to that used for $\xB$.
\item With the approximation scheme using the $y$ variable, the fixed point iteration we actually use in Algorithm \ref{a6} for the local optimization step is $\xBt_i$ defined in \eqref{8-2}, which is similar to $\xBh_i$ but using $\piBt_i$ as the linearization constant. This makes the algorithm viable in practice in a fully decentralized scenario. To compare the behaviors of $\xBh_i(\xBb)$ and $\xBt_i=\xBt_i^*(\xB_i,\piBt_i)$, we construct a new ``averaging system" assuming that the $\xB_i$'s already converge to $\xBb$; this includes $\yB_i^{av}$, the average gradient evaluated at $\xBb$, and $\xBt_i^{av}=\xBt^*(\xBb,\piBt_i^{av})$, the local optimization map where $\piBt_i^{av}$ computed from $\yB_i^{av}$ is used as the linearization constant. Note that this ``averaging system" is constructed purely for analysis purposes.
\item We use a series of functions $\{\ft_{i,n}^*\}$ to approximate $f_i$ so that the local optimization map with ideal linearization is $\xBh_{i,n}$, and $\xBt_i^*$ with the $\yB$-approximation in contrast to just $\xBt_i$ in \emph{NEXT}. This is the approximation we add in addition to what was done in \emph{Inexact NEXT} \cite{Scutari_NEXT_2016}.
\item The inexactness of the algorithm chooses $\xB_i^{inx}$ within $\ep_i$ range of $\xBt_i$, which leads to another source of approximation. Because of the function approximation we use with the relaxed constraints on the schedules of $\{L_{i,n}\}$ and $\{\tau_{i,n}\}$, the difference between $\xB_i^{inx}$ and $\xB_i$ can potentially become increasingly larger if the error propagates, which also increases the proof hardness in contrast to the case of \emph{Inexact NEXT} where the difference $\|\xB_i^{inx}-\xBt_i\|$ is bounded.
\end{enumerate}

The proof of \cref{t7} consists of six parts. The theorem basically makes two claims: the nodes' iterates asymptotically agree, and they converge to one of the optima. The former will be the side product as we aim to prove the latter. In the first part of the proof, we first summarize the list of notations that will be used in the proof in \cref{App:B-1-1}, and then describe a few results and one key proposition in \cref{App:B-1-2}. \cref{p19} as the variant of Proposition 5 in \cite{Scutari_NEXT_2016} shows Lipschitz properties of $\xBh_{i,n}$. \cref{app_c:l1} and \cref{f24} describe the main machinery we use to show the ``$x$-consensus convergence" and ``$y$-consensus convergence," that is, the \emph{geometric convergence} of the product of doubly stochastic matrices to the all one matrix. The results \cref{l21}, \cref{l22}, \cref{l24}, and Technical Assumption T are technical lemmas regarding series or summations that arise in the analysis. The key proposition is \cref{p23}, which constitutes the core components of the proof.

We prove \cref{p23} (a) in the second part \cref{App:B-2}, which says the difference between $\xB_i^{inx}$ and $\xB_i$ cannot grow beyond a certain rate, and the tools used are the definition of minimization \eqref{8-2}, the strong convexity of $\ft_{i,n}^*$, and the $y$-consensus convergence. The difference is decomposed into $\|\xB_i^{inx}-\xBt_i\|$ and $\|\xBt_i-\xB_i\|$, where the former is bounded by $\ep_i[n]$ by definition, and the latter is bounded by $O\left(\frac{L_{i,n}}{\tau_{i,n}}\right)$ using a mathematical induction argument.

Part (b) of \cref{p23} establishes asymptotic consensus on $\xB$ among nodes as the third part of the proof \cref{App:B-3}. This part involves exploiting \cref{app_c:l1}, \cref{l21}, \cref{f24}, Part (a) of \cref{p23}, and series convergence. The generalization to multiple local dependency sets is also taken care of in Part (a) and (b) of \cref{p23} using simple inequalities regarding multi-dimensional spaces.

In the fourth part contained in \cref{App:B-4}, Part (c) of \cref{p23} is proved, which shows that the locally-optimized result using the ``$y$ approximation" $\xBt_i^{av}$ converges to the locally-optimized result with ideal linearization $\xBh_{i,n}(\xBb)$ evaluated at $\xBb$. In addition to applying the definitions of the maps, it is intuitive that the ``$y$-consensus convergence" in the ``averaging system" and hence \cref{app_c:l1} play a crucial role in the proof; Part (a) of \cref{p23} and Technical Assumption T which comes from \cref{l24} and the conditions of \cref{t7} are also used.

We prove Part (d) of \cref{p23} in the fifth part \cref{App:B-5}. Part (d) claims the actual locally-optimized result in the algorithm $\xBt_i$ converges to the locally-optimized result using the ``$y$ approximation" $\xBt_i^{av}$. The underlying reason of the convergence is the ``$x$-consensus convergence." Several previous results are all used in this proof, including all of Parts (a), (b), and (c), \cref{l21}, \cref{l24}, and Technical Assumption T.

\cref{af} then combines Parts (a), (c), and (d), convexity of $G$, and \cref{l22} to show that $\xBb$ converges to $\xBh_{i,\infty}(\xBb)$. With $\tau_n$ going to $0$, $\xBh_{i,\infty}$ is no longer a function but a correspondence; variational analysis is thus introduced to deal with the minimizers of correspondences in the first case of \cref{af}. Finally, we deal with unbounded gradient interior point in the second case of \cref{af}, using convexity and series convergence. Generalization to study the case of an unbounded gradient boundary point is left as an open question.

Comparing to the proof of \emph{NEXT}, the generalization to multiple local dependency set is not a technically hard one; it requires some simple inequalities as shown in part (a) and (b) of \cref{p23}. For the second generalization, we replace what was Lipschitz constant $L$ and strongly convex constant $\tau$ by series $\{L_n\}$ and $\{\tau_n\}$, which now could grow to infinity and decrease to zero, respectively. This does significantly increase the hardness of the proof. The conditions in \cref{t7}, the Technical Assumption T stated below, and \cref{l24} are made such that all the series now with $\{L_n\}$ and $\{\tau_n\}$ still converge. The unbounded gradient issue and the correspondence nature of $\xBh_{i,\infty}$ also make our scheme much trickier to analyze in comparison to \emph{NEXT}.

\subsection{Notations}\label{App:B-1-1}
We define a set of notations to proceed with the proof. All notation is defined for all $i$, $m$, and $n$, whenever applicable.\\
\textbf{Original variables}
\begin{itemize}
\item
$\xB^m[n]=(\I\{i\in\Nm\}\xB^m_i[n])_{i\in\Nc}=[\I\{1\in\Nc_m\}\xB^m_1[n]^T\es\cdots\es\I\{I\in\Nc_m\}\xB^m_I[n]^T]^T$: the concatenation of part $m$ decision variables from all nodes in $\Nm$ with padded zero for nodes not in $\Nm$; we also use $\xB^m[n]$ to refer to non-padded zero version $(\xB^m_i[n])_{i\in\Nm}$, i.e. the vector containing only $\xB^m_i[n]$ when $i$ is in $\Nm$, when the context is clear; the notation $(v_i)_{i\in\Sc}$, which denotes the vector concatenated from all the vectors of the form $v_i$ where the index $i$ is in the set $\Sc$, will be used throughout this Appendix
\item
$\yB^m[n]=(\I\{i\in\Nm\}\yB^m_i[n])_{i\in\Nc}$: the concatenation of part $m$ of $\yB$, which tracks the average (among nodes) gradients of $\gd_{\xB^m}f_i$ from the nodes in $\Nm$ in the algorithm
\item
$\rB^m[n]=(\gd_{\xB^m}f^*_{i,n}[n])_{i\in\Nc}$: the concatenation of ground truth gradient, with $\gd_{\xB^m}f^*_{i,n}[n]=\gd_{\xB^m}f^*_{i,n}(\xB_i[n])$; adding $\I\{i\in\Nm\}$ is unnecessary as the gradient would be zero for those nodes not depending on $\xB^m$
\item
$\Dl\rB^m[l,n]=(\Dl\rB^m_i[l,n])_{i\in\Nc}$: the gradient difference, with $\Dl\rB^m_i[l,n]=\gd_{\xB^m}f^*_{i,l}[l]-\gd_{\xB^m}f^*_{i,n}[n]$ ($l\leq n$)
\item
$\piBt_i[n]$: see Line 11 of \cref{a6}
\item
$\xBt_i[n]=\xBt^*_i(\xB_i[n],\piBt_i[n])$: see Line 5 of \cref{a6} and \cref{8-2}
\end{itemize}
\textbf{Average variables}
\begin{itemize}
\item
$\xBb^m[n]=\frac{1}{I_m}\sum_{i\in\Nc_m}\xB^m_i[n]$: average of decision variable
\item
$\yBb^m[n]=\frac{1}{I_m}\sum_{i\in\Nc_m}\yB^1_i[n]$: average of gradient tracking variable
\item
$\rBb^m[n]=\frac{1}{I_m}\sum_{i\in\Nc_m}\gd_{\xB^m}f^*_{i,n}[n]$: average of ground truth gradient
\item
$\Dl\rBb^m[l,n]=\frac{1}{I_m}\sum_{i\in\Nc_m}\Dl\rB^m_i[l,n]$: average of gradient difference
\end{itemize}
\textbf{Tracking system using average variables}
\begin{itemize}
\item
$\gd f^{*,av}_{i,n}[n]=\gd f^*_{i,n}(\xBb[n])$
\item
$\rB^{m,av}[n]=(\gd_{\xB^m}f^{*,av}_{i,n}[n])_{i\in\Nc}$: the concatenation of ground truth gradient evaluated at average decision variable
\item
$\Dl\rB^{m,av}[l,n]=(\Dl\rB^{m,av}_i[l,n])_{i\in\Nc}$: the gradient difference evaluated at average decision variable, with $\Dl\rB^{m,av}_i[l,n]=\gd_{\xB^m}f^{*,av}_{i,n}[l]-\gd_{\xB^m}f^{*,av}_{i,n}[n]$
\item
$\yB^{m,av}_i[n+1]=\sum_jw^m_{ij}[n]\yB^{m,av}_j[n]+\Dl\rB^{m,av}_i[n+1,n]$: tracking of average gradient evaluated at average decision variable, with $\yB^{m,av}_i[0]=\gd_{\xB^m}f^{*,av}_{i,n}[0]$; concatenating $\yB^{m,av}_i[n+1]$ for $i\in\Nc$ makes $\yB^{m,av}[n]=(\I\{i\in\Nm\}\yB^{m,av}_i[n])_{i\in\Nc}$
\item
$\piBt^{m,av}_i[n]=I_m\yB^{m,av}_i[n]-\gd_{\xB^m}f^{*,av}_{i,n}[n]$
\item
$\xBt^{av}_i[n]=\xBt^*_i(\xBb_i[n],\piBt^{av}_i[n])$: optimized result evaluated at average decision variable and average tracking system
\item
$\rBb^{m,av}[n]=\frac{1}{I_m}\sum_{i\in\Nc_m}\gd_{\xB^m}f^{*,av}_{i,n}[n]$: average of ground truth gradient evaluated at average decision variable
\end{itemize}
\textbf{Doubly stochastic matrices}
\begin{itemize}
\item
$\PB^m[n,l]=\WB^m[n]\WB^m[n-1]\cdots\WB^m[l]\quad n\geq l$
\item
$\WBh^m[n]=\WB^m[n]\otimes I_{d_m}$
\item
$\PBh^m[n,l]=\WBh^m[n]\WBh^m[n-1]\cdots\WBh^m[l]=\PB^m[n,l]\otimes I_{d_m}\quad n\geq l$
\item
$J^m=\frac{1}{I_m}\1_{\Nc_m}\1_{\Nc_m}^T\otimes\IB_{d_m}$, where $\1_{\Nc_m}=\{\I\{i\in\Nm\}:i\in\Nc\}$, and $\IB$ is the identity matrix
\item
$J^m_\perp=\IB_{d_mI_m}-J^m$, where $\IB_{d_mI_m}$ is the identity matrix with dimension $d_m\times I_m$
\end{itemize}

\subsection{Key Propositions}\label{App:B-1-2}
The next proposition is a variant of Proposition 5 in \cite{Scutari_NEXT_2016}.

\begin{proposition}\label{p19}
Let $\piB^\Si_i(\xBt)=\sum_{j\neq i}\gd_{\xB^\Si}f^*_{j,n}(\xBt)=\left(\sum_{j\in\Nm,j\neq i}\gd_{\xB^m}f^*_{j,n}(\xBt^\Nm)\right)_{m\in\Si}$ be the concatenation of $\sum_{j\in\Nm,j\neq i}\gd_{\xB^m}f^*_{j,n}(\xBt^\Nm)$ for all $m$ in $\Sc_i$. Define the mapping $\xBh^\Si_{i,n}(\cdot):\Kc\ra\Kc_\Si=\Pi_{m\in\Si}\Kc_m$ by
\beq\label{28}
\xBh^\Si_{i,n}(\xBt)=\underset{\xB^\Si}{\arg\min}\es\tilde{f}^*_{i,n}(\xB^\Si;\xBt^\Si)+\piB^\Si_i(\xBt)^T(\xB^\Si-\xBt^\Si)+G(\xB^c)\quad\forall\es i\in\Nc,
\eeq
and the mapping $\xBh_{i,n}(\cdot):\Kc\ra\Kc$ by $\xBh_{i,n}(\xBt)=\left(\xBh^\Si_{i,n}(\xBt),\xBt^{\Nc\sm\Si}\right)$, that is, preserving everything in the $\Kc_{\Nc\sm\Si}$ subspace unchanged while mapping with $\xBh^\Si_{i,n}(\cdot)$ in the $\Si$ subspace. Then, under Assumptions A, F, and N, the mapping $\xBh_{i,n}(\cdot)$ has the following properties:
\begin{enumerate}[label=(\alph*)]
\item
$\forall\es\zB\in\Kc$ and $i\in\Nc$,
\[
(\xBh_{i,n}(\zB)-\zB)^T\gd F(\zB)+G(\xBh_{i,n}(\zB))-G(\zB)\leq-\taumin_n\|\xBh_{i,n}(\zB)-\zB\|^2,
\]
where $F(\xB)=\sum_{i=1}^If_i(\xB)$. Here we use $G(\xB)$ and $G(\xB^c)$ interchangeably as they are the same thing.
\item
$\xBh_{i,n}(\cdot)$ is Lipschitz continuous, i.e. $\|\xBh_{i,n}(\wB)-\xBh_{i,n}(\zB)\|\leq L_{i,n}\|\wB-\zB\|\quad\forall\es\wB,\zB\in\Kc$ for $i\in\Nc$\footnote{Note that this holds for $\xBh^\Si_{i,n}$ as well because the elements in the $\Kc_{\Nc\sm\Si}$ subspace just cancel each other out.}.
%\item
%The set of fixed points of $\xBh_i(\cdot)$ coincides with the set of stationary solutions of the original problem.
\end{enumerate}
\end{proposition}
The only thing we do is to substitute $\ft_i$ in \cite{Scutari_NEXT_2016} as $\ft^*_{i,n}$. Although the equations are written in localization form, it does not really change anything here.

\begin{lemma}\label{app_c:l1}
Define $\PB[n,l]\triangleq\WB[n]\WB[n-1]\cdots\WB[l]$. Then under Assumption L\textprime (doubly stochasticity and lower bounded entries for edges),
\[
\left\|\PB[n,l]-\frac{1}{I}\1\1^T\right\|\leq c_0\rho^{n-l+1},\es\forall\es n\geq l
\]
for some $c_0>0$ and $\rho\in(0,1)$.
\end{lemma}
Strictly speaking the above is for the $\WB^c=\WB^{M+1}$, which is the matrix used for the averaging of the entire network, in Assumption L\textprime. For $\WB^m$ where $m\in[M]$, the Lemma also holds after we delete the zero rows/columns. In this case the product converges to $\frac{1}{I_m}\1\1^T$ where the $\1$ is of the proper dimension. From now on we will take $\rho$ as the largest geometric convergence factor among all $\WB^m$, $m\in[M+1]$.

Before proving \cref{t7}, we will first prove the following proposition.

\begin{proposition}\label{p23}
Let $\{\xB^m[n]\}_n\triangleq\{(\xB^m_i[n])_{i\in\Nc_m}\}$ and $\{\xBb^m[n]\}_n\triangleq\left\{\frac{1}{I_m}\sum_{i\in\Nc_m}\xB^m_i[n]\right\}_n$, $m\in[M+1]$ be the sequences generated by \cref{a6}, in the settings of the \cref{t7}. Then the following holds:
\begin{enumerate}[label=(\alph*)]
\item
For all $n$, $\|\xB^{m,inx}_i[n]-\xB^m_i[n]\|\leq\frac{c^mL_{i,n}}{\tau_{i,n}}\quad\forall\es i\in\Nc_m,m\in[M+1]$.
\item
$\lim_{n\ra\infty}\|\xB^m_i[n]-\xBb^m[n]\|=0$, $\sum_{n=1}^\infty\al[n]\|\xB^m_i[n]-\xBb^m[n]\|<\infty$, $\sum_{n=1}^\infty\|\xB^m_i[n]-\xBb^m[n]\|^2<\infty\quad\forall\es i\in\Nc_m,m\in[M+1]$.
\item
$\lim_{n\ra\infty}\|\xBt^{av}_i[n]-\xBh^\Si_{i,n}(\xBb[n])\|=0$, $\sum_{n=1}^\infty\al[n]\Lmax_n\|\xBt^{av}_i[n]-\xBh^\Si_{i,n}(\xBb[n])\|$ $<\infty\quad\forall\es i\in\Nc$.
\item
$\lim_{n\ra\infty}\|\xBt^m_i[n]-\xBt^{m,av}_i[n]\|=0$, $\sum_{n=1}^\infty\al[n]\Lmax_n\|\xBt^m_i[n]-\xBt^{m,av}_i[n]\|<\infty\quad\forall\es i\in\Nc_m,m\in[M+1]$.
\end{enumerate}
\end{proposition}

We will use the following assumption frequently when proving \cref{p23}. These are the exact technical inequalities we use in the proof, while all of them are implicitly implied by the conditions of \cref{t7} as we will show below.\\
\noindent{\bf Technical Assumption T}\\
{\bf (T1)} $\lim_{n\ra\infty}\rho^n\frac{\Lmax_n}{\taumin_n}=0$;\\
{\bf (T2)} $\lim_{n\ra\infty}\frac{\Lmax_n}{\taumin_n}\sum_{l=0}^{n-1}\rho^{n-l}\frac{\al[l](\Lmax_l)^2}{\taumin_l}=0$;\\
{\bf (T3)} $\lim_{n\ra\infty}\frac{1}{\taumin_n}\sum_{l=0}^{n-1}\rho^{n-l}\|\gd f^*_{i,l}(\xB)-\gd f^*_{i,l-1}(\xB)\|=0$ for all $\xB$ and $i$;\\
{\bf (T4)} $\sum_{n=1}^\infty\rho^n\frac{\Lmax_n}{\taumin_n}<\infty$;\\
{\bf (T5)} $\sum_{n=1}^\infty\frac{\al[n]\Lmax_n}{\taumin_n}\sum_{l=0}^{n-1}\rho^{n-l}\frac{\al[l](\Lmax_l)^2}{\taumin_l}<\infty$;\\
{\bf (T6)} $\sum_{n=1}^\infty\frac{\al[n]\Lmax_n}{\taumin_n}\sum_{l=0}^{n-1}\rho^{n-l}\|\gd f^*_{i,l}(\xB)-\gd f^*_{i,l-1}(\xB)\|<\infty$ for all $\xB$ and $i$.
\begin{proof}\hfill\\
{\bf (T1)}: it should be evident from the conditions of \cref{t7} that none of the parameters could be growing or decaying at an exponential rate. We are considering the setting where $\al[n]$ and $\taumin_n$ are going to zero while $\Lmax_n$ is going to infinity. The condition $\sum_{n=0}^\infty(\Lmax_n)^3\left(\frac{\al[n]}{\taumin_n}\right)^2<\infty$ implies that if either $\Lmax_n$ is growing exponentially or $\taumin_n$ is decaying exponentially, then $\al[n]$ must also be decaying exponentially. But then $\sum_{n=0}^\infty\taumin_n\al[n]=\infty$ would never be possible.\\
{\bf (T2)}: recall the conditions of \cref{t7} imply $\lim_{n\ra\infty}\al[n]\frac{(\Lmax_n)^3}{(\taumin_n)^3}=0$, then apply the first part of \cref{l24}.\\
{\bf (T3)}: from $\lim_{n\ra\infty}\frac{\etamax_n}{\taumin_n}=0$ and the first part of \cref{l24}.\\
{\bf (T4)}: again the parameters are not growing at an exponential rate.\\
{\bf (T5)}: from $\sum_{n=0}^\infty(\Lmax_n)^3\left(\frac{\al[n]}{\taumin_n}\right)^2<\infty$ and the second part of \cref{l24}.\\
{\bf (T6)}: from $\sum_{n=0}^\infty\frac{\al[n]\Lmax_n\etamax_n}{\taumin_n}<\infty$ and the second part of \cref{l24}.
\end{proof}

\subsection{Proof of \cref{p23} (a)}\label{App:B-2}
Consider a local dependency set $\Nm$ and any node $i\in\Nm$. By the definition of $\xB^\Si_i$ defined in the minimization of \cref{8-2}, we have
\beq\label{29}
\sum_{m\in\Si}(\xB^m_i[n]-\xBt^m_i[n])^T\left[\gd_{\xB^m_i}\ft^*_{i,n}(\xBt^\Si_i[n];\xB^\Si_i[n])+\piBt^m_i[n]\right]+(\xB^c_i[n]-\xBt^c_i[n])^T\pa G(\xBt^c_i[n])\geq0.
\eeq
From the Line 11 of \cref{a6} and (F2\textprime), we have
\beq
\piBt^m_i[n]=I_m\cdot\yB^m_i[n]-\gd_{\xB^m_i}\ft^*_{i,n}(\xB^\Si_i[n];\xB^\Si_i[n]).
\eeq
Substitute this result into \cref{29} and rearrange the terms to get
\bal\label{30}
&\quad\tau_{i,n}\|\xB^\Si_i-\xBt^\Si_i\|^2=\tau_{i,n}\sum_{m\in\Si}\|\xB^m_i-\xBt^m_i\|^2&\\
&\leq(\xB^\Si_i-\xBt^\Si_i)^T\cdot[\gd_{\xB^\Si_i}\ft^*_{i,n}(\xB^\Si_i;\xB^\Si_i)-(\gd_{\xB^\Si_i}\ft^*_{i,n}(\xBt^\Si_i;\xB^\Si_i)]&(\text{strong convexity of }\ft^*_{i,n})\\
&\leq\sum_{m\in\Si}(\xB^m_i-\xBt^m_i)^T(I_m\yB^m_i)+(\xB^c_i-\xBt^c_i)^T\pa G(\xBt^c_i)]&(\text{from }\cref{29})\\
&\leq\sum_{m\in\Si}(I_m\|\yB^m_i\|)\cdot\|\xB^m_i-\xBt^m_i\|+L_G\|\xB^c_i-\xBt^c_i\|&(\text{C-S inequality+(A2)}).
\eal
We have omitted all the time indices in \cref{30} since all the variables have the same time index $[n]$.

Suppose that $\|\yB^m_i\|$ is bounded by $l_mL_{i,n}$ for all $m\in\Si$. Then \cref{30} is of the form
\beq\label{31}
\tau_{i,n}\sum_{m\in\Si}\|\xB^m_i-\xBt^m_i\|^2\leq\sum_{m\in\Si}l_mL_{i,n}\|\xB^m_i-\xBt^m_i\|,
\eeq
which implies that all $\|\xB^m_i-\xBt^m_i\|$'s are bounded by $\frac{\sum_{m\in\Si}l_mL_{i,n}}{\tau_{i,n}}$\footnote{Note that $\xB^c$ refers to $\xB^{M+1}$ and $M+1$ is in $\Si$ as well if the part exists. Therefore, the second term $L_G\|\xB^c_i-\xBt^c_i\|$ can be put into the summation in the first term.}. This is due to the following argument: if $\{x_i\},\{l_i\}$ are non-negative and $\sum_ix^2_i\leq\sum_il_ix_i$, then $\max\{x_i\}\leq\sum_il_i$; otherwise, W.O.L.G. we can assume $x_1=\max\{x_i\}$ and hence $\sum_il_i<x_1$, then the following holds
\[
\sum_ix^2_i>x^2_1>x_1\sum_il_i>\sum_il_ix_i,
\]
which is a contradiction. Thus, with \cref{31}, we get
\[
\|\xB^{m,inx}_i[n]-\xB^m_i[n]\|\leq\|\xB^{m,inx}_i[n]-\xBt^m_i[n]\|+\|\xBt^m_i[n]-\xB^m_i[n]\|\leq\ep^m_i[n]+\frac{\sum_{m\in\Si}l_mL_{i,n}}{\tau_{i,n}}\leq\frac{c^mL_{i,n}}{\tau_{i,n}},
\]
where $c^m$ is some constant independent of $n$ and $i$. This proves the claim. It only remains to show that $\|\yB^m_i\|$ is actually bounded by $l_mL_{i,n}$.

We use mathematical induction to finish the proof. The statement is that
\beq
\|\Dl\xB^{m,inx}_i[n]\|=\|\xB^{m,inx}_i[n]-\xB^m_i[n]\|\leq\frac{c^mL_{i,n}}{\tau_{i,n}},\es\|\yB^m_i[n]\|\leq l_mL_{i,n}
\eeq
holds for all $n$. We have already shown that the latter implies the former. The base case is obvious as we initialize $\yB^m_i[0]$ to be $\gd_{\xB^m}f^*_{i,0}[0]$, which is assumed to be Lipschitz continuous. For the induction step, we assume the statement is true for $n-1$ and proved the latter part $\es\|\yB^m_i[n]\|\leq l_mL_{i,n}$ holds for $n$.

By the definition of $\yB$, we have
\beq
\yB^m_i[n]=\WBh^m[n-1]\yB^m_i[n-1]+\Dl\rB^m_i[n,n-1]
\eeq
where
\bal\label{31-2}
&\quad\|\Dl\rB^m_i[n,n-1]\|\\
&=\|\gd_{\xB^m}f^*_{i,n}[n]-\gd_{\xB^m}f^*_{i,n-1}[n-1]\|
=\|\gd f^*_{i,n}(\xB_i[n])-\gd_{\xB^m}f^*_{i,n-1}(\xB_i[n-1])\|\\
&\leq\|\gd_{\xB^m}f^*_{i,n}(\xB_i[n])-\gd_{\xB^m}f^*_{i,n}(\xB_i[n-1])+\gd_{\xB^m}f^*_{i,n}(\xB_i[n-1])-\gd_{\xB^m}f^*_{i,n-1}(\xB_i[n-1])\|\\
&\leq L_{i,n}\|\xB_i[n]-\xB_i[n-1]\|+\|\gd_{\xB^m}f^*_{i,n}(\xB_i[n-1])-\gd_{\xB^m}f^*_{i,n-1}(\xB_i[n-1])\|.
\eal
To reach the last line we utilize the triangle inequality and the Lipschitz continuity of $\gd f^*_{i,n}$. For the first term,
\bal\label{31-3}
\|\xB^m_i[n]-\xB^m_i[n-1]\|&\leq\|\xB^m[n]-\xB^m[n-1]\|
\leq\left\|\frac{\1_\Nm^T\otimes\IB_{d_m}}{I_m}(\xB^m[n]-\xB^m[n-1])\right\|\\
&=\|\xBb^m[n]-\xBb^m[n-1]\|\qquad(\text{\cref{f24} (c)})\\
&=\al[n-1]\left\|\frac{\1_\Nm^T\otimes\IB_{d_m}}{I_m}\Dl\xB^{m,inx}[n-1]\right\|\qquad(\text{\cref{f24} (e)})\\
&\leq c_1\al[n-1]\|\Dl\xB^{m,inx}[n-1]\|.
\eal
Using the induction hypothesis of $\Dl\xB$, we obtain
\bals
&\quad\|\Dl\rB^m_i[n,n-1]\|
\leq L_{i,n}\|\xB_i[n]-\xB_i[n-1]\|+c_2\eta_{i,n}\\
&\leq c_1\al[n-1]L_{i,n}\|\Dl\xB^{m,inx}[n-1]\|+c_2\eta_{i,n}
\leq\frac{c_3\al[n-1]L_{i,n}L_{i,n-1}}{\tau_{i,n-1}}+c_2\eta_{i,n}.
\eals
Therefore, we finally obtain
\bals
\|\yB^m_i[n]\|
&\leq\left\|\WBh^m[n-1]\yB^m_i[n-1]\right\|+\|\Dl\rB^m_i[n,n-1]\|
\leq c_4l_mL_{i,n-1}+\frac{c_3\al[n-1]L_{i,n}L_{i,n-1}}{\tau_{i,n-1}}+c_2\eta_{i,n}\\
&\leq L_{i,n}\left(c_4l_m+\frac{c_3\al[n-1]L_{i,n-1}}{\tau_{i,n-1}}+\frac{c_2\eta_{i,n}}{L_{i,n}}\right)
\leq c_5L_{i,n}
\eals
using the induction hypothesis of $\Dl\yB$ and the fact that $\frac{\al[n-1]L_{i,n-1}}{\tau_{i,n-1}}$ also goes to zero when $n\ra\infty$ implied by the condition of \cref{t7}.

\subsection{Proof of \cref{p23} (b)}\label{App:B-3}
We only prove the case for $\xB^c=\xB^{M+1}$ to save the ubiquitous subscript of $m$. The proof of the claims for general $\xB^m$ is exactly the same with appropriate substitutions of $\xB^c,\WBh,\PB,\rB^c,J,J_\perp,\1_I,I$ by $\xB^m,\WBh^m,\PB^m,\rB^m,J^m,J^m_\perp,\1_{\Nc_m},I_m$.

\begin{enumerate}[label=(\roman*),wide]
\item
\[
\xB^c[n]-\1_I\otimes\xBb^c[n]=\xB^c[n]-J\xB^c[n]=J_\perp\xB^c[n]\triangleq\xB^c_\perp[n].
\]
Notice that with \cref{f24} (d) and (e), the difference of $\xB^c[n]$ and $\1_I\otimes\xBb^c[n]$ which is $\xB^c_\perp[n]$ can be expressed as a linear combination of $\xB^c_\perp[n-1]$ and $\Dl\xB^{c,inx}[n-1]$. We can thus expand $\xB^c_\perp[n-1]$ iteratively as follows:
\bal\label{35}
\xB^c_\perp[n]
&=J_\perp\WBh[n-1]\xB^c_\perp[n-1]+\al[n-1]J_\perp\WBh[n-1]\Dl\xB^{c,inx}[n-1]\\
&=J_\perp\WBh[n-1](J_\perp\WBh[n-2]\xB^c_\perp[n-2]+\al[n-2]J_\perp\WBh[n-2]\Dl\xB^{c,inx}[n-2])\\
&\qquad\qquad+\al[n-1]J_\perp\WBh[n-1]\Dl\xB^{c,inx}[n-1]\\[-10pt]
&\es\vdots\\[-10pt]
&=J_\perp\WBh[n-1]J_\perp\WBh[n-2]\cdots J_\perp\WBh[0]\xB^c_\perp[0]
+\sum_{l=0}^{n-1}J_\perp\WBh[n-1]\cdots J_\perp\WBh[l]\al[l]\Dl\xB^{c,inx}[l]\\
&=\left[\left(\PB[n-1,0]-\frac{1}{I}\1_I\1_I^T\right)\otimes I_m\right]\xB^c_\perp[0]+\sum_{l=0}^{n-1}\left[\left(\PB[n-1,l]-\frac{1}{I}\1_I\1_I^T\right)\otimes I_m\right]\al[l]\Dl\xB^{c,inx}[l]
\eal
where the last equation resulted from \cref{f24} (b). From \cref{p23} (a) we know
\beq\label{36}
\|\Dl\xB^{c,inx}[n]\|\leq\|\Dl\xB^{inx}[n]\|=\sqrt{\sum_{i=1}^I\|\Dl\xB^{inx}_i[n]\|^2}\leq c_1\max_{i}\|\Dl\xB^{inx}_i[n]\|\leq\frac{c_2\Lmax_n}{\taumin_n}
\eeq
for some constants $c_1$ and $c_2$. Consequently, we get
\beq\label{37}
\|\xB^c_\perp[n]\|\leq c_3\rho^n+c_4\sum_{l=0}^{n-1}\rho^{n-l}\frac{\al[l]\Lmax_l}{\taumin_l}\xrightarrow{n\ra\infty}0
\eeq
by first utilizing triangle inequality, and then using \cref{36}, \cref{app_c:l1}, and finally \cref{l21} (a). Remark that $\lim_{n\ra\infty}\al[n]\left(\frac{\Lmax_n}{\taumin_n}\right)^3=0$ implies $\lim_{n\ra\infty}\frac{\al[n]\Lmax_n}{\taumin_n}=0$, which we use in \cref{37} as the condition of \cref{l21} (a).
\item
\bals
%\hspace*{-15pt}
&\lim_{n\ra\infty}\sum_{k=1}^n\al[k]\|\xB^c_i[k]-\xBb^c[k]\|
\leq\lim_{n\ra\infty}\sum_{k=1}^n\al[k]\|\xB^c_\perp[k]\|&\\
\leq&\lim_{n\ra\infty}\sum_{k=1}^n\al[k]\left(c_3\rho^k+c_4\sum_{l=0}^{k-1}\rho^{k-l}\frac{\al[l]\Lmax_l}{\taumin_l}\right)&(\text{from \cref{37}})\\
\leq&\lim_{n\ra\infty}\left(c_3\sum_{k=1}^n\rho^k\al[k]+c_4\rho\sum_{k=1}^n\sum_{l=1}^k\rho^{k-l}\al[k]\frac{\al[l-1]\Lmax_{l-1}}{\taumin_{l-1}}\right)<\infty.
\eals
The bound for the last term comes from \cref{l21} (b).
\item
\bals
\lim_{n\ra\infty}\sum_{k=1}^n\|\xB^c_\perp[k]\|^2
&\leq\lim_{n\ra\infty}\left(c_3^2\sum_{k=1}^n\rho^{2k}+2c_3c_4\rho\sum_{k=1}^n\sum_{l=1}^k\rho^{2k-l}\frac{\al[l-1]\Lmax_{l-1}}{\taumin_{l-1}}\right.\\
&\left.\qquad+c_4^2\rho^2\sum_{k=1}^n\sum_{l=1}^k\sum_{t=1}^k\rho^{2k-l-t}\frac{\al[l-1]\Lmax_{l-1}}{\taumin_{l-1}}\frac{\al[t-1]\Lmax_{t-1}}{\taumin_{t-1}}\right)<\infty.
\eals
The bound for the first term is natural. The double summation is bounded due to the second equality \cref{l21} (b) with $(\lm,\be[k],\nu[l])$ being $(\rho,\rho^k,\frac{\al[l-1]\Lmax_{l-1}}{\taumin_{l-1}})$. The condition of \cref{t7} $\sum_{n=1}^\infty(\Lmax_n)^3\left(\frac{\al[n]}{\taumin_n}\right)^2<\infty$ guarantees that $\sum_{n=1}^\infty\left(\frac{\al[n]\Lmax_n}{\taumin_n}\right)^2<\infty$. The inequality of the triple summation follows from
\bals
&\lim_{n\ra\infty}\sum_{k=1}^n\sum_{l=1}^k\sum_{t=1}^k\rho^{2k-l-t}\frac{\al[l-1]\Lmax_{l-1}}{\taumin_{l-1}}\frac{\al[t-1]\Lmax_{t-1}}{\taumin_{t-1}}\\
\leq&\lim_{n\ra\infty}\sum_{k=1}^n\sum_{l=1}^k\sum_{t=1}^k\rho^{k-l}\cdot\rho^{k-t}\cdot\frac{\left(\frac{\al[l-1]\Lmax_{l-1}}{\taumin_{l-1}}\right)^2+\left(\frac{\al[t-1]\Lmax_{t-1}}{\taumin_{t-1}}\right)^2}{2}\\
\leq&\lim_{n\ra\infty}\frac{1}{1-\rho}\sum_{k=1}^n\sum_{l=1}^k\rho^{k-l}\left(\frac{\al[l-1]\Lmax_{l-1}}{\taumin_{l-1}}\right)^2<\infty,
\eals
where the last inequality is due to the first equality of \cref{l21} (b). Again, the convergence of $\sum_{n=1}^\infty\left(\frac{\al[n]\Lmax_n}{\taumin_n}\right)^2$ is implied by the convergence of $\sum_{n=1}^\infty(\Lmax_n)^3\left(\frac{\al[n]}{\taumin_n}\right)^2$.
\end{enumerate}

\subsection{Proof of \cref{p23} (c)}\label{App:B-4}
We exploit the optimality of $\xBt^{av}_i$ and (F1\textprime) and (A2) to get
\beq\label{38}
\left[\xBh^\Si_{i,n}(\xBb)-\xBt^{av}_i\right]^T
\left[\gd_{\xB^\Si}\ft^*_{i,n}(\xBt^{av}_i;\xBb^\Si)+\piBt^{av}_i+(\0^{\Si\sm\{c\}},\pa G(\xBt^{c,av}_i))\right]\geq0;
\eeq
and the optimality of $\xBb$ (for the mapping of $\xBh^\Si_{i,n}(\xBb)$) leads to
\beq\label{39}
\left[\xBt^{av}_i-\xBh^\Si_{i,n}(\xBb)\right]^T
\left[\gd_{\xB^\Si}\ft^*_{i,n}(\xBh^\Si_{i,n}(\xBb);\xBb^\Si)+\piB^\Si_i(\xBb)+(\0^{\Si\sm\{c\}},\pa G(\xBh^c_{i,n}(\xBb))\right]\geq0.
\eeq
$\0^{\Si\sm\{c\}}$ is an all zero vector in the subspace $\Kc_{\Si\sm\{c\}}$. It should be clear that $\xBh^c_{i,n}(\xBb)$ refers to the component of $\xBh_{i,n}(\xBb)$ in the subspace $\Kc_c$. Then
\bal\label{40}
%\hspace*{-10pt}
&\tau_{i,n}\left\|\xBh^\Si_{i,n}(\xBb)-\xBt^{av}_i\right\|^2&\\
\leq&\left[\xBh^\Si_{i,n}(\xBb)-\xBt^{av}_i\right]^T\cdot\left[\gd_{\xB^\Si}\ft^*_{i,n}(\xBt^{av}_i;\xBb^\Si)-\gd_{\xB^\Si}\ft^*_{i,n}(\xBh^\Si_{i,n}(\xBb);\xBb^\Si)\right.&\\
&\qquad\qquad\left.+\left(\0^{\Si\sm\{c\}},\pa G(\xBh^c_{i,n}(\xBb))-\pa G(\xBt^{c,av}_i))\right)\right]&(\text{(F1\textprime) and (A2)})\\
\leq&\left[\xBh^\Si_{i,n}(\xBb)-\xBt^{av}_i\right]^T\cdot\left[\piBt^{av}_i-\piB^\Si_i(\xBb)\right]&(\text{\cref{38} and \cref{39}})\\
\leq&\left\|\xBh^\Si_{i,n}(\xBb)-\xBt^{av}_i\right\|\cdot\left\|\piBt^{av}_i-\piB^\Si_i(\xBb)\right\|&(\text{C-S inequality}).
\eal
From \cref{40},
\bals
&\left\|\xBh^\Si_{i,n}(\xBb)-\xBt^{av}_i\right\|\leq\frac{1}{\tau_{i,n}}\left\|\piBt^{av}_i-\piB^\Si_i(\xBb)\right\|\\
=&\frac{1}{\tau_{i,n}}\left\|(I_m\yB^{m,av}_i)_{m\in\Si}-\gd_{\xB^\Si}f^*_{i,n}(\xBb)-\sum_{j\neq i}\gd_{\xB^\Si}f^*_{j,n}(\xBb)\right\|\\
\leq&\frac{1}{\tau_{i,n}}\sum_{m\in\Si}\|I_m\yB^{m,av}_i-I_m\rBb^{m,av}\|\\
\leq&\frac{1}{\tau_{i,n}}\sum_{m\in\Si}I_m\|\yB^{m,av}-\1_{\Nc_m}\otimes\rBb^{m,av}\|.
\eals
Up until now, the context is clear enough to allow us to drop all $[n]$ time index. Again, we only focus on the case of $m=M+1$; that is, proving $\frac{1}{\tau_{i,n}}\|\yB^{c,av}-\1_I\otimes\rBb^{c,av}\|$ goes to zero. We calculate
\bal\label{33}
\yB^{c,av}[n]&=\WBh[n-1]\yB^{c,av}[n-1]+\Dl\rB^{c,av}[n,n-1]\\
&=\WBh[n-1](\WBh[n-2]\yB^{c,av}[n-2]+\Dl\rB^{c,av}[n-1,n-2])+\Dl\rB^{c,av}[n,n-1]\\
&=\PBh[n-1,n-2](\WBh[n-3]\yB^{c,av}[n-3]+\Dl\rB^{c,av}[n-2,n-3])\\
&\qquad+\PBh[n-1,n-1]\Dl\rB^{c,av}[n-1,n-2]+\Dl\rB^{c,av}[n,n-1]\\[-10pt]
&\es\vdots\\[-10pt]
&=\PBh[n-1,0]\rB^{c,av}[0]+\sum_{l=1}^{n-1}\PBh[n-1,l]\Dl\rB^{c,av}[l,l-1]+\Dl\rB^{c,av}[n,n-1],
\eal
and
\bal\label{34}
\1_I\otimes\rBb^{c,av}[n]
&=\1_I\otimes\left(\rBb^{c,av}[0]+\sum_{l=1}^n\Dl\rBb^{c,av}[l,l-1]\right)\\
&=\1_I\cdot\frac{\1_I^T\otimes I_d}{I}\left(\rB^{c,av}[0]+\sum_{l=1}^n\Dl\rB^{c,av}[l,l-1]\right)\\
&=J\rB^{c,av}[0]+\sum_{l=1}^{n-1}J\Dl\rB^{c,av}[l,l-1]+J\Dl\rB^{c,av}[n,n-1].
\eal
Similar to \cref{31-2} we have
\bal\label{41}
&\|\Dl\rB^{c,av}[l,l-1]\|\\
=&\sum_{i\in\Nc}\|\gd_{\xB^c}f^*_{i,l}(\xBb[l])-\gd_{\xB^c}f^*_{i,l-1}(\xBb[l-1])\|\\
=&\sum_{i\in\Nc}\|\gd_{\xB^c}f^*_{i,l}(\xBb[l])-\gd_{\xB^c}f^*_{i,l-1}(\xBb[l])+\gd_{\xB^c}f^*_{i,l-1}(\xBb[l])-\gd_{\xB^c}f^*_{i,l-1}(\xBb[l-1])\|\\
\leq&I\Lmax_{l-1}\|\xBb[l]-\xBb[l-1]\|+\sum_{i\in\Nc}\|\gd_{\xB^c}f^*_{i,l}(\xBb[l])-\gd_{\xB^c}f^*_{i,l-1}(\xBb[l])\|.
\eal
Similar to the technique as in \cref{35} to \eqref{37}, by combining \cref{33}, \cref{34}, plus \cref{app_c:l1}, and then \cref{41}, we have
\begingroup
\allowdisplaybreaks[4]
\begin{align*}\label{42}
&\quad\frac{1}{\tau_{i,n}}\|\yB^{c,av}[n]-\1_I\otimes\rBb^{c,av}[n]\|\\
&\leq c_1\frac{\rho^n}{\taumin_n}+c_2\sum_{l=1}^{n-1}\frac{\rho^{n-l}}{\taumin_n}\|\Dl\rB^{c,av}[l,l-1]\|+c_3\frac{1}{\taumin_n}\|\Dl\rB^{c,av}[n,n-1]\|\\
&\leq c_1\frac{\rho^n}{\taumin_n}+c_2\sum_{l=1}^{n-1}\frac{\rho^{n-l}}{\taumin_n}\left(I\Lmax_{l-1}\left\|\xBb[l]-\xBb[l-1]\right\|+\sum_i\|\gd_{\xB^c}f^*_{i,l}(\xBb[l])-\gd_{\xB^c}f^*_{i,l-1}(\xBb[l])\|\right)\\
&\qquad\qquad+c_3\frac{1}{\taumin_n}\left(I\Lmax_{n-1}\left\|\xBb[n]-\xBb[n-1]\right\|+\sum_i\|\gd_{\xB^c}f^*_{i,n}(\xBb[n])-\gd_{\xB^c}f^*_{i,n-1}(\xBb[n])\|\right)\\
&=c_1\frac{\rho^n}{\taumin_n}+c_2\sum_{l=1}^{n-1}\frac{\rho^{n-l}\al[l-1]\Lmax_{l-1}}{\taumin_n}\left\|\left(\frac{1}{I_m}\1_{I_m}^T\otimes I_{d_m}\Dl\xB^{m,inx}[l-1]\right)_{m\in[M+1]}\right\|\\
&\qquad\qquad+c_3\frac{\al[n-1]\Lmax_{n-1}}{\taumin_n}\left\|\left(\frac{1}{I_m}\1_{I_m}^T\otimes I_{d_m}\Dl\xB^{m,inx}[n-1]\right)_{m\in[M+1]}\right\|\\
&\qquad\qquad+c_2\sum_{l=1}^{n-1}\frac{\rho^{n-l}}{\taumin_n}\sum_i\|\gd_{\xB^c}f^*_{i,l}(\xBb[l])-\gd_{\xB^c}f^*_{i,l-1}(\xBb[l])\|\\
&\qquad\qquad+c_3\frac{1}{\taumin_n}\sum_i\|\gd_{\xB^c}f^*_{i,n}(\xBb[n])-\gd_{\xB^c}f^*_{i,n-1}(\xBb[n])\|\\
&\leq c_1\frac{\rho^n}{\taumin_n}+c_4\sum_{l=1}^{n-1}\rho^{n-l}\frac{\al[l-1](\Lmax_{l-1})^2}{\taumin_n\taumin_{l-1}}+c_5\frac{\al[n-1](\Lmax_{n-1})^2}{\taumin_n\taumin_{n-1}}\\
&\qquad\qquad+c_2\sum_{l=1}^{n-1}\frac{\rho^{n-l}}{\taumin_n}\sum_i\|\gd_{\xB^c}f^*_{i,l}(\xBb[l])-\gd_{\xB^c}f^*_{i,l-1}(\xBb[l])\|\\
&\qquad\qquad+c_3\frac{1}{\taumin_n}\sum_i\|\gd_{\xB^c}f^*_{i,n}(\xBb[n])-\gd_{\xB^c}f^*_{i,n-1}(\xBb[n])\|\hspace{1.5in}(\text{\cref{p23} (a)})\\
&\xrightarrow{n\ra\infty}0\hspace{4.4in}(\text{(T1), (T2), and (T3)}).
\end{align*}
\endgroup
In the last equation, we also have $\lim_{n\ra\infty}\frac{\al[n](\Lmax_{n})^2}{(\taumin_{n})^2}=0$ and $\lim_{n\ra\infty}\frac{1}{\taumin_n}\|\gd f^*_{i,n}(\xBb[n])-\gd f^*_{i,n-1}(\xBb[n])\|=0$ implied by the conditions of \cref{t7}. For the second part of the claim, we can equivalently prove
\[
\sum_{n=1}^\infty\frac{\al[n]\Lmax_n}{\taumin_n}\|\yB^{c,av}[n]-\1_I\otimes\rBb^{c,av}[n]\|<\infty.
\]
This is true because
\bals
&\sum_{n=1}^\infty\frac{\al[n]\Lmax_n}{\taumin_n}\|\yB^{c,av}[n]-\1_I\otimes\rBb^{c,av}[n]\|\\
\leq&c_1\sum_{n=1}^\infty\rho^n\frac{\al[n]\Lmax_n}{\taumin_n}
+c_4\sum_{n=1}^\infty\al[n]\Lmax_n\sum_{l=1}^{n-1}\rho^{n-l}\frac{\al[l-1](\Lmax_{l-1})^2}{\taumin_n\taumin_{l-1}}
+c_5\sum_{n=1}^\infty\frac{\al[n]\al[n-1]\Lmax_n(\Lmax_{n-1})^2}{\taumin_n\taumin_{n-1}}\\
&\qquad\qquad+c_2\sum_{n=1}^\infty\al[n]\Lmax_n\sum_{l=1}^{n-1}\frac{\rho^{n-l}}{\taumin_n}\sum_i\|\gd_{\xB^c}f^*_{i,l}(\xBb[l])-\gd_{\xB^c}f^*_{i,l-1}(\xBb[l])\|\\
&\qquad\qquad+c_3\sum_{n=1}^\infty\frac{\al[n]\Lmax_n}{\taumin_n}\sum_i\|\gd_{\xB^c}f^*_{i,n}(\xBb[n])-\gd_{\xB^c}f^*_{i,n-1}(\xBb[n])\|<\infty.
\eals
All the terms are finite because of the following. The first term is due to (T4) -- after multiplying a going-to-zero $\al[n]$, the term remains to be bounded. The second term is due to (T5). The third term is in the condition of \cref{t7}. The fourth term is due to (T6). The last term is also in the condition of \cref{t7}.

\subsection{Proof of \cref{p23} (d)}\label{App:B-5}
Recall we have $\xBt_i[n]=\underset{\xB_i\in\Kc_\Si}{\arg\min}\es\tilde{U}_{i,n}(\xB_i;\xB_i[n],\piBt_i[n])$ and $\xBt^{av}_i[n]=\underset{\xB_i\in\Kc_\Si}{\arg\min}\es\tilde{U}_{i,n}(\xB_i;\xBb_i[n],\piBt^{av}_i[n])$, where
\[
\tilde{U}_{i,n}(\xB_i;\xB_i[n],\piBt_i[n])=\ft^*_{i,n}(\xB_i;\xB_i[n])+\sum_{k\in\Sc_i}\piBt^k_i[n]^T(\xB^k_i-\xB^k_i[n])+G(\xB^c).
\]
These along with (F1\textprime) and (A2) lead to the following:
\beq\label{43}
(\xBt^{av}_i-\xBt_i)^T\cdot\left[\gd_{\xB^\Si}\ft^*_{i,n}(\xBt_i;\xB_i)+\piBt_i+(\0^{\Si\sm\{c\}},\pa G(\xBt^c_i))\right]\geq0
\eeq
and
\beq\label{44}
(\xBt_i-\xBt^{av}_i)^T\cdot\left[\gd_{\xB^\Si}\ft^*_{i,n}(\xBt^{av}_i;\xBb_i)+\piBt^{av}_i+(\0^{\Si\sm\{c\}},\pa G(\xBt^{c,av}_i))\right]\geq0.
\eeq
As we did in \cref{40},
\begingroup
\allowdisplaybreaks
\bals
&\tau_{i,n}\|\xBt_i-\xBt^{av}_i\|^2&\\
\leq&(\xBt_i-\xBt^{av}_i)^T\cdot\left[\gd_{\xB^\Si}\ft^*_{i,n}(\xBt_i;\xB_i)-\gd_{\xB^\Si}\ft^*_{i,n}(\xBt^{av}_i;\xB_i)+(\0^{\Si\sm\{c\}},\pa G(\xBt^c_i)-\pa G(\xBt^{c,av}_i)\right]&(\text{(A2) and (F1\textprime)})\\
\leq&(\xBt_i-\xBt^{av}_i)^T\cdot\left[\gd_{\xB^\Si}\ft^*_{i,n}(\xBt^{av}_i;\xBb_i)-\gd_{\xB^\Si}\ft^*_{i,n}(\xBt^{av}_i;\xB_i)+\piBt^{av}_i-\piBt_i\right]&\text{(\cref{43} and \cref{44})}\\
\leq&\|\xBt_i-\xBt^{av}_i\|\cdot\left[L_{i,n}\|\xBb_i-\xB_i\|+\left\|(I_m(\yB^{m,av}_i-\yB^m_i))_{m\in\Si}-\gd_{\xB^\Si}f^*_{i,n}(\xBb_i)-\gd_{\xB^\Si}f^*_{i,n}(\xB_i)\right\|\right]&\text{((N1))}\\
\leq&\|\xBt_i-\xBt^{av}_i\|\cdot\left[2L_{i,n}\|\xBb_i-\xB_i\|+\sum_{m\in\Si}\left\|I_m(\yB^m_i-\yB^{m,av}_i)\right\|\right].&
\eals
\endgroup
Hence,
\bal\label{45}
&\quad\left[\sum_{m\in\Si}\left\|\xBt^m_i-\xBt^{m,av}_i\right\|^2\right]^{1/2}\\
&\leq\frac{2L_{i,n}}{\tau_{i,n}}\left(\sum_{m\in\Si}\|\xBb^m-\xB^m_i\|\right)+\sum_{m\in\Si}\frac{I_m}{\tau_{i,n}}\|\yB^m-\yB^{m,av}\|\\
&\leq\frac{2L_{i,n}}{\tau_{i,n}}\left(\sum_{m\in\Si}\|\xBb^m-\xB^m_i\|\right)+\sum_{m\in\Si}\frac{I_m}{\tau_{i,n}}\big(\|\1_{\Nc_m}\otimes(\rBb^m-\rBb^{m,av})\|\\
&\qquad\qquad+\|\yB^m-\yB^{m,av}-\1_{\Nc_m}\otimes(\rBb^m-\rBb^{m,av})\|\big).
\eal

Since $\|\xBt^m_i-\xBt^{m,av}_i\|$ is not larger than $\left[\sum_{m\in\Si}\left\|\xBt^m_i-\xBt^{m,av}_i\right\|^2\right]^{1/2}$, \cref{45} implies the former goes to zero as $n$ goes to infinity if we can show all terms in the RHS do so. The first term does go to zero as we showed in part (b) (combining \cref{37}, \cref{l21} (a), and the fact that $\lim_{n\ra\infty}\al[n]\left(\frac{\Lmax_n}{\taumin_n}\right)^2$). The following shows this property holds for the remaining two terms as well. As always we omit all time index $[n]$ from above as the context is clear enough.

We have\\
\resizebox{1\linewidth}{!}{\begin{minipage}{\linewidth}
\bal
\frac{1}{\tau_{i,n}}\|\1_\Nm\otimes(\rBb^m[n]-\rBb^{m,av}[n])\|
&\leq\frac{1}{\tau_{i,n}}\sum_{i\in\Nm}\left\|\gd_{\xB^m}f^*_{i,n}(\xB_i[n])-\gd_{\xB^m}f^*_{i,n}(\xBb^\Si[n])\right\|\\
&\leq\sum_{i\in\Nm}\frac{L_{i,n}}{\tau_{i,n}}\|\xB_i[n]-\xBb^\Si[n]\|\hspace{1.8in}\text{((N1))}\\
&\xrightarrow{n\ra\infty}0\hspace{2in}(\text{\cref{p23} (b)}),
\eal
\end{minipage}}
and
\begingroup
\allowdisplaybreaks[4]
\begin{align*}
&\quad\frac{1}{\tau_{i,n}}\|\yB^m-\yB^{m,av}-\1_\Nm\otimes(\rBb^m-\rBb^{m,av})\|&\\
&\leq c_1\frac{\rho^n}{\taumin_n}+c_2\sum_{l=1}^{n-1}\frac{\rho^{n-l}}{\taumin_n}\left\|\Dl\rB^m[l,l-1]-\Dl\rB^{m,av}[l,l-1]\right\|\\
&\qquad\qquad+c_3\frac{1}{\taumin_n}\left\|\Dl\rB^m[n,n-1]-\Dl\rB^{m,av}[n,n-1]\right\|\hspace{1.7in}\text{(\eqref{33}, \eqref{34}, and \eqref{42})}\\
&\leq c_1\frac{\rho^n}{\taumin_n}+c_4\sum_{l=1}^{n-1}\frac{\rho^{n-l}}{\taumin_n}\sum_{i\in\Nm}\left(\Lmax_l\|\xB^m_i[l]-\xBb^m[l]\|+\Lmax_{l-1}\|\xB^m_i[l-1]-\xBb^m[l-1]\|\right.\\
&\qquad\qquad+\left.\|\gd_{\xB^m}f^*_{i,l}(\xB_i[l-1])-\gd_{\xB^m}f^*_{i,l-1}(\xB_i[l-1])\|+\|\gd_{\xB^m}f^*_{i,l}(\xBb[l-1])-\gd_{\xB^m}f^*_{i,l-1}(\xBb[l-1])\|\right)\\
&\qquad\qquad+c_5\frac{1}{\taumin_n}\sum_{i\in\Nm}\left(\Lmax_n\|\xB^m_i[n]-\xBb^m[n]\|+\Lmax_{n-1}\|\xB^m_i[n-1]-\xBb^m[n-1]\|\right.\\
&\qquad\qquad+\|\gd_{\xB^m}f^*_{i,l}(\xB_i[n-1])-\gd_{\xB^m}f^*_{i,l-1}(\xB_i[n-1])\|\\
&\qquad\qquad+\left.\|\gd_{\xB^m}f^*_{i,l}(\xBb[n-1])-\gd_{\xB^m}f^*_{i,l-1}(\xBb[n-1])\|\right)\hspace{2.5in}\text{((N1))}\\
&\xrightarrow{n\ra\infty}0\hspace{2.1in}\text{((T1), (T2), \cref{p23} (b), \cref{l21} (a), and (T3))}.
\end{align*}
\endgroup
For the terms of the form $\frac{L_n}{\tau_n}\|\xB_\perp[n]\|$ to converge to zero, refer to \cref{37} and Assumption T1 and T2. In the second line, from \cref{33}, \cref{34}, and \cref{42} we know that $\|\yB^{m,av}-\1_\Nm\otimes\rBb^{m,av}\|$ can be represented as a sum of $\Dl\rB^{m,av}[l,l-1]$'s; using the same method $\|\yB^m-\1_\Nm\otimes\rBb^m\|$ can also be represented as a sum of $\Dl\rB^m[l,l-1]$'s, which we omit here. In the last inequality one can alternatively use \cref{31-3} to bound $\Dl\rB^m[l,l-1]$ and $\Dl\rB^{m,av}[l,l-1]$, which is simpler and sufficient for our purposes.

For the second part of the claim,
\begingroup
\allowdisplaybreaks[4]
\begin{align*}
&\quad\sum_{n=1}^\infty\al[n]\Lmax_n\left\|\xBt^m_i[n]-\xBt^{m,av}_i[n]\right\|\\
&\leq c_6\sum_{n=1}^\infty\frac{\al[n](\Lmax_n)^2}{\taumin_n}\|\xB^\Si_\perp[n]\|+c_1\sum_{n=1}^\infty\frac{\rho^n\al[n]\Lmax_n}{\taumin_n}\\
&\quad+c_4\sum_{n=1}^\infty\sum_{l=1}^{n-1}\frac{\rho^{n-l}\al[n]\Lmax_n}{\taumin_n}\sum_{i\in\Nm}\big(\Lmax_l\|\xB^m_i[l]-\xBb^m[l]\|+\Lmax_{l-1}\|\xB^m_i[l-1]-\xBb^m[l-1]\|\\
&\qquad+\|\gd_{\xB^m}f^*_{i,l}(\xB_i[l-1])-\gd_{\xB^m}f^*_{i,l-1}(\xB_i[l-1])\|+\|\gd_{\xB^m}f^*_{i,l}(\xBb[l-1])-\gd_{\xB^m}f^*_{i,l-1}(\xBb[l-1])\|\big)\\
&\quad+c_5\sum_{n=1}^\infty\frac{\al[n]\Lmax_n}{\taumin_n}\sum_{i\in\Nm}\big(\Lmax_n\|\xB^m_i[n]-\xBb^m[n]\|+\Lmax_{n-1}\|\xB^m_i[n-1]-\xBb^m[n-1]\|\\
&\qquad+\|\gd_{\xB^m}f^*_{i,l}(\xB_i[n-1])-\gd_{\xB^m}f^*_{i,l-1}(\xB_i[n-1])\|+\|\gd_{\xB^m}f^*_{i,l}(\xBb[n-1])-\gd_{\xB^m}f^*_{i,l-1}(\xBb[n-1])\|\big)\\
&<\infty.
\end{align*}
\endgroup
For the first term, use \cref{37}, (T4), and (T5). Second term is finite due to (T4) with additional $\al[n]$. The terms in the forth line are just like the first term. The terms in the fifth line converge by the condition of \cref{t7}. The terms in the second line are of the type $\sum_n\frac{\al[n]L_n}{\tau_n}\sum_l\rho^{n-l}L_l\|\xB_\perp[l]\|$, from \cref{37} and (T4) one can show that $\sum_n\frac{\al[n]L^2_n}{\tau_n}\|\xB_\perp[n]\|$ converges, hence the convergence of the terms by applying second part of \cref{l24}. The terms in the third line converge because of (T6).

\subsection{Proof of \cref{t7}}\label{af}
Denote $F^*_n=\sum_{i\in\Nc}f^*_{i,n}$. By descent Lemma,
\begingroup
\allowdisplaybreaks[4]
\begin{align*}\label{48}
&\quad F^*_n(\xBb[n+1])\\
&\leq F^*_n(\xBb[n])+\gd F^*_n(\xBb[n])^T(\xBb[n+1]-\xBb[n])+\frac{\Lmax_n}{2}\|\xBb[n+1]-\xBb[n]\|^2\\
&=F^*_n(\xBb[n])+\sum_m\left[\gd_{\xB^m}F^*_n(\xBb[n])^T(\xBb^m[n+1]-\xBb^m[n])+\frac{\Lmax_n}{2}\|\xBb^m[n+1]-\xBb^m[n]\|^2\right]\\
&=F^*_n(\xBb[n])+\sum_m\left[\frac{\al[n]}{I_m}\gd_{\xB^m}F^*_n(\xBb[n])^T\sum_{i\in\Nc_m}(\xB^{m,inx}_i[n]-\xBb^m_i[n])+\frac{\Lmax_n}{2}\|\xBb^m[n+1]-\xBb^m[n]\|^2\right]\\
&\leq F^*_n(\xBb[n])+\sum_m\Bigg[\frac{\al[n]}{I_m}\gd_{\xB^m}F^*_n(\xBb[n])^T\sum_{i\in\Nc_m}\bigg[\left(\xBh^m_{i,n}(\xBb[n])-\xBb^m_i[n]\right)+\left(\xBt^{m,av}_i[n]-\xBh^m_{i,n}(\xBb[n])\right)\\
&\qquad\qquad+\left(\xBt^m_i[n]-\xBt^{m,av}_i[n]\right)+\left(\xB^{m,inx}_i[n]-\xBt^m_i[n]\right)\bigg]+\frac{\Lmax_n}{2}\|\xBb^m[n+1]-\xBb^m[n]\|^2\Bigg].\\
\end{align*}
\endgroup
By the convexity of $G$ (A2),
\bal\label{49}
G(\xBb^c[n+1])&\leq(1-\al[n])G(\xBb^c[n])+\al[n]G\left(\frac{1}{I}\sum_{i=1}^I\xB^{c,inx}_i[n]\right)\\
&\leq(1-\al[n])G(\xBb^c[n])+\frac{\al[n]}{I}\sum_{i=1}^IG(\xB^{c,inx}_i[n]).
\eal
Then using \cref{p19} (a) and the fact that $G$ has bounded subgradients,
\bal\label{50}
&\quad\sum_m\frac{\al[n]}{I_m}\gd_{\xB^m}F^*_n(\xBb[n])^T\sum_{i\in\Nc_m}\left(\xBh^m_{i,n}(\xBb[n])-\xBb^m_i[n]\right)\\
&\leq-\taumin_n\al[n]\diamondsuit[n]+\al[n]\left[G(\xBb^c[n])-\frac{1}{I}\sum_{i=1}^IG\left(\xBh^c_{i,n}(\xBb[n])\right)\right]\hspace{1.5in}\text{(\cref{p19} (a))}\\
&\leq-\taumin_n\al[n]\diamondsuit[n]+G(\xBb^c[n])-G(\xBb^c[n+1])+\frac{\al[n]}{I}\sum_{i=1}^I\left\|G(\xB^{c,inx}_i[n])-G\left(\xBh^c_{i,n}(\xBb[n])\right)\right\|\hspace{0.3in}\text{(by \cref{49})}\\
&\leq-\taumin_n\al[n]\diamondsuit[n]+G(\xBb^c[n])-G(\xBb^c[n+1])+\frac{L_G\al[n]}{I}\sum_{i=1}^I\left\|\xB^{c,inx}_i[n]-\xBh^c_{i,n}(\xBb[n])\right\|\hspace{0.85in}\text{((A2))}\\
\eal
where $\diamondsuit[n]$ stands for the expression $\sum_m\sum_{i\in\Nc_m}\left\|\xBh^m_i(\xBb[n])-\xBb^m_i[n]\right\|^2$. Combining \cref{48}, \cref{50} and (N1) with Cauchy-Schwarz inequality as well as triangle inequality, we get
\bal\label{51}
&F^*_n(\xBb[n+1])
\leq F^*_n(\xBb[n])+G(\xBb^c[n])-G(\xBb^c[n+1])+\frac{\al[n]L_G}{I}\sum_{i=1}^I\left\|\xB^{c,inx}_i[n]-\xBh^c_{i,n}(\xBb^c[n])\right\|\\
&\qquad+\sum_m\left[\frac{\al[n]\Lmax_n}{I_m}\sum_{i\in\Nc_m}\left(\left\|\xBt^{m,av}_i[n]-\xBh^m_{i,n}(\xBb[n])\right\|+\left\|\xBt^m_i[n]-\xBt^{m,av}_i[n]\right\|+\ep^m_i[n]\right)\right]\\
&\qquad+\sum_m\frac{\Lmax_n}{2}\|\xBb^m[n+1]-\xBb^m[n]\|^2-\taumin_n\al[n]\diamondsuit[n].
\eal
From the triangle inequality, \cref{p23} (a) and \cref{f24} (e),
\bals
\left\|\xB^{c,inx}_i[n]-\xBh^c_{i,n}(\xBb^c[n])\right\|
&\leq\left\|\xB^{c,inx}_i[n]-\xBt^c_i[n]\right\|+\left\|\xBt^c_i[n]-\xBt^{c,av}_i[n]\right\|+\left\|\xBt^{c,av}_i[n]-\xBh^c_{i,n}(\xBb^c[n])\right\|,\\
\|\xBb^m[n+1]-\xBb^m[n]\|^2&\leq\left(\frac{c^m\al[n]\Lmax_n}{I_m\taumin_n}\right)^2\quad\forall\es m.
\eals
Substitute these expression back into \eqref{51} and rearrange the terms to get
\bal\label{51-1}
&U^*_{n+1}(\xBb[n+1])
\leq U^*_n(\xBb[n])-\taumin_n\al[n]\diamondsuit[n]+c_1(\Lmax_n)^3\left(\frac{\al[n]}{\taumin_n}\right)^2\\
&\qquad+F^*_{n+1}(\xBb[n+1])-F^*_n(\xBb[n+1])\\
&\qquad+c_2\sum_m\left[\al[n]\Lmax_n\sum_{i\in\Nc_m}\left(\left\|\xBt^{m,av}_i[n]-\xBh^m_{i,n}(\xBb[n])\right\|+\left\|\xBt^m_i[n]-\xBt^{m,av}_i[n]\right\|+\ep^m_i[n]\right)\right].
\eal

We now exploit \cref{l22} with $Y[n]=U^*_n(\xBb[n])$, $X[n]=\taumin_n\al[n]\diamondsuit[n]$ and
\bals
&Z[n]=c_1(\Lmax_n)^3\left(\frac{\al[n]}{\taumin_n}\right)^2+F^*_{n+1}(\xBb[n+1])-F^*_n(\xBb[n+1])\\
&\qquad+c_2\sum_m\left[\al[n]\Lmax_n\sum_{i\in\Nc_m}\left(\left\|\xBt^{m,av}_i[n]-\xBh^m_{i,n}(\xBb[n])\right\|+\left\|\xBt^m_i[n]-\xBt^{m,av}_i[n]\right\|+\ep^m_i[n]\right)\right].
\eals
Since $U(\xBb[n])$ is coercive ((A3)), $Y[n]\not\ra-\infty$; on the other hand, from \cref{p23} (c), (d), and the assumption of the Theorem, $\sum_{n=1}^\infty Z[n]<\infty$. Thus, by \cref{l22} $\{U^*_n(\xBb[n])\}$ converges to a finite value and $\sum_{n=1}^\infty\taumin_n\al[n]\diamondsuit[n]$ converges as well, which means
\[
\sum_{n=1}^\infty\taumin_n\al[n]\left\|\xBh^m_{i,n}(\xBb[n])-\xBb^m[n]\right\|^2<\infty\quad\forall\es i\in\Nc_m,\es\forall\es m.
\]
This in turn implies
\[
\lim_{n\ra\infty}\left\|\xBh^m_{i,n}(\xBb[n])-\xBb^m[n]\right\|=0\quad\forall\es i\in\Nc_m,\es\forall\es m.
\]
At this point the localization is no longer an issue, and we will use the generalized definition of $\xBh_{i,n}(\xBb[n])\in\Kc$ so that we have $\lim_{n\ra\infty}\left\|\xBh_{i,n}(\xBb[n])-\xBb[n]\right\|=0$ for all $i\in\Nc$.

Since $\{\xBb[n]\}$ is bounded following from the convergence of $\{U^*_n(\xBb[n])\}$, there exists a limit point $\xBb^\infty\in\Kc$ of the set. We assume $\xBb[n]\ra\xBb^\infty$. If this is not the case, then one can find a subsequence $\xBb[n_k]$ indexed by $k$ such that $\xBb[n_k]\ra\xBb^\infty$ as $k\ra\infty$. We consider a partition of three cases: (1) bounded gradient ($\exists\es B\text{ s.t. }\|\gd f_i(\xB)\|<B\es\forall\es i,\xB$), (2) unbounded gradient and interior point ($\xBb^\infty\in int(\Kc)$), and (3) unbounded gradient and boundary point ($\xBb^\infty\in bd(\Kc)$).

\noindent\underline{(1) bounded gradient}: Recall the map defined in \cref{p19}
\[
\xBh_{i,n}(\xBt)=\underset{\xB}{\arg\min}\es\ft^*_{i,n}(\xB;\xBt)+\piB_i(\xBt)^T(\xB-\xBt)+G(\xB)\triangleq\underset{\xB}{\arg\min}\es\tilde{U}_{i,n}(\xB;\xBt).
\]
This map is converging to the following map
\beq\label{51-2}
\xBh_{i}(\xBt)=\underset{\xB}{\arg\min}\es\ft_i(\xB;\xBt)+\piB_i(\xBt)^T(\xB-\xBt)+G(\xB)\triangleq\underset{\xB}{\arg\min}\es\tilde{U}_{i}(\xB;\xBt),
\eeq
which might be multi-valued since we do not require $\ft_i$ to be strongly convex. The latter map is well-defined everywhere only with bounded gradient. Otherwise $\piB_i$ could be infinite; moreover, when $\gd f_i(\xB)=\infty$ and $\xB\in int(\Kc)$, it is not possible to achieve $\gd\ft_i(\xB;\xB)=\gd f_i(\xB)$, $\ft_i$ being defined everywhere and being convex simultaneously. Thus, the analysis for this case does not work for the other two cases.

Now consider the two maps evaluated at $\xBb[n]$ and $\xBb^\infty$ respectively, $\xBh_{i,n}(\xBb[n])$, the minimizer of $\tilde{U}_{i,n}(\bullet;\xBb[n])\triangleq\psi_n$, and $\xBh_{i}(\xBb^\infty)$, the set of minimizers of $\tilde{U}_{i}(\bullet;\xBb^\infty)\triangleq\psi$. We have the following two properties.
\begin{itemize}
\item
$\{\psi_n\}$ is eventually level-bounded, i.e. $\forall\es\al\in\R$, $\buU_{n\in N,N\in\Nc_\infty}lev_{\leq\al}\psi_n$ is bounded. Refer to \cite{Rockafellar_VarAna_1998}, p. 8, p. 109, and p. 123 for the definitions of the notations. This is ensured by Assumption F3, i.e. either $\ft_i(\bullet;\xB)$ is coercive $\forall\es\xB,i$ or $G(\bullet)$ is coercive.
\item
$\psi_n\overset{e}{\ra}\psi$, i.e. $\psi_n$ epi-converges to $\psi$. See \cite{Rockafellar_VarAna_1998}, p. 241 for the definition. This is due to $\{\tilde{U}_{i,n}\}$ and $\tilde{U}_i$ being continuous and $\lim_{n\ra\infty}\tilde{U}_{i,n}=\tilde{U}_i$, then by \cite{Rockafellar_VarAna_1998} Theorem 7.2, p. 241 we have $\psi_n\overset{e}{\ra}\psi$.
\end{itemize}

By \cite{Rockafellar_VarAna_1998} Theorem 7.33, p. 266, with these two properties, we then have
\[
\xBb^\infty=\lim_{n\ra\infty}\xBh_{i,n}(\xBb[n])=\underset{n\ra\infty}{\lim\sup}(\arg\min\psi_n)\rU\arg\min\psi=\xBh_{i}(\xBb^\infty).
\]
In \cite{Scutari_NEXT_2016} Proposition 5(b) says that the fixed point of $\xBh_{i}$ is also the stationary solution of the original optimization problem, which is proved in \cite{Scutari_Parallel_2015} Proposition 8(b). Things change slightly here as the minimizer of $\xBh_{i}$ may not be unique. However, in the proof they did not exploit any strong convexity property. Hence, we still have $\xBb^\infty$ being a stationary solution.

\noindent\underline{(2) unbounded gradient and interior point}: Effectively we want to show
\beq\label{52}
\gd F(\xBb^\infty)^T(\zB-\xBb^\infty)+G(\zB)-G(\xBb^\infty)\geq0\quad\forall\es\zB\in\Kc,
\eeq
but we can no longer argue anything with $\xBh_i$. Only for the following we will write $\xBb_n$ instead of $\xBb[n]$ for simplicity. From the optimality condition of $\xBh_{i,n}(\xBb_n)$, we have that for all $\zB\in\Kc$,
\bal\label{53}
0&\leq\left[\gd\ft^*_{i,n}(\xBh_{i,n}(\xBb_n);\xBb_n)+\sum_{j\neq i}\gd f^*_{i,n}(\xBb_n)\right]^T\left(\zB-\xBh_{i,n}(\xBb_n)\right)+G(z)-G\left(\xBh_{i,n}(\xBb_n)\right)\\
&=\left[\gd\ft^*_{i,n}(\xBb_n;\xBb_n)+\sum_{j\neq i}\gd f^*_{i,n}(\xBb_n)\right]^T\left(\zB-\xBh_{i,n}(\xBb_n)\right)+G(z)-G\left(\xBh_{i,n}(\xBb_n)\right)\\
&\qquad+\left[\gd\ft^*_{i,n}(\xBh_{i,n}(\xBb_n);\xBb_n)-\gd\ft^*_{i,n}(\xBb_n;\xBb_n)\right]^T\left(\zB-\xBh_{i,n}(\xBb_n)\right),
\eal
where $\gd\ft^*_{i,n}(\xBb_n;\xBb_n)+\sum_{j\neq i}\gd f^*_{i,n}(\xBb_n)$ is just $\gd F^*_n(\xBb_n)$. The terms in the second bracket are bounded as follows
\bals
&\left\|\gd\ft^*_{i,n}(\xBh_{i,n}(\xBb_n);\xBb_n)-\gd\ft^*_{i,n}(\xBb_n;\xBb_n)\right\|\\
\leq&\left\|\gd\ft^*_{i,n}(\xBh_{i,n}(\xBb_n);\xBb_n)-\gd\ft^*_{i,n}(\xBh_{i,n}(\xBb_n);\xBh_{i,n}(\xBb_n))\right\|+\left\|\gd\ft^*_{i,n}(\xBh_{i,n}(\xBb_n);\xBh_{i,n}(\xBb_n))-\gd\ft^*_{i,n}(\xBb_n;\xBb_n)\right\|\\
=&\left\|\gd\ft^*_{i,n}(\xBh_{i,n}(\xBb_n);\xBb_n)-\gd\ft^*_{i,n}(\xBh_{i,n}(\xBb_n);\xBh_{i,n}(\xBb_n))\right\|+\left\|\gd f^*_{i,n}(\xBh_{i,n}(\xBb_n))-\gd f^*_{i,n}(\xBb_n)\right\|\\
\leq&L_{i,n}\left\|\xBh_{i,n}(\xBb_n)-\xBb_n\right\|+L_{i,n}\left\|\xBh_{i,n}(\xBb_n)-\xBb_n\right\|
\leq2\Lmax_n\left\|\xBh_{i,n}(\xBb_n)-\xBb_n\right\|,
\eals
where the second inequality is due to the Lipschitz continuities of $\gd\ft^*_{i,n}(\xB;\bullet)$ and $\gd f^*_{i,n}(\bullet)$. Since we assume $\sum_n(\Lmax_n)^3\left(\frac{\al[n]}{\taumin_n}\right)^2<\infty$ in the condition and get $\sum_n\al[n]\taumin_n\|\xBh_{i,n}(\xBb_n)-\xBb_n\|^2<\infty$, it must be that $\|\xBh_{i,n}(\xBb_n)-\xBb_n\|^2=O\left(\al[n]\frac{(\Lmax_n)^3}{(\taumin_n)^3}\right)$. Hence, with the conditions of $\lim_{n\ra\infty}\al[n]\frac{(\Lmax_n)^5}{(\taumin_n)^3}=0$ and $\lim_{n\ra\infty}\gd F^*_n=\gd F$, taking $n\ra\infty$ in \cref{53} yields exactly \cref{52}. It is evident that $\xBb^\infty$ must be a point such that $\gd F(\xBb^\infty)<\infty$, because if not so $\xBb^\infty$ is an interior point and there must exist one descent direction.

\noindent\underline{(3) unbounded gradient and boundary point}: We can consider two subcases.
\vspace{3pt}
\begin{itemize}
\item
$\gd F(\xBb^\infty)<\infty$: we can use the same argument in case (2) to show that $\xBb^\infty$ is a stationary solution. If we have $\|\gd f_i(\xBb^\infty)\|<B\es\forall\es i$, we can also use the same argument in case (1) confined to a small neighborhood of $\xBb^\infty$.
\item
$\gd F(\xBb^\infty)=\infty$: the definition of stationary solution fails here and we can only turn to the definition of local minimum. However, both \emph{NEXT} and our algorithm can numerically converge to a point which is not a local minimum.
\end{itemize}

\section{Pseudo Code of the Algorithms}\label{app:c}
In this appendix we give the complete pseudo code of algorithms in \cref{sec:app} for reader's reference. All $\nb$'s represent $n+1$ for compression.

\subsection{LXGP-RM}
The LXGP-RM algorithm is given as follows:
\begin{algorithm}[H]
\caption{LXGP-RM}
\label{a13}
{\fontsize{8}{8}
\begin{algorithmic}[1]
\State{Initialization: $\forall\es b$, $p^b_{BK}[0]=\ep$, $x_{bI(b)K}[0]=0$, $r^b_{BK}[0]=0$, $\pit^b_{BK}[0]=0$, $n=0$}
\While{$p^b_{BK}[n]$ and $x_{bI(b)K}[n]$ do not satisfy the termination criterion}
\State{$n\la n+1$}
\State{$\al[n]=\frac{\al_0}{(n+1)^\be}$}
%\State{\textit{Local SCA optimization}}
\State{$(\pt^b_{BK}[n],\xt_{bI(b)K}[n])=\underset{p^b_{BK},x_{bI(b)K}}{\arg\min}\ft_b(p^b_{BK},x_{bI(b)K};p^b_{BK}[n],x_{bI(b)K}[n])+\pit^b_{BK}[n]\cdot(p^b_{BK}-p^b_{BK}[n])$}
\State{$q^b_{BK}[n]=p^b_{BK}[n]+\al[n](\pt^b_{BK}[n]-p^b_{BK}[n])$}
%\State{\textit{Consensus update}}
\State{$x_{bI(b)K}[\nb]=x_{bI(b)K}[n]+\al[n](\xt_{bI(b)K}[n]-x_{bI(b)K}[n])$}
\State{$p^b_{BK}[\nb]=\sum_{b'\in Nb(b)}W_{bb'}q^{b'}_{BK}[n]$}
\State{$r^b_{BK}[\nb]=\sum_{b'\in Nb(b)}W_{bb'}r^{b'}_{BK}[n]+\left[\gd_{p^b_{BK}}f_b(p^b_{BK}[\nb],x_{bI(b)K}[\nb])-\gd_{p^b_{BK}}f_b(p^b_{BK}[n],x_{bI(b)K}[n])\right]$}
\State{$\pit^b_{BK}[\nb]=|B|\1_{BK}\circ r^b_{BK}[\nb]-\gd_{p^b_{BK}}f_b(p^b_{BK}[\nb],x_{bI(b)K}[\nb])$}
\EndWhile
\Ensure{$p^b_{bK}[n]$ and $x_{bI(b)K}[n]$}
\end{algorithmic}
}
\end{algorithm}

In the algorithm, $p^b_{BK}$ is the collection of the variables $p^b_{b'k}\es\forall\es b'\in B,k\in K$, and so are the other quantities $x_{bI(b)K}$, $q^b_{BK}$, etc. The collection of variables $p^b_{BK}$ can be viewed as a $B\times K$ matrix, or a vector of $BK$ dimensions. $\1_{BK}$ is a $B\times K$ matrix (or vector) consisting of all 1's. $\circ$ denotes the element-wise product, also known as Hadamard product or Schur product.

\subsection{LXLP-RM}
Using the same notations as in LXGP-RM, the LXLP-RM algorithm is given as follows:
\begin{algorithm}[H]
\caption{LXLP-RM}
\label{a14}
{\fontsize{8}{8}
%\resizebox{.9\hsize}{!}{}
\begin{algorithmic}[1]
\State{Initialization: $\forall\es b$, $p^b_{Nb(b)K}[0]=\ep$, $x_{bI(b)K}[0]=0$, $r^b_{Nb(b)K}[0]=0$, $\pit^b_{Nb(b)K}[0]=0$, $n=0$}
\While{$p^b_{Nb(b)K}[n]$ and $x_{bI(b)K}[n]$ do not satisfy the termination criterion}
\State{$n\la n+1$}
\State{$\al[n]=\frac{\al_0}{(n+1)^\be}$}
%\State{\textit{Local SCA optimization}}
\State{$(\pt^b_{Nb(b)K}[n],\xt_{bI(b)K}[n])=\underset{p^b_{Nb(b)K},x_{bI(b)K}}{\arg\min}\ft_b(p^b_{Nb(b)K},x_{bI(b)K};p^b_{Nb(b)K}[n],x_{bI(b)K}[n])$}
\Statex{\hspace{70pt}$+\pit^b_{Nb(b)K}[n]\cdot(p^b_{Nb(b)K}-p^b_{Nb(b)K}[n])$}
\State{$q^b_{Nb(b)K}[n]=p^b_{Nb(b)K}[n]+\al[n](\pt^b_{Nb(b)K}[n]-p^b_{Nb(b)K}[n])$}
%\State{\textit{Consensus update}}
\State{$x_{bI(b)K}[\nb]=x_{bI(b)K}[n]+\al[n](\xt_{bI(b)K}[n]-x_{bI(b)K}[n])$}
\State{$p^b_{Nb(b)K}[\nb]=\sum_{b'\in Nb(b)}W_{bb'}(Nb(b))q^{b'}_{Nb(b)K}[n]$}
\State{$r^b_{Nb(b)K}[\nb]=\sum_{b'\in Nb(b)}W_{bb'}(Nb(b))r^{b'}_{Nb(b)K}[n]$}
\Statex{\hspace{10pt}$+\left[\gd_{p^b_{Nb(b)K}}f_b(p^b_{Nb(b)K}[\nb],x_{bI(b)K}[\nb])-\gd_{p^b_{Nb(b)K}}f_b(p^b_{Nb(b)K}[n],x_{bI(b)K}[n])\right]$}
\State{$\pit^b_{Nb(b)K}[\nb]=[(d_{Nb(b)}+\1_{Nb(b)})\1_{K}^T]\circ r^b_{Nb(b)K}[\nb]-\gd_{p^b_{Nb(b)K}}f_b(p^b_{Nb(b)K}[\nb],x_{bI(b)K}[\nb])$}
\EndWhile
\Ensure{$p^b_{bK}[n]$ and $x_{bI(b)K}[n]$}
\end{algorithmic}
}
\end{algorithm}

In line 10, the notation stands for $p^b_{b''K}[n+1]$ being updated as $\sum_{b'\in Nb(b)}W_{bb'}(b'')q^{b'}_{b''K}[n]$ for all $b''\in Nb(b)$, and so do line 11 and 12. In line 12, $d_{b''}$ means the degree of $b''$, i.e. $d_{b''}=|N(b'')|$. If considering $\pit^b_{Nb(b)K}$ to be a $Nb(b)\times K$ matrix, then $(d_{Nb(b)}+\1_{Nb(b)})\1_{K}^T$ consists of $Nb(b)$ rows; each row has $|K|$ elements, and every element in the $b''$-th row is $d_{b''}+1=|Nb(b'')|$.

\subsection{GXGP-CM}
the GXGP-CM algorithm is given as follows:
\vspace{-5pt}
\begin{algorithm}[H]
\caption{GXGP-CM}
\label{a15}
{\fontsize{8}{8}
\begin{algorithmic}[1]
\State{Initialization: $\forall\es b$, $p^b_{BK}[0]=\ep$, $x^b_{BI(B)K}[0]=0$, $r^b_{BK}[0]=0$, $\pit^b_{BK}[0]=0$, $y^b_{BI(B)K}[0]=0$, $\taut^b_{BI(B)K}[0]=0$, $n=0$}
\While{$p^b_{BK}[n]$ and $x^b_{BI(B)K}[n]$ do not satisfy the termination criterion}
\State{$n\la n+1$}
\State{$\al[n]=\frac{\al_0}{(n+1)^\be}$}
%\State{\textit{Local SCA optimization}}
\State{$(\pt^b_{BK}[n],\xt^b_{BI(B)K}[n])=\underset{p^b_{BK},x^b_{BI(B)K}}{\arg\min}\ft_b(p^b_{BK},x^b_{BI(B)K};p^b_{BK}[n],x^b_{BI(B)K}[n])$}
\Statex{$+\pit^b_{BK}[n]\cdot(p^b_{BK}-p^b_{BK}[n])+\taut^b_{BI(B)K}[n]\cdot(x^b_{BI(B)K}-x^b_{BI(B)K}[n])+G(p^b_{BK},x^b_{BI(B)K})$}
\State{$q^b_{BK}[n]=p^b_{BK}[n]+\al[n](\pt^b_{BK}[n]-p^b_{BK}[n])$}
\State{$z^b_{BI(B)K}[\nb]=x^b_{BI(B)K}[n]+\al[n](\xt^b_{BI(B)K}[n]-x^b_{BI(B)K}[n])$}
%\State{\textit{Consensus update}}
\State{$p^b_{BK}[\nb]=\sum_{b'\in Nb(b)}W_{bb'}q^{b'}_{BK}[n]$}
\State{$x^b_{BI(B)K}[\nb]=\sum_{b'\in Nb(b)}W_{bb'}z^{b'}_{BI(B)K}[n]$}
\State{$r^b_{BK}[\nb]=\sum_{b'\in Nb(b)}W_{bb'}r^{b'}_{BK}[n]+\left[\gd_{p^b_{BK}}f_b(p^b_{BK}[n+1],x^b_{BI(B)K}[\nb])\-\gd_{p^b_{BK}}f_b(p^b_{BK}[n],x^b_{BI(B)K}[n])\right]$}
\State{$y^b_{BI(B)K}[\nb]=\sum_{b'\in Nb(b)}W_{bb'}y^{b'}_{BI(B)K}[n]+\left[\gd_{x^b_{BI(B)K}}f_b(p^b_{BK}[\nb],x^b_{BI(B)K}[\nb])-\gd_{x^b_{BI(B)K}}f_b(p^b_{BK}[n],x^b_{BI(B)K}[n])\right]$}
\State{$\pit^b_{BK}[\nb]=|B|\1_{BK}\circ r^b_{BK}[\nb]-\gd_{p^b_{BK}}f_b(p^b_{BK}[\nb],x_{bI(b)K}[\nb])$}
\State{$\taut^b_{BI(B)K}[\nb]=|B|\1_{BI(B)K}\circ y^b_{BI(B)K}[\nb]-\gd_{x^b_{BI(B)K}}f_b(p^b_{BK}[\nb],x^b_{BI(B)K}[\nb])$}
\EndWhile
\Ensure{$p^b_{bK}[n]$ and $x^b_{bI(b)K}[n]$}
\end{algorithmic}
}
\end{algorithm}

\section{A Stochastic Approximation Viewpoint}\label{sec:sa}
In this appendix, we review \emph{NEXT} from the viewpoint of stochastic approximation. We provide an alternative proof of Theorem 4 in \cite{Scutari_NEXT_2016} using results from stochastic approximation in \cref{sec:sa-1}. As we will see, \emph{NEXT} can be seen as a two time-scale process, with the faster $\yB$ tracking the total gradient, the slower $\xB$ tracking the fixed point iteration of $\xBh(\bullet)$, and a repeated projection onto the \emph{consensus plane}. From this viewpoint and the fact that the local optimization in \cref{8-2} (in \emph{NEXT} version, see Equation (8), \cite{Scutari_NEXT_2016}) can be solved by the projected gradient descent method described in \cref{sec:sa-2}, we can interleave each ``descent" as another time-scale of the algorithm. We relate the result with another distributed non-convex optimization method proposed in \cite{Bianchi_DistGrad_2013}.
\subsection{Alternative Proof of NEXT}\label{sec:sa-1}
Substituting the definitions of $\zB$ and $\piBt$ into \emph{Inexact NEXT} (Algorithm 2, \cite{Scutari_NEXT_2016}), we can rewrite each iteration of the algorithm in two steps:
\begin{comment}
\begin{align}
%&\yB_i[n]=\sum_{j=1}^Iw_{ij}[n]\yB_j[n-1]+\left[\gd f_i(\xB_i[n])-\gd f_i(\xB_i[n-1])\right],\label{sa1-1}\\
%&\xB_i[n+1]=\sum_{j=1}^Iw_{ij}[n]\left[\xB_j[n]+\al[n]\left(\xBt_j(\xB_j[n],\yB_j[n])-\xB_j[n]+\eB_j[n]\right)\right],\label{sa1-2}
\end{align}
\end{comment}
\begin{align}
&\yB_i[n]=\sum_{j=1}^Iw_{ij}\yB_j[n-1]+\left[\gd f_i(\xB_i[n]-\gd f_i(\xB_i[n-1]))\right],\label{sa1-1}\\
&\xB_i[n+1]=\sum_{j=1}^Iw_{ij}\left[\xB_j[n]+\al[n]\left(\xBt_j(\xB_j[n],\yB_j[n])-\xB_j[n]+\eB_j[n]\right)\right],\label{sa1-2}
\end{align}
where $\|\eB_i[n]\|\leq\ep_i[n]\es\forall\es i$, and $\xBt_i(\xB_i[n],\yB_i[n])$ is given by (8) in \cite{Scutari_NEXT_2016} with $\piBt_i$ substituted by $\yB_i$ using (S.3) (c) of Algorithm 1 in \cite{Scutari_NEXT_2016}. By letting $\uB_i[n]=\yB_i[n]-\gd f_i(\xB_i[n])$, \cref{sa1-1} can be rewritten as
\begin{comment}
\bal\label{sa2}
\uB_i[n]&=\sum_{j=1}^Iw_{ij}[n]\left[\uB_j[n-1]+\gd f_j(\xB_j[n-1])\right]-\gd f_i(\xB_i[n-1])\\
&=\uB_i[n-1]+\be[n]\left\{\sum_{j=1}^Iw_{ij}[n]\left[\uB_j[n-1]+\gd f_j(\xB_j[n-1])\right]\right.\\
&\left.-\left[\uB_i[n-1]+\gd f_i(\xB_i[n-1])\right]\right\},
\eal
\end{comment}
\bal\label{sa2}
\uB_i[n]&=\sum_{j=1}^Iw_{ij}\left[\uB_j[n-1]+\gd f_j(\xB_j[n-1])\right]-\gd f_i(\xB_i[n-1])\\
&=\uB_i[n-1]+\be[n]\left\{\sum_{j=1}^Iw_{ij}\left[\uB_j[n-1]+\gd f_j(\xB_j[n-1])\right]\right.\\
&\quad\left.-\left[\uB_i[n-1]+\gd f_i(\xB_i[n-1])\right]\right\},
\eal
where $\be[n]=1$. It is evident that $\al[n]=o(\be[n])$. As a result, \cref{sa1-2} and \cref{sa2} together form a two time-scale stochastic approximation algorithm \cite{Borkar_TwoTime_1997}, where $\uB_i$ or $\yB_i$ is on a faster, natural time-scale with constant step sizes, and $\xB_i$ goes on a slower, algorithmic time-scale with shrinking step sizes.

To analyze this process, we first begin with the fact that the fast variable $\uB$ or $\yB$ views the slow variable $\xB$ as quasi-static, i.e. we can see $\xB$ as constant in \cref{sa2}. Denote $\uB$ as the ensemble of $\uB_i$'s, i.e. $\uB=\bms\uB_1^T&\cdots&\uB_I^T\bme^T$, $\gd f$ as the ensemble of $\gd f_i$'s, and also $\xB$, $\yB$, etc. Then the iterate of $\uB$ will asymptotically track the following ordinary differential equation (ODE)
%\begin{equation}\label{sa3}
%\dot{\uB}(t)=[W(t)\otimes\IB_d-\IB_{dI}][\uB(t)-\gd f(\xB)],\es\uB(0)=0,
%\end{equation}
\begin{equation}\label{sa3}
\dot{\uB}(t)=[W\otimes\IB_d-\IB_{dI}][\uB(t)-\gd f(\xB)],\es\uB(0)=0,
\end{equation}
%where $W(t)$ is the interpolation of $W[n]$, 
where $\IB$ is the identity matrix, $d$ denotes the dimension of $\xB_i$'s, and $\otimes$ means the Kronecker product.
%For simplicity we will treat $W(t)$ as constant; as long as it remains doubly stochastic, irreducible, and nondegeneracy (all edges having corresponding entries in $[\vartheta,1]$), the fact that it is time-varying does not change anything.
Then \cref{sa3} would become
\begin{equation}\label{sa4}
\dot{\yB}(t)=(W\otimes\IB_d-\IB_{dI})\yB(t),\es\yB(0)=\gd f(\xB).
\end{equation}
\begin{lemma}\label{l1}
We have $\limti\yB(t)=\overline{\gd f}(\xB)\otimes\1_I$.
\end{lemma}
\begin{proof}
We have
\bal\label{sa5}
\limti\yB(t)&=\limti e^{t(W\otimes\IB_d-\IB_{dI})}\yB(0)\\
&=\limti e^{tW\otimes\IB_d}e^{-t\IB_{dI}}\gd f(\xB)\\
&=\limti (e^{tW}\otimes\IB_d)\cdot\frac{1}{e^t}\IB_{dI}\gd f(\xB)\\
&=\limti \left(\frac{e^t}{I}\1_I\1_I^T\right)\otimes\IB_d\cdot\frac{1}{e^t}\gd f(\xB)\\
&=\left[\frac{1}{I}(\1_I\1_I^T)\otimes\IB_d\right]\cdot\gd f(\xB)=\overline{\gd f}(\xB)\otimes\1_I,
\eal
where $\1$ is the all one vector and $\overline{\gd f}=\frac{1}{I}\sum_{i=1}^I\gd f_i$. The second equality follows from the two matrices being multiplication commutative, third from $e^{\mathbf{A}\otimes\IB+\IB\otimes\mathbf{B}}=e^{\mathbf{A}}\otimes e^{\mathbf{B}}$, and fourth from the fact that $\limti W^t=\frac{1}{I}(\1_I\1_I^T)$.
\end{proof}

\noindent We see that $\yB(t)$ indeed goes to the unique global asymptotically stable equilibrium, where every component of $\yB$, i.e. $\yB_i$'s, equals to the average of the gradients as desired.

Next, from the perspective of the slow variable $\xB$, the fast variable $\yB$ already reaches its equilibrium $\yBb(\xB)$. That is to say, in \cref{sa1-2} $\xBt_j(\xB_j[n],\yB_j[n])$ can be seen as $\xBt_j(\xB_j[n],\overline{\gd f}(\xB_j[n]))$, which is exactly $\xBh_j(\xB_j[n])$, making \cref{sa1-2} become
\begin{equation}\label{sa6}
\xB_i[n+1]=\sum_{j=1}^Iw_{ij}\left[\xB_j[n]+\al[n]\left(\xBh_j(\xB_j[n])-\xB_j[n]+\eB_j[n]\right)\right].
\end{equation}
As stated in \cite{Borkar_Gossip_2016}, this recursive relation is again a two time-scale stochastic approximation in disguise, with fast averaging and slow learning processes. In fact, the averaging process is also on natural time-scale as $\yB$. From \cite{Borkar_Gossip_2016} we know that the iterates of $\xB$ will reach consensus $\xB[n]=\bms\xB_c[n]^T&\cdots&\xB_c[n]^T\bme^T$, while each of its component $\xB_c[n]\in\R^d$ tracks the ODE
\bal\label{sa7}
\dot{\xB}_c(t)&=\frac{1}{I}(\1_I\otimes\IB_d)^T\bms\cdots&\xBh_i(\xB_c(t))^T-\xB_c(t)^T&\cdots\bme^T\\
&=\frac{1}{I}\sum_{i=1}^I\xBh_i(\xB_c(t))-\xB_c(t),
\eal
as $\1_I/I$ is the unique stationary distribution resulted from $W$.

Note that with $\xBh_i$'s being Lipschitz continuous (\cite{Scutari_NEXT_2016}, Prop. 5a), this ODE is well-posed. We assume the differentiability of $G$ for now to avoid dealing with trickier non-differentiable Lyapunov function here. We consider the whole objective itself as the Lyapunov function $V(\xB_c)=U(\xB_c)=F(\xB_c)+G(\xB_c)$. Then
\bal\label{sa8}
\dot{V}(\xB_c(t))&=[\gd U(\xB_c(t))]^T\cdot\left[\frac{1}{I}\sum_{i=1}^I\xBh_i(\xB_c(t))-\xB_c(t)\right]\\
&\leq-c_\tau\frac{1}{I}\sum_{i=1}^I\|\xBh_i(\xB_c(t))-\xB_c(t)\|^2\leq0,
\eal
for some positive constant $c_\tau$. The first inequality is established similarly as \cite{Scutari_NEXT_2016}, Prop. 5b. By Lasalle's invariance principle, the iterates converge to the set of equilibria $\{\xB_c:\frac{1}{I}\sum_{i=1}^I\xBh_i(\xB_c)=\xB_c\}$ (\cite{Borkar_SABook_2008}, p. 57 and p. 118), which is the set of stationary solutions of the original optimization problem (\cite{Scutari_NEXT_2016}, Prop. 2).

We note that the conditions for applying \cite{Borkar_TwoTime_1997} and \cite{Borkar_Gossip_2016} are either established in \cite{Scutari_NEXT_2016} or implied by the assumptions of Theorem 4 in \cite{Scutari_NEXT_2016}. Specifically, the boundedness of $\xB$ follows from the recursive relation \cref{sa1-2} and Proposition 9 (a) in \cite{Scutari_NEXT_2016}, $\sum_n\al[n]=\sum_n\be[n]=\infty$, $\sum_n\al[n]^2<\infty$, and $\sup\sum_n\al[n]\eB_i[n]<\infty$ are just assumed in Theorem 4 in \cite{Scutari_NEXT_2016}. Note that we do not need $\sum_n\be[n]^2<\infty$ as there is no noise in the recursion of $\yB$. Also, we have deterministic convergence rather than almost sure convergence, since instead of being martingale differences, our noise term $\eB_i[n]$ is actually deterministically bounded.

\subsection{A Remark on Using One-Step Gradient Descent}\label{sec:sa-2}
Solving the local optimization in Equation (8) in \cite{Scutari_NEXT_2016} may be costly. Instead, from the stochastic approximation viewpoint, we can solve it iteratively as well using a time-scale faster than $\al[n]$. Once again assume that $G$ is continuously differentiable to avoid working with subgradients. The optimization problem in Equation (8) in \cite{Scutari_NEXT_2016} belongs to the class of constrained convex optimization problems, and can be solved by projected gradient descent:
\beq\label{sa8-1}
\xB'_i[n'+1]=\Pc_\Kc^e(\xB'_i[n']-\ga[n]\gd_{\xB_i}\tilde{U}(\xB'_i[n'];\xB_i[n],\piBt_i[n]))
\eeq
where $\ga[n]$ is some time-scale faster than $\al[n]$ and $P_\Kc$ is the Euclidean projection onto the set $\Kc$
\[
P_\Kc^e(\zB)=\underset{\xB\in\Kc}{\arg\min}\frac{1}{2}\|\xB-\zB\|^2.
\]
Consider using the natural time-scale $\ga[n]=1$. The idea of stochastic approximation is essentially blending \cref{sa8-1} into the original algorithm \cite{Borkar_Proj_2017}. Namely, at $\xB_i[n]$ instead of running \cref{sa8-1} infinitely many times to exactly solve (8) in \cite{Scutari_NEXT_2016}, we only run one step gradient descent of \cref{sa8-1} from $\xB_i[n]$:
\bal\label{sa8-2}
\xBt'_i[n]&=\Pc_\Kc^e(\xB_i[n]-\gd_{\xB_i}\tilde{U}(\xB_i[n];\xB_i[n],\piBt_i[n]))\\
&=P_\Kc^e\left[\xB_i[n]-(\gd\ft_i(\xB_i[n];\xB_i[n])+\piBt_i[n]+\gd G(\xB_i[n]))\right]\\
&=P_\Kc^e\left[\xB_i[n]-(\gd f_i(\xB_i[n])+\piBt_i[n]+\gd G(\xB_i[n]))\right],
\eal
and then use $\xBt'_i[n]$ instead of $\xBt_i[n]$ in (S.2) (a) of Algorithm 1 in \cite{Scutari_NEXT_2016}.
As $\xB_i[n]$ converges, so does the coupling process of \cref{sa8-1}.

From the previous subsection we know that from the perspective of the slow variable $\xB$, $\yB_j$ can be seen as $\overline{\gd f}(\xB_j)$, and thus $\piBt_i[n]$ as $\sum_{j\neq i}\gd f_j(\xB_i[n])$. Combining this, \cref{sa8-2}, and Algorithm 2 in \cite{Scutari_NEXT_2016} yield
\beq\label{sa8-3}
\xB_i[n+1]=\Pc_\Kc^e\left\{\sum_jw_{ij}\left[\xB_j[n]+\al[n](-\gd U(\xB_j[n])+\eB_j[n])\right]\right\}.
\eeq
Note that \cref{sa8-3} contains two projections. It is simple gradient descent followed by a projection to the consensus plane $\Cc:=\{\XB=[\xB_1^T\es\cdots\es\xB_I^T]^T\in\R^{dI}:\xB_1=\cdots=\xB_I\}$ similar to Example 3 in \cite{Borkar_Gossip_2016}, and then a further projection onto $\Kc$. Similar to \cref{sec:sa-1}, the iterates of $\xB$ will reach consensus $\xB[n]=\bms\xB_c[n]^T&\cdots&\xB_c[n]^T\bme^T$, while each of its component $\xB_c[n]\in\R^d$ tracks the ODE
\beq\label{sa8-4}
\dot{\xB}_c(t)=\Pc_\Kc^e[-\gd U(\xB_c(t))],
\eeq
called the projected gradient flow. From Proposition 5 of \cite{Bianchi_DistGrad_2013} (also \cite{Borkar_Proj_2017}, p. 10), $U$ works as a Lyapunov function for the set of its stationary solutions, then by Theorem 2 of \cite{Bianchi_DistGrad_2013} (also \cite{Borkar_Proj_2017}, Proposition 9 and Remark 10) the iterates converge almost surely to the set of stationary solutions\footnote{Again, ``almost surely" is unnecessary here, since our noise is deterministically bounded.}\footnote{If in general we consider $\dot{\xB}_c(t)=\Pc_\Kc^e[h(\xB_c(t))]$, a Lyapunov function may not exist, and we are only guaranteed convergence to a ``nonempty compact connected internally chain transitive invariant set" in $\Kc\dU\Cc$ \cite{Borkar_Gossip_2016}. A detailed tutorial covering the difference of these sets can be found in \cite{Benaim_SA_2005}.}\footnote{Coerciveness of $U$ is required to apply the LaSalle invariance principle. Without it, $U$ has to be analytic to make the set of local minima and the set of Lyapunov stable points of the gradient flow equal to each other \cite{Absil_GradFlow_2006}.}.

In \cref{sa8-3}, the matrix $\WB$ only needs to be column stochastic instead of doubly stochastic. There exists some distribution $\{\zt_i\}_{i=1}^I$ to which $\WB$ converges, and the ODE will be minimizing $\sum_{i=1}^I\zt_i U=U$.
%We will describe this in more detail in \cref{sec:cons}.
It immediately follows that with doubly stochastic $\WB$ (which induces $\{\frac{1}{I}\}_{i=1}^I$), we can reduce $\gd U(\xB_j[n])$ in \cref{sa8-3} to $\gd f_j(\xB_j[n])+\gd G(\xB_j[n])$ and still minimizing $\sum_{i=1}^I\frac{1}{I}(f_i+G)=U$, implying that we can simply let the gradient of each node spread through the gossip and do not have to track $\yB$ and $\piBt$ anymore. This is exactly what is done in \cite{Bianchi_DistGrad_2013} (with $G=0$).

The two papers \cite{Scutari_NEXT_2016} and \cite{Bianchi_DistGrad_2013} have their own strengths. The method in \cite{Bianchi_DistGrad_2013} can solve the whole problem using less computation power due to the high cost of local optimization, while \textit{NEXT} not only allows non-differentiable $G$ but also achieves the optimal within fewer iterations. This is useful in delay sensitive applications with abundant computing resource.

\section{Technical Lemmas}\label{app:e}
We put some technical lemmas used in the proof in this appendix. Some of them are from \cite{Scutari_NEXT_2016}.

\begin{fact}\label{f24}
%\hfill
For all $m$ we have the following.
\begin{enumerate}[label=(\alph*)]
\item
$J^m_\perp\WBh^m[n]=J^m_\perp\WBh^m[n]J^m_\perp=\WBh^m[n]-\frac{1_m}{I}\1_\Nm\1_\Nm^T\otimes I_{d_m}$, where we also have
\bals
J^m_\perp\WBh^m[n]
&=\WBh^m[n]-\left(\WBh^m[n]\cdot\frac{1}{I_m}\1_\Nm\1_\Nm^T\right)\otimes\IB_{d_m}\\
&=\left(\IB_{d_mI_m}-\1_\Nm\1_\Nm^T\otimes\IB_{d_m}\right)\WBh^m[n]\left(\IB_{d_mI_m}-\1_\Nm\1_\Nm^T\otimes\IB_{d_m}\right)\\
&=J^m_\perp\WBh^m[n]J^m_\perp.
\eals
We use the equality $(A\otimes B)\cdot(C\otimes D)=(A\cdot C)\otimes(B\cdot D)$ in showing the above equation.
\item
$J^m_\perp\WBh^m[n]J^m_\perp\WBh^m[n-1]\cdots J^m_\perp\WBh^m[l]=J^m_\perp\PBh^m[n,l]=\left(\PB^m[n,l]-\frac{1}{I_m}\1_\Nm\1_\Nm^T\right)\otimes\IB_{d_m}$.
\item
$\bar{\mathbf{q}}^m\triangleq\frac{1}{I_m}\sum_{i\in\Nm}q_i=\frac{\1_\Nm^T\otimes\IB_{d_m}}{I_m}\mathbf{q}$ where $\mathbf{q}=[q_1^T\es\cdots\es q_I^T]^T$ and $q_1,\dots,q_I$ are all arbitrary in $\R^{d_m}$.
\item
$\xB^m[n]=\WBh^m[n-1]\xB^m[n-1]+\al[n-1]\WBh^m[n-1]\Dl\xB^{m,inx}[n-1]$ where $\Dl\xB^{m,inx}[n]=\left(\I\{i\in\Nm\}(\xB^{m,inx}_i[n]-\xB^m_i[n])\right)_{i\in\Nc}$. This simply follows from Lines 7 and 9 of \cref{a6}.
\item
$\xBb^m[n]=\xBb^m[n-1]+\frac{\al[n-1]}{I_m}\left(\1_\Nm^T\otimes\IB_{d_m}\right)\Dl\xB^{m,inx}[n-1]$. This follows from applying (c) to (d).
\end{enumerate}
\end{fact}

\begin{lemma}\label{l21}
Let $0<\lm<1$, and let $\{\be[n]\}$ and $\{\nu[n]\}$ be two positive scalar sequences. Then
\begin{enumerate}[label=(\alph*)]
\item
If $\lim_{n\ra\infty}\be[n]=0$, then $\lim_{n\ra\infty}\sum_{l=1}^n\lm^{n-l}\be[l]=0$.
\item
If further we have $\sum_{n=1}^\infty\be^2[n]<\infty$ and $\sum_{n=1}^\infty\nu^2[n]<\infty$, then $\lim_{n\ra\infty}\sum_{k=1}^n\sum_{l=1}^k\lm^{k-l}\be^2[l]<\infty$ and $\lim_{n\ra\infty}\sum_{k=1}^n\sum_{l=1}^k\lm^{k-l}\be[k]\nu[l]<\infty$.
\end{enumerate}
\end{lemma}

\begin{lemma}\label{l22}
Let $\{Y[n]\}$, $\{X[n]\}$, and $\{Z[n]\}$ be three sequences of numbers such that $X[n]\geq0$ for all $n$. If $Y[n+1]\leq Y[n]-X[n]+Z[n]$ for all $n$ and $\sum_{n=1}^\infty Z[n]<\infty$, then either $Y[n]\ra-\infty$ or $\{Y[n]\}$ converges to a finite value and $\sum_{n=1}^\infty X[n]<\infty$.
\end{lemma}

\begin{lemma}\label{l24}
Let $0<\lm<1$, and let $\{\be[n]\}$ and $\{\nu[n]\}$ be two positive scalar sequences such that $\be[n]\ra0$, $\nu[n]\ra\infty$, and $\be[n]\nu[n]\ra0$. If further there exist $1>\tilde{\lm}>\lm$ and $N$ such that $\frac{\be[n]}{\be[l]}\geq\tilde{\lm}^{n-l}$ for all $n\geq l\geq N$, then $\lim_{n\ra\infty}\nu[n]\sum_{l=1}^n\lm^{n-l}\be[l]=0$. Moreover, if $\be[n]\nu[n]$ is summable, then so is $\nu[n]\sum_{l=1}^n\lm^{n-l}\be[l]$.
\end{lemma}
\begin{proof}
The proof of the first part of the claim is straightforward.
\bal\label{28-1}
\nu[n]\sum_{l=1}^n\lm^{n-l}\be[l]
&=\nu[n]\sum_{l=1}^{N-1}\lm^{n-l}\be[l]+\be[n]\nu[n]\sum_{l=N}^n\lm^{n-l}\frac{\be[l]}{\be[n]}\\
&\leq\nu[n]\sum_{l=1}^{N-1}\lm^{n-l}\be[l]+\be[n]\nu[n]\sum_{l=N}^n\frac{\lm^{n-l}}{\tilde{\lm}^{n-l}}\\
&\leq\nu[n]\sum_{l=1}^{N-1}\lm^{n-l}\be[l]+\be[n]\nu[n]\cdot\frac{1}{1-\lm/\tilde{\lm}}.
\eal
The second term goes to zero by the condition. Note that the meaning of $\frac{\be[n]}{\be[l]}\geq\tilde{\lm}^{n-l}$ basically says $\be[n]$ cannot decay to zero faster than at an exponential rate. Thus, $\be[n]\nu[n]\ra0$ would imply that $\lm^n\nu[n]\ra0$ as well. For the second part of the claim, just sum \cref{28-1} over $n$.
\end{proof}

\end{document}